\documentclass[11pt,a4paper,reqno]{amsart}
\usepackage[english]{babel}
\usepackage[T1]{fontenc}
\usepackage{verbatim}
\usepackage{palatino}
\usepackage{amsmath}
\usepackage{mathabx}
\usepackage{amssymb}
\usepackage{amsthm}
\usepackage{amsfonts}
\usepackage{graphicx}
\usepackage{esint}
\usepackage{color}
\usepackage{mathtools}
\usepackage{overpic}
\DeclarePairedDelimiter\ceil{\lceil}{\rceil}

\usepackage[colorlinks = true, citecolor = black]{hyperref}
\title{On the dimension of $s$-Nikod\'ym sets}

\author{Damian D{\k a}browski}
\address{Department of Mathematics and Statistics\\ University of Jyv\"askyl\"a,
P.O. Box 35 (MaD)\\
FI-40014 University of Jyv\"askyl\"a\\
Finland} 

\email{\href{mailto:damian.m.dabrowski@jyu.fi}{damian.m.dabrowski@jyu.fi}}

\author{Max Goering}
\address{Department of Mathematics and Statistics\\ University of Jyv\"askyl\"a,
P.O. Box 35 (MaD)\\
FI-40014 University of Jyv\"askyl\"a\\
Finland} 
\email{\href{mailto:max.l.goering@jyu.fi}{max.l.goering@jyu.fi}}

\author{Tuomas Orponen}
\address{Department of Mathematics and Statistics\\ University of Jyv\"askyl\"a,
P.O. Box 35 (MaD)\\
FI-40014 University of Jyv\"askyl\"a\\
Finland} 
\email{\vspace{-1cm}\href{mailto:tuomas.t.orponen@jyu.fi}{tuomas.t.orponen@jyu.fi}}

\date{\today}
\subjclass[2010]{28A78 (primary) 28A80 (secondary)}
\keywords{Nikod\'ym sets, fractals, slicing, Hausdorff dimension}
\thanks{D.D is supported by the Research Council of Finland postdoctoral grant \emph{Quantitative rectifiability and harmonic measure beyond the Ahlfors-David-regular setting}, grant
		No. 347123. M.G. and T.O. are supported by the European Research Council (ERC) under the European Union’s Horizon Europe research and innovation programme (grant agreement No 101087499). T.O. is supported by the Research Council of Finland via the project \emph{Approximate incidence geometry}, grant no. 355453.}

\newcommand{\R}{\mathbb{R}}

\newcommand{\N}{\mathbb{N}}

\newcommand{\C}{\mathbb{C}}
\newcommand{\Z}{\mathbb{Z}}

\newcommand{\calT}{\mathcal{T}}

\newcommand{\calD}{\mathcal{D}}

\newcommand{\calF}{\mathcal{F}}
\newcommand{\calS}{\mathcal{S}}

\newcommand{\spt}{\operatorname{spt}}
\newcommand{\Hd}{\dim_{\mathrm{H}}}

\newcommand{\calP}{\mathcal{P}}
\newcommand{\cH}{\mathcal{H}}

\newcommand{\calA}{\mathcal{A}}
\newcommand{\1}{\mathbf{1}}

\newcommand{\diam}{\operatorname{diam}}

\newcommand{\dist}{\operatorname{dist}}

\newcommand{\bF}{\mathbf{F}}
\newcommand{\ve}{\varepsilon}

\def\Barint_#1{\mathchoice
          {\mathop{\vrule width 6pt height 3 pt depth -2.5pt
                  \kern -8pt \intop}\nolimits_{#1}}%
          {\mathop{\vrule width 5pt height 3 pt depth -2.6pt
                  \kern -6pt \intop}\nolimits_{#1}}%
          {\mathop{\vrule width 5pt height 3 pt depth -2.6pt
                  \kern -6pt \intop}\nolimits_{#1}}%
          {\mathop{\vrule width 5pt height 3 pt depth -2.6pt
                  \kern -6pt \intop}\nolimits_{#1}}}

\numberwithin{equation}{section}

\theoremstyle{plain}
\newtheorem{thm}[equation]{Theorem}
\newtheorem*{"thm"}{"Theorem"}

\newtheorem{lemma}[equation]{Lemma}

\newtheorem{proposition}[equation]{Proposition}
\newtheorem{question}{Question}

\theoremstyle{definition}

\newtheorem{definition}[equation]{Definition}

\newtheorem{notation}[equation]{Notation}

\theoremstyle{remark}
\newtheorem{remark}[equation]{Remark}

\newcommand{\nref}[1]{(\hyperref[#1]{#1})}

\DeclareMathSymbol{\intop}  {\mathop}{mathx}{"B3}

\begin{document}

\begin{abstract} Let $s \in [0,1]$. We show that a Borel set $N \subset \R^{2}$ whose every point is linearly accessible by an $s$-dimensional family of lines has Hausdorff dimension at most $2 - s$. \end{abstract}

\maketitle

\tableofcontents

\section{Introduction}

A set $K \subset \R^{2}$ is called \emph{linearly accessible} if for every $x \in K$ there exists a line $\ell_{x} \subset \R^{2}$ containing $x$, but no other points of $K$ -- in other words $K \cap \ell_{x} = \{x\}$. In 1927, answering a question of Banach \cite{ABSU24}, Nikod\'ym \cite{Ni} constructed a linearly accessible Borel set $N \subset [0,1]^{2}$ of Lebesgue measure unity. Nowadays linearly accessible sets of positive measure are often called \emph{Nikod\'ym sets}. A simpler construction was found by Davies \cite{MR45795} in the 50s, and a different projection theoretic viewpoint (along with generalisations to higher dimensions) was provided by Falconer in his famous "digital sundial paper" \cite{MR842156} in the 80s. A Baire category based construction is described in the paper \cite{MR3771142} by Chang, Cs\"ornyei, H\'era, and Keleti. See also the paper \cite{MR639466} of Casas and de Guzm\'{a}n, where it is shown that the accessibility lines $\ell_{x}$ cannot be chosen to depend on $x$ in a Lipschitz way (the result is also recorded in \cite[Theorem 9.6.4]{MR596037}).
 
What if we require every point of $N$ to be accessible by many lines? Davies \cite{MR45795} already in the 50s showed that there exist Borel sets $N \subset [0,1]^{2}$ of full measure such that every point $x \in N$ is accessible by uncountably many lines, and indeed uncountably many lines in every non-trivial arc of directions. Davies called such sets \emph{$c$-densely accessible}.

Can the set of accessibility lines have Hausdorff dimension $s \in (0,1]$? Our main result answers this negatively for Borel sets $N \subset \R^{2}$ of full dimension:

\begin{thm}\label{t:Nikodym} Let $s \in [0,1]$, and let $N \subset \R^{2}$ be an \emph{$s$-Nikod\'ym set}: a Borel set whose every point is accessible by an $s$-dimensional family of lines. Then $\Hd N \leq 2 - s$. \end{thm}

\begin{remark} There is work in progress by other authors on the question of the sharpness of Theorem \ref{t:Nikodym}. We can hopefully report on that in a later version of this paper. \end{remark}

We derive Theorem \ref{t:Nikodym} from the following, slightly stronger, statement:

\begin{thm}\label{t:main} Let $t \in [1,2]$, and let $K \subset \R^{2}$ be compact. For $x \in K$, let 
\begin{displaymath} E_{x}(K) := \{e \in S^{1} : K \cap \ell_{x,e} = \{x\}\} = \{e \in S^{1} : |K \cap \ell_{x,e}| = 1\}. \end{displaymath}
Then, $\Hd E_{x}(K) \leq 2 - t$ for $\mathcal{H}^{t}$ almost all $x \in K$.  \end{thm} 

\begin{proof}[Proof of Theorem \ref{t:Nikodym} assuming Theorem \ref{t:main}] Assume to the contrary that $\Hd N > 2 - s$, and start by finding a compact subset $K \subset N$ with $\Hd K > 2 - s$. Now $\Hd E_{x}(K) \geq \Hd E_{x}(N) \geq s$ for all $x \in K$. But according to Theorem \ref{t:main} applied with any $t \in (2 - s,\Hd K)$, we have $\Hd E_{x}(K) \leq 2 - t < s$ for $\mathcal{H}^{t}$ almost all $x \in K$. Since $\Hd K > t$, in particular $\Hd E_{x}(K) < s$ for some $x \in K$. This contradiction completes the proof of Theorem \ref{t:Nikodym}. \end{proof}

\subsection{Related work} Prior to this paper it was only known that if $N \subset \R^{2}$ is Borel, and the accessibility lines have positive linear measure for all $x \in N$, then $\Hd N \leq 1$. This is a consequence of the following radial slicing theorem of Marstrand \cite{Mar} from 1954:
\begin{thm}\label{t:Marstrand} Let $t \in (1,2]$, and let $B \subset \R^{2}$ be a Borel set with $0 < \mathcal{H}^{t}(B) < \infty$. Then, for $\mathcal{H}^{t}$ almost all $x \in B$, it holds 
\begin{displaymath} \Hd (B \cap \ell_{x,e}) \geq t - 1 \end{displaymath}
for $\mathcal{H}^{1}$ almost all $e \in S^{1}$. Here $\ell_{x,e} = x + \mathrm{span}(e)$. \end{thm} 

Why does Theorem \ref{t:Marstrand} imply the claim above it? If $\Hd N > 1$, choose $t > 1$ and a Borel subset $B \subset N$ with $0 < \mathcal{H}^{t}(B) < \infty$. Then $\Hd (N \cap \ell_{x,e}) \geq \Hd (B \cap \ell_{x,e}) > 0$ for $(\mathcal{H}^{t} \times \mathcal{H}^{1})$ almost all $(x,e) \in B \times S^{1}$. This contradicts the assumption that every $x \in N$ has $\mathcal{H}^{1}$ positively many accessibility lines.

Another relevant previous result is the following exceptional set estimate for the (linear) Marstrand slicing theorem obtained in \cite{MR3145914}:
\begin{thm}\label{sharpMarstrand} Let $t \in (1,2]$, and let $B \subset \R^{2}$ be a Borel set with $0 < \mathcal{H}^{t}(B) < \infty$. Then, the following holds for all $e \in S^{1} \, \setminus \, E$, where $\Hd E \leq 2 - t$. There are $\mathcal{H}^{1}$ positively many lines $\ell \subset \R^{2}$ parallel to $e$ with $\Hd (B \cap \ell) \geq t - 1$. \end{thm} 

Theorem \ref{sharpMarstrand} contains the exponent "$2 - t$", so one might be led to think that it implies Theorem \ref{t:Nikodym}. In fact, it only implies the following special case, where all the accessibility lines have the same directions. Let $N \subset \R^{2}$ be Borel, and let $E \subset S^{1}$ with $\Hd E = s \in [0,1]$. If all the points in $N$ are accessible by a line parallel to each of the directions in $S$, then $\Hd N \leq 2 - s$. This statement can be derived from Theorem \ref{sharpMarstrand} exactly like Theorem \ref{t:Nikodym} is derived from Theorem \ref{t:main}. 

The technique in \cite{MR3145914} for proving Theorem \ref{sharpMarstrand} is not powerful enough to establish Theorem \ref{t:main}, where the key difficulty is that the direction sets $E_{x}(K)$ may vary as a function of $x$. On the other hand, in its restricted context, Theorem \ref{sharpMarstrand} gives a stronger conclusion: outside a small exceptional set of directions, it guarantees the existence of lines $\ell \subset \R^{2}$ with $\Hd (B \cap \ell) \geq t - 1 > 0$. A positive answer to the following question would supersede both Theorems \ref{t:main} and \ref{sharpMarstrand}:
\begin{question}\label{q1} Let $t \in (1,2]$, and let $K \subset \R^{2}$ be compact. For $x \in K$, let
\begin{displaymath} E_{x}^{t - 1}(K) := \{e \in S^{1} : \Hd (K \cap \ell_{x,e}) < t - 1\}. \end{displaymath}
Is it true that $\Hd E_{x}^{t - 1}(K) \leq 2 - t$ for $\mathcal{H}^{t}$ almost all $x \in K$? 
 \end{question}

\begin{remark}\label{rem2} The following result was claimed in \cite[Theorem 6.9]{MR3617376}: if $t \in (1,2]$, and $K \subset \R^{2}$ is compact with $0 < \mathcal{H}^{t}(K) < \infty$, then there exists a set $E \subset S^{1}$ such that $\Hd E \leq 2 - t$, and for $\mathcal{H}^{t}$ almost every $x \in K$ it holds
\begin{equation}\label{form70} \Hd (K \cap \ell_{x,e}) \geq t - 1, \qquad e \in S^{1} \, \setminus \, E. \end{equation}
If true, this result would even supersede Question \ref{q1}. 

However, \cite[Theorem 6.9]{MR3617376} is not correct. Counter examples are provided by "radial graphs", see Appendix \ref{appB}. In fact, even the weaker version of \cite[Theorem 6.9]{MR3617376} is false, where \eqref{form70} is relaxed to $|K \cap \ell_{x,e}| \geq 2$. We have confirmed this with the author of \cite{MR3617376}. 

The main problem in \cite[Theorem 6.9]{MR3617376} is the claim that the set $E \subset S^{1}$ can be chosen independently of the point $x$ (in some $\mathcal{H}^{t}$ full measure subset of $K$). While it seems possible that $\Hd E_{x}^{t - 1}(K) \leq 2 - t$ for $\mathcal{H}^{t}$ almost all $x \in K$, \cite[Theorem 6.9]{MR3617376} would imply that the \textbf{union} of the sets $E_{x}^{t - 1}(K)$ over a $\mathcal{H}^{t}$ full measure subset of points $x \in K$ has dimension $\leq 2 - t$. This is not true. Even the weaker version is false, where $E_{x}^{t - 1}(K)$ is replaced by $E_{x}(K)$. We show in Appendix \ref{appB} that there are no non-trivial estimates for the dimension of this union, at least for $t \in (1,2)$. \end{remark} 

\begin{remark} Theorems \ref{t:Marstrand} and \ref{t:main} are results on \emph{incidence lower bounds}: they state that a set $K \subset \R^{2}$ must have non-trivially large intersections with families of lines that \emph{a priori} only contain one point from $K$. The available literature on incidence lower bounds is quite thin compared to literature on \emph{incidence upper bounds}.  

A typical incidence upper bound problem asks to show that if $K \subset \R^{d}$ is a set with $\Hd K = t$, and $\mathcal{L}$ is a family of lines with $\Hd \mathcal{L} = s$ (and possibly other properties), then there exist lines $\ell \in \mathcal{L}$ with $\Hd (K \cap \ell) \leq \gamma(s,t)$ for a non-trivial exponent $\gamma(s,t) \in [0,1)$. The \emph{$(s,t)$-Furstenberg set} and \emph{Kakeya} problems can be viewed as special cases of this general incidence upper bound framework. These problems have recently witnessed great progress, see for example \cite{2023arXiv230511587F,fu2022incidence,2023arXiv230514544G,MR4521046,2023arXiv230110199O,2023arXiv230808819R,2022arXiv221009581W,2024arXiv240112337W,2023arXiv230705894Z,MR4550118}, and the proof of Theorem \ref{t:main} uses ideas and tools from \cite{fu2022incidence,2023arXiv230110199O}. For further discussion on outstanding incidence lower bound problems, see \cite{2023arXiv230518253C}, in particular \cite[Section 7]{2023arXiv230518253C}. 
\end{remark} 

\subsection{Proof outline} The first step in the proof of Theorem \ref{t:main} is to transform the "set theoretic" problem into a measure theoretic one. This is accomplished in Section \ref{s3}. The measure theoretic problem is stated in terms of \emph{configurations} and (associated) \emph{$X$-ray measures}. A \emph{configuration} is a pair $(\mu,\{\sigma_{x}\})$, where $\mu$ is a Radon measure on $[0,1]^{2}$, and each $\sigma_{x}$ is a measure supported on $S^{1}$ (or in fact $[0,1]$ for technical convenience). There is also a measurability hypothesis, see Definition \ref{def:configuration}.

The $X$-ray measure $\mathbf{X} = X[\mu,\{\sigma_{x}\}]$ associated to a configuration $(\mu,\{\sigma_{x}\})$ is defined by
\begin{displaymath} \int f \, d\mathbf{X} = \int \int Xf(\theta,\pi_{\theta}(x)) \, d\sigma_{x}(\theta) \, d\mu(x), \qquad f \in C_{c}(\R^{2}), \end{displaymath}
where $Xf \in C_{c}([0,1] \times \R)$ is the usual $X$-ray transform of $f$ (with the convention that lines in $\R^{2}$ are parametrised by $[0,1] \times \R$).

The idea, in the context of Theorem \ref{t:main}, is that $\mu$ is a $t$-Frostman measure on $K$ and each $\sigma_{x}$ is an $s$-Frostman measure on the set $E_{x}(K)$ with $s > 2 - t$. The measures $\sigma_{x}$ are provided by an initial counter assumption that $\Hd E_{x}(K) > 2 - t$ for most $x \in K$. Roughly speaking, the contradiction is eventually reached by exhibiting two sets $F,G \subset K$ with $\mathcal{H}_{\infty}^{t}(F) \sim 1 \sim \mathcal{H}_{\infty}^{t}(G)$ and $\dist(F,G) > 0$ such that
\begin{equation}\label{form72} \spt \mathbf{X}[\mu|_{F},\{\sigma_{x}\}] \cap G \neq \emptyset. \end{equation}
This is indeed a contradiction because 
\begin{displaymath} \Big( \bigcup_{x \in F} \{\ell_{x,\theta} : \theta \in E_{x}(K)\} \Big) \cap G = \emptyset \end{displaymath}
by definition of $E_{x}(K)$, and $\mathbf{X}[\mu|_{F},\{\sigma_{x}\}]$ is supported on $\bigcup_{x \in F} \{\ell_{x,\theta} : \theta \in E_{x}(K)\}$ (up to closures, a technicality we carefully track in the actual proof). The rigorous measure theoretic analogue of Theorem \ref{t:main} is stated in Theorem \ref{t:configurations}. 

How to prove \eqref{form72}? It turns out that if $\mu$ is $t$-Frostman and each $\sigma_{x}$ is (uniformly) $s$-Frostman with $s + t > 2$, then $\mathbf{X} \in \dot{H}^{\sigma/2}$ for all $\sigma < s$, see Proposition \ref{prop4}. This is used alongside Proposition \ref{prop3}, which contains the following general fact about Sobolev functions: If $h \in \dot{H}^{\sigma/2}$ is a non-negative function, and $\varphi_{\Delta}$ is an approximate identity at scale $\Delta$, then 
\begin{equation}\label{form73} \mathcal{H}^{t}_{\infty}(\{h \ast \varphi_{\Delta} \approx 1\} \, \setminus \, \spt h) = o_{\Delta \to 0}(1), \qquad \sigma + t > 2. \end{equation}
Informally, this says there are very few points which are $\Delta$-close to $\spt h$, yet not in $\spt h$. This principle is likely classical, and at least the $1^{st}$ and $3^{rd}$ authors have used it earlier to make progress on the \emph{visibility problem} \cite{2023arXiv230516026D,MR4592863}. 

In the context of proving \eqref{form72}, Proposition \ref{prop3} applied to $\mathbf{X} = \mathbf{X}[\mu|_{F},\{\sigma_{x}\}]$ enables the following argument: instead of proving \eqref{form72} directly, it suffices to demonstrate that
\begin{equation}\label{form74} G \subset \{\mathbf{X} \ast \varphi_{\Delta} \approx 1\} \end{equation}
for some small but fixed $\Delta > 0$. This proves \eqref{form72}: since $\mathcal{H}_{\infty}^{t}(G) \sim 1$, we may infer from \eqref{form73} that $G \not\subset \{\mathbf{X} \ast \varphi_{\Delta} \approx 1\} \, \setminus \, \spt \mathbf{X}$ (if $\Delta > 0$ is sufficiently small).

\begin{remark} For readers familiar with the arguments in \cite{2023arXiv230518253C}, the reduction from proving \eqref{form72} to proving \eqref{form74} can be compared (at least philosophically) with the steps of proving the "initial incidence lower bound at a large scale" and then deducing an incidence lower bound at a (given) small scale as a corollary, see \cite[Section 2]{2023arXiv230518253C} for details. We thank Alex Cohen for bringing this connection to our attention. The comparison of incidences at large and small scales in \cite{2023arXiv230518253C} is implemented by the \emph{high-low method}. Here, we instead rely on Proposition \ref{prop3}. However, the proof of Proposition \ref{prop3} bears resemblance to the proof of the key proposition in the high-low method, see \cite[Proposition 2.1]{GSW}.   \end{remark}

We have now reduced the proof of Theorem \ref{t:main} to the task of finding a small scale $\Delta > 0$, and two well-separated, relatively large, sets $F,G \subset K$ such that \eqref{form74} holds. This is the most technical part of the argument. Section \ref{s6} contains Lemma \ref{l:main}, where the sets $F,G$ are constructed under an additional hypothesis on the geometry of $K$ and the families $E_{x}(K)$ at scale $\Delta$. This hypothesis is called \emph{tightness} (of the configuration $(\mu,\{\sigma_{x}\})$ at scale $\Delta$), see Definition \ref{def:tightness}. Informally speaking, tightness means that all the sets $E_{x}(K)$ essentially coincide at scale $\Delta$. This would be too much to literally require, but for the discussion in this introduction it is instructive to interpret ``tightness''  as ``$E_{x}(K) \equiv E \subset S^{1}$ for all $x \in K$''.

Tightness is crucial for obtaining \eqref{form74}, because it can be used to show that "many tubes intersect the expected number of squares at scale $\Delta$". This is needed for obtaining lower bounds for $\mathbf{X} \ast \varphi_{\Delta}$. We give a few more details. Let $\mathcal{Q} := \mathcal{D}_{\Delta}(\spt \mu)$, and for each $Q \in \mathcal{Q}$, let $\mathbb{T}_{Q}$ be the family of $\Delta$-tubes intersecting $Q$ containing all the lines with directions in $E$. The \emph{slope set} $\sigma(\mathbb{T}_{Q}) \subset \Delta \cdot \Z$ is now independent of $Q \in \mathcal{Q}$, thanks to tightness. Let $M := |\sigma(\mathbb{T}_{Q})|$ (for any $Q \in \mathcal{Q}$). The "total" tube family $\mathbb{T} = \bigcup_{Q \in \mathcal{Q}} \mathbb{T}_{Q}$ satisfies $|\mathbb{T}| \lesssim M\Delta^{-1}$,
simply because there can be $\lesssim \Delta^{-1}$ tubes in a fixed direction. On the other hand, every square $Q \in \mathcal{Q}$ meets at least $M$ tubes from $\mathbb{T}$, namely those in $\mathbb{T}_{Q}$. Therefore,
\begin{displaymath} |\mathcal{I}(\mathcal{Q},\mathbb{T})| := |\{(Q,T) \in \mathcal{Q} \times \mathbb{T} : Q \cap T \neq \emptyset\}| \geq M|\mathcal{Q}|. \end{displaymath}
Putting these estimates together, we obtain the following lower bound for the average number of squares in $\mathcal{Q}$ meeting a tube in $\mathbb{T}$:
\begin{displaymath} \frac{|\mathcal{I}(\mathcal{Q},\mathbb{T})|}{|\mathbb{T}|} \gtrsim \frac{M|\mathcal{Q}|}{M\Delta^{-1}} = \Delta|\mathcal{Q}|. \end{displaymath} 
This number coincides with the "expected" value, for example in a situation where the squares in $\mathcal{Q}$ would be selected randomly from $\mathcal{D}_{\Delta}$. 

After some pigeonholing, every tube in $\mathbb{T}$ contains (at least) the "expected" number of squares from $\mathcal{Q}$. Using the definition of the $X$-ray measure $\mathbf{X}$, this leads to the conclusion that $\mathbf{X} \ast \varphi_{\Delta} \gtrapprox 1$ on most squares $Q \in \mathcal{Q}$. Finding the well-separated sets $F,G \subset K$ is relatively easy at this point; we employ a lemma of Erd\H{o}s \cite[Lemma 1]{MR190027} which guarantees the existence of large bi-partite sub-graphs inside arbitrary (undirected) graphs.

There is probably no reason why the tightness hypothesis would be satisfied by the "original" configuration $(\mu,\{\sigma_{x}\})$ at any scale $\Delta > 0$. However, in Section \ref{s7} we show that there exists a relatively large dyadic square $Q \subset [0,1)^{2}$ such that the \emph{renormalised configuration} $(\mu^{Q},\{\sigma^{Q}_{x}\})$ satisfies tightness at some small scale $\Delta > 0$. Here 
\begin{displaymath} \mu^{Q} = \mu(Q)^{-1}S_{Q}(\mu|_{Q}) \quad \text{and} \quad \sigma_{x}^{Q} = \sigma_{S_{Q}^{-1}(x)}, \end{displaymath}
where $S_{Q}$ is a homothety mapping $Q$ to $[0,1)^{2}$. A similar argument first appeared in the proof of \cite[Theorem 5.7]{2023arXiv230110199O}. It is a quantitative version of the following rough heuristic: the map $x \mapsto E_{x}(K)$ is Borel, hence approximately continuous. Thus, at a typical $x$, after sufficient "zooming in" the sets $E_{x}(K)$ do not vary too much. Making this idea precise requires a lot of pigeonholing, most of which is carried out in Proposition \ref{prop6}. 

Eventually, Lemma \ref{l:main} is applied to the tight configuration $(\mu^{Q},\{\sigma_{x}^{Q}\})$. This leads to an analogue of \eqref{form72} for $\mu^{Q}$ and $\sigma_{x}^{Q}$. Since we have an effective lower bound on the side-length of $Q$, this turns out to be good enough to conclude the proof of Theorem \ref{t:main}. 

We conclude the discussion by mentioning another major technicality: since the square $Q$ is selected by pigeonholing, the renormalisation $\mu^{Q}$ may fail to be $t$-Frostman (if $\mu(Q) \ll \ell(Q)^{t}$). This is problematic, since we need  the Sobolev regularity of $\mathbf{X}[\mu^{Q},\{\sigma_{x}^{Q}\}]$, and that was crucially based on the $t$-Frostman property of $\mu^{Q}$. 

We solve this issue by proving something like this: given $0 < \tau < t$ and a long enough scale sequence $\Delta_{n} < \Delta_{n - 1} < \ldots < \Delta_{0} = 1$, there exist $\gtrsim_{t,\tau} n$ scales $\Delta_{j}$ such that $\mu^{Q}$ is $\tau$-Frostman for some $Q \in \mathcal{D}_{\Delta_{j}}$. The precise formulation, Proposition \ref{lemma5}, is actually stated in terms of \emph{uniform $(\delta,t)$-sets} and involves no measures. In any case, this result allows us to restrict our search for "$Q$" to the "good scales" where we know that $\mu^{Q}$ still satisfies a $\tau$-dimensional Frostman condition for $\tau < t$ arbitrarily close to $t$.

\subsection{Acknowledgements} We are grateful to Pertti Mattila for bringing \cite[Theorem 6.9]{MR3617376} to our attention, and at the same time mentioning that the proof is incomplete.

\section{Notation and preliminaries}

We write $A \lesssim B$ if there exists a constant $C$ independent of scale so that $A \le C B$. Then $A \gtrsim B$ and $A \sim B$ are defined analogously. In contrast, given some scale $\delta$, unless otherwise stated, $A \lessapprox_{\delta} B$ means either that for any arbitrarily small $\epsilon$, and all sufficiently small $\delta$, $A \le \delta^{-\epsilon} B$ or that there exists an absolute constant $C$ so that for all sufficiently small $\delta$, $A \le \log(1/\delta)^{C} B$. The notation $\gtrapprox_{\delta}$ and $\approx_{\delta}$ are defined analogously.

\subsection{A metric on the space of lines}  The family of all lines in $\R^{2}$ is denoted $\mathcal{A}(2,1)$. The lines containing $0$ are denoted $\mathcal{G}(2,1)$. We define the following metric on $\mathcal{A}(2,1)$:
\begin{displaymath} d_{\mathcal{A}(2,1)}(\ell_{1},\ell_{2}) := \|\pi_{L_{1}} - \pi_{L_{2}}\| + |a_{1} - a_{2}|, \qquad \ell_{j} = a_{j} + L_{j}. \end{displaymath}
Here $L_{j} \in \mathcal{G}(2,1)$ is a $1$-dimensional subspace of $\R^{2}$, $\pi_{L_{j}}$ denotes the orthogonal projection onto $L_{j}$, $\| \cdot \|$ denotes the operator norm, and $a_{j} \in L_{j}^{\perp}$. For $\mathcal{L} \subset \mathcal{A}(2,1)$, the Hausdorff dimension $\Hd \mathcal{L}$ is defined relative to the metric $d_{\mathcal{A}(2,1)}$.

\subsection{Dyadic cubes, tubes, $(\delta,t)$-sets, and Frostman measures}

\begin{definition}[Dyadic cubes]\label{def:dyadicSquares} For $\delta \in 2^{-\N}$ and $A \subset \R^{d}$, the notation $\mathcal{D}_{\delta}(A)$ stands for the family of standard dyadic $\delta$-cubes intersecting $A$. In the important special case $A = [0,1)^{d}$, we abbreviate $\mathcal{D}_{\delta} := \mathcal{D}_{\delta}([0,1)^{d})$.  \end{definition}

\begin{definition}[$(\delta,t)$-sets]
	Let $(X,d)$ be a metric space, in practice $X = \R^{d}$ or $X = \mathcal{A}(2,1)$. For $C,\delta,t > 0$, we say that a set $P\subset (X,d)$ is a \emph{$(\delta,t,C)$-set} if
	\begin{equation*}
	|P\cap B(x,r)|_{\delta} \le Cr^t|P|_{\delta}, \qquad x\in P, \, \delta\le r\le 1. 
	\end{equation*}
	Here $|\cdot|_{\delta}$ stands for the $\delta$-covering number. A family $\mathcal{P} \subset \mathcal{D}_{\delta}$ is called a $(\delta,t,C)$-set if the union $P := \cup \mathcal{P}$ is a $(\delta,t,C)$-set.
\end{definition}
\begin{definition}[Katz-Tao $(\delta,t)$-sets]\label{def:KatzTao}
	Let $(X,d)$ be a metric space. For $C,\delta,t > 0$, we say that a set $P\subset (X,d)$ is a \emph{Katz-Tao $(\delta,t,C)$-set} if
	\begin{equation*}
	|P\cap B(x,r)|_{\delta} \le C\left(\frac{r}{\delta}\right)^t, \qquad x\in P, \, \delta\le r\le 1. 
	\end{equation*}
	A family $\mathcal{P} \subset \mathcal{D}_{\delta}$ is called a Katz-Tao $(\delta,t,C)$-set if the union $P := \cup \mathcal{P}$ is a Katz-Tao $(\delta,t,C)$-set.
\end{definition}

\begin{remark}\label{rem4} The $(\delta,t,C)$-set property is not inherited by subsets, except if the subset has nearly maximal $\delta$-covering number. Indeed, if $\mathcal{P} \subset \mathcal{D}_{\delta}$ is a $(\delta,t,C)$-set, and $\mathcal{P}' \subset \mathcal{P}$ is arbitrary, then $\mathcal{P'}$ is a $(\delta,t,C|\mathcal{P}|/|\mathcal{P}'|)$-set.

On the other hand, it is clear from the definition that the Katz-Tao $(\delta,t,C)$-set property is inherited by subsets, with the same constant $C$.  \end{remark}

\begin{definition}[Point-line duality map]\label{def:pointLineDuality} The following map $\mathbf{D} \colon \R^{2} \to \mathcal{A}(2,1)$ is called the \emph{point-line duality map}:
\begin{displaymath} \mathbf{D}(a,b) := \{(x,y) \in \R^{2} : y = ax + b\}, \qquad (a,b) \in \R^{2}. \end{displaymath}
\end{definition}

\begin{definition}[Dyadic tubes]\label{def:dyadicTubes} Let $\delta \in 2^{-\N}$. For any $p \in \mathcal{D}_{\delta}([-1,1) \times \R)$, the set $\mathbf{D}(p)$ is a collection of lines in $\mathcal{A}(2,1)$, and $T = \cup \mathbf{D}(p) \vcentcolon= \cup_{\ell \in \mathbf{D}(p)} \ell \subset \R^{2}$ is called a \emph{dyadic $\delta$-tube}. The family of all dyadic $\delta$-tubes is denoted $\mathcal{T}^{\delta}$. For $p = [a,a + \delta) \times [b,b + \delta) \in \mathcal{D}_{\delta}([-1,1) \times \R)$ with $a \in (\delta \cdot \Z) \cap [-1,1)$ and $b \in \delta \cdot \Z$, the \emph{slope} of the dyadic $\delta$-tube $T = \cup \mathbf{D}(p)$ is defined to be $\sigma(T) := a$. For a family $\mathcal{T} \subset \mathcal{T}^{\delta}$, we write
\begin{displaymath} \sigma(\mathcal{T}) := \{\sigma(T) : T \in \mathcal{T}\} \subset (\delta \cdot \Z) \cap [-1,1). \end{displaymath}
This is the \emph{slope set of $\mathcal{T}$}.  \end{definition}

\begin{remark} A dyadic $\delta$-tube $T = \cup \mathbf{D}(p)$ is a subset of $\R^{2}$. If there is no risk of confusion, we reserve the right to view $T$ as a subset of $\mathcal{A}(2,1)$, i.e., as $\mathbf{D}(p)$, when convenient. In particular, if $\mathcal{L} \subset \mathcal{A}(2,1)$ is a set of lines, and $T \in \mathcal{T}^{\delta}$, we will write 
\begin{displaymath} \mathcal{L} \cap T := \{\ell \in \mathcal{L} : \ell \in \mathbf{D}(p)\} = \{\ell \in \mathcal{L} : \ell \subset T\}. \end{displaymath} \end{remark}

\begin{definition}[$(\delta,s)$-sets of dyadic tubes]\label{def:deltaSTubes} For $C,s > 0$, a family $\mathcal{T} = \{\cup \mathbf{D}(p) : p \in \mathcal{P}\} \subset \mathcal{T}^{\delta}$ is called a $(\delta,s,C)$-set if the set $\mathcal{P} \subset \mathcal{D}_{\delta}([-1,1) \times \R)$ is a $(\delta,s,C)$-set. The Katz-Tao $(\delta,s,C)$-sets of tubes are defined analogously. \end{definition}

\begin{remark} Another, equivalent up to constants, definition would be to require that the line set $\cup \{\mathbf{D}(p) : p \in \mathcal{P}\} \subset \mathcal{A}(2,1)$ is a $(\delta,s,C)$-set in the metric $d_{\mathcal{A}(2,1)}$.  \end{remark}

The $(\delta,s)$-set properties of $\mathcal{T}$ and $\sigma(\mathcal{T})$ coincide whenever all the elements of $\mathcal{T}$ intersect a common $\delta$-square:

\begin{lemma}\label{lemma4} Let $C,s > 0$. Assume that $\mathcal{T}_{p} \subset \mathcal{T}^{\delta}$ is a family of dyadic $\delta$-tubes, each of which intersects a common square $p \in \mathcal{D}_{\delta}$. If $\mathcal{T}_{p}$ is a $(\delta,s,C)$-set, then $\sigma(\mathcal{T}_{p})$ is a $(\delta,s,C')$-set, where $C' \sim C$. Conversely, if $\sigma(\mathcal{T}_{p})$ is a $(\delta,s,C)$-set, then $\mathcal{T}_{p}$ is a $(\delta,s,C')$-set with $C' \sim C$. \end{lemma}

\begin{proof} See \cite[Corollary 2.12]{OS23}. \end{proof}

A Borel measure $\mu$ on a metric space $(X,d)$ is called an $(s,C)$-Frostman measure if $\mu(B(x,r)) \le C r^{s}$ for all $r > 0$ and $x \in X$. A Borel measure $\mu$ is called an $s$-Frostman measure if there exists some $C$ so that it is an $(s,C)$-Frostman measure.

\subsection{Configurations and the $X$-ray measure} We now introduce the main measure theoretic tool in the proof of Theorem \ref{t:main}.

\begin{definition}\label{def:configuration} A \emph{configuration} is a pair $(\mu,\{\sigma_{x}\})$, where $\mu$ is a Radon measure on $[0,1]^{2}$, and for each $x \in \R^{2}$, $\sigma_{x}$ is a Radon measure on $[0,1]$, and
\begin{displaymath} x \mapsto \int g(x,\theta) \, d\sigma_{x}(\theta) \end{displaymath}
is a Borel function for all $g \in C(\R^{2} \times [0,1])$.    \end{definition}

\begin{remark} The correct intuition is to think of $\sigma_{x}$ as a measure on $S^{1}$, but having the support contained in $[0,1]$ brings minor technical convenience. \end{remark} 

\begin{remark} The configurations we really encounter in this paper are of the following special form. For each dyadic square $p \in \mathcal{D}_{\delta}$ of side-length $\delta > 0$ (for $\delta \in 2^{-\N}$ arbitrarily small but fixed), the measures $\sigma_{x}$ coincide for all $x \in p$. The measurability hypothesis in Definition \ref{def:configuration} is clear for such configurations, in fact $x \mapsto \int g(x,\theta) \, d\sigma_{x}(\theta)$ is piecewise continuous for $g \in C(\R^{2} \times [0,1])$.  \end{remark} 

\begin{definition}[$\mu(\sigma_{x})$]\label{def:productMeasure} Given a configuration $(\mu,\{\sigma_{x}\})$, we define the Radon measure $\mu(\sigma_{x})$ on $[0,1] \times \R$ by 
\begin{equation}\label{form15} \int g \, d\mu(\sigma_{x}) := \int \int g(\theta,\pi_{\theta}(x)) \, d\sigma_{x}(\theta) \, d\mu(x), \qquad g \in C_{c}([0,1] \times \R). \end{equation}
 \end{definition}
 
 \begin{remark} The definition is well-posed, since the right hand side of \eqref{form15} defines a positive linear functional on $C_{c}([0,1] \times \R)$ by the hypotheses in Definition \ref{def:configuration} (applied to the continuous function $\tilde{g}(x,\theta) := g(\theta,\pi_{\theta}(x))$). 
 
One should view $\mu(\sigma_{x})$ as a measure on the space of lines $\mathcal{A}(2,1)$. This makes sense when one parametrises $\mathcal{A}(2,1)$ by $(\theta,r) \mapsto \pi_{\theta}^{-1}\{r\}$, where 
 \begin{displaymath} \pi_{\theta}(x) := x \cdot (\cos 2\pi \theta,\sin 2\pi \theta), \qquad \theta \in [0,1], \, x \in \R^{2}. \end{displaymath}
 Under this parametrisation, the point $(\theta,\pi_{\theta}(x))$ appearing in the argument of $g$ in \eqref{form15} corresponds to the line $\ell_{x,\theta} := \pi_{\theta}^{-1}\{\pi_{\theta}(x)\}$ containing $x$. 
 
 Arguably, a more natural way to define $\mu(\sigma_{x})$ would be to first transfer each $\sigma_{x}$ to a measure $\mathcal{L}_{x}$ supported on the lines passing through $x$, and then define $\mu(\sigma_{x}) = \mu(\mathcal{L}_{x})$ as a $\mu$-weighted average over the measures $\mathcal{L}_{x}$. With this method one would end up with measure on $\mathcal{A}(2,1)$. However, in the sequel it is technically simpler to deal with a measure supported on the parameter space $[0,1] \times \R$.  \end{remark}

\begin{definition}[$X$-ray measure]\label{def:XRayMeasure} Given a configuration $(\mu,\{\sigma_{x}\})$, we define the \emph{$X$-ray measure} $X[\mu,\{\sigma_{x}\}]$ as the Radon measure on $\R^2$ determined by the relation
\begin{equation}\label{form47} \int f \, dX[\mu,\{\sigma_{x}\}] := \int \int Xf(\theta,\pi_{\theta}(x)) \, d\sigma_{x}(\theta) \, d\mu(x), \qquad f \in C_{c}(\R^{2}), \end{equation}
where $Xf(\theta,r) := \int_{\pi_{\theta}^{-1}\{r\}} f \, d\mathcal{H}^{1}$ is the \emph{$X$-ray transform} of $f$ evaluated at $(\theta,r)$. \end{definition}

\begin{remark} The definition is well-posed, since $Xf \in C_{c}([0,1] \times \R)$ for $f \in C_{c}(\R^{2})$. Therefore $x \mapsto \int Xf(\theta,\pi_{\theta}(x)) \, d\sigma_{x}(\theta)$ defines a bounded Borel function by the hypothesis in Definition \ref{def:configuration}, applied to the continuous function $g(x,\theta) := Xf(\theta,\pi_{\theta}(x))$. So, the RHS of \eqref{form47} defines a positive linear functional on $C_{c}(\R^{2})$.  \end{remark}

\begin{remark}\label{rem5} We record here that if $(\mu,\{\sigma_{x}\})$ is a configuration with $K := \spt \mu$, then the associated $X$-ray measure is supported on the set
\begin{displaymath} \mathbf{X}(\mu,\{\sigma_{x}\}) := \overline{\bigcup_{x \in K} L_{x}}, \end{displaymath}
where $L_{x} := \cup \{\ell_{x,\theta} : \theta \in \spt \sigma_{x}\}$. This follows readily from \eqref{form47}, but let us spell out the details. Let $f \in C_{c}(\R^{2})$ be a continuous function with support disjoint from $\mathbf{X}(\mu,\{\sigma_{x}\})$. Then, $f \equiv 0$ on each line $\ell_{x,\theta}$ with $x \in K$ and $\theta \in \spt \sigma_{x}$. Therefore, 
\begin{displaymath} Xf(\theta,\pi_{\theta}(x)) = \int_{\pi_{\theta}^{-1}\{\pi_{\theta}(x)\}} f \, d\mathcal{H}^{1} = \int_{\ell_{x,\theta}} f \, d\mathcal{H}^{1} = 0, \qquad x \in K, \, \theta \in \spt \sigma_{x}, \end{displaymath}
and so $\int f \, dX[\mu,\{\sigma_{x}\}] = 0$.  \end{remark}

There is a simple relationship between the measures in Definitions \ref{def:productMeasure} and \ref{def:XRayMeasure}. To describe it, we recall the definition of the adjoint $X$-ray transform:
\begin{definition}\label{def:AdjointXRayTransform} Let $\nu$ be a Radon measure on $[0,1] \times \R$. The \emph{adjoint $X$-ray transform of $\nu$}, denoted $X^{\ast}\nu$, is the Radon measure on $\R^{2}$ determined by the relation
\begin{displaymath} \int f \, d(X^{\ast}\nu) := \int Xf \, d\nu = \int \int_{\pi_{\theta}^{-1}\{r\}} f \, d\mathcal{H}^{1} \, d\nu(\theta,r), \qquad f \in C_{c}(\R^{2}). \end{displaymath}
\end{definition}

We can now clarify the connection between the measures in Definitions \ref{def:productMeasure} and \ref{def:XRayMeasure}:
\begin{proposition}\label{prop:adjRad} Let $(\mu,\{\sigma_{x}\})$ be a configuration. Then,
\begin{displaymath} X[\mu,\{\sigma_{x}\}] = X^{\ast}\mu(\sigma_{x}). \end{displaymath}
 \end{proposition}
 
 \begin{proof} Fix $f \in C_{c}(\R^{2})$. Consecutively applying Definitions \ref{def:XRayMeasure}, \ref{def:productMeasure}, and \ref{def:AdjointXRayTransform} confirms
 \begin{displaymath} \int f \, d X[\mu,\{\sigma_{x}\}] =\int \int Xf (\theta,\pi_{\theta}(x)) \, d\sigma_{x}(\theta) \, d\mu(x) = \int Xf \, d\mu(\sigma_{x}) = \int f \, dX^{\ast}\mu(\sigma_{x}). \end{displaymath}
 \end{proof} 


\section{From sets to configurations}\label{s3}

In this section we state a measure theoretic variant of Theorem \ref{t:main} formulated in terms of configurations, see Theorem \ref{t:configurations}, and demonstrate how Theorem \ref{t:main} can be deduced from Theorem \ref{t:configurations}. Here is Theorem \ref{t:main} once more (with the small modification that now $\ell_{x,\theta} := \pi_{\theta}^{-1}\{\pi_{\theta}(x)\}$ for $\theta \in [0,1]$):

\begin{thm}\label{t:mainRestated} Let $t \in [1,2]$, and let $K \subset \R^{2}$ be compact. For $x \in K$, let 
\begin{displaymath} E_{x}(K) := \{\theta \in [0,1] : K \cap \ell_{x,\theta} = \{x\}\} = \{\theta \in [0,1] : |K \cap \ell_{x,\theta}| = 1\}. \end{displaymath}
Then, $\Hd E_{x}(K) \leq 2 - t$ for $\mathcal{H}^{t}$ almost all $x \in K$. \end{thm} 

\begin{thm}\label{t:configurations} For every $t \in (1,2]$, $s \in (2 - t,1]$, and $C > 0$, there exist a radius $r = r(C,s,t) > 0$ such that the following holds. Let $(\mu,\{\sigma_{x}\})$ be a configuration, where $\mu$ is a $(t,C)$-Frostman probability measure, and $\sigma_{x}$ is an $(s,C)$-Frostman probability measure for $\mu$ almost all $x \in \R^{2}$. Then, \begin{displaymath} \inf \{\dist(y,L_{x}) : x,y \in \spt \mu, \, |x - y| \geq r\} = 0, \end{displaymath} 
where $L_{x} := \cup \{\ell_{x,\theta} : \theta \in \spt \sigma_{x}\}$, and $\ell_{x,\theta} = \pi_{\theta}^{-1}\{\pi_{\theta}(x)\}$.
 \end{thm}

We now use Theorem \ref{t:configurations} to derive Theorem \ref{t:mainRestated}. 

\begin{proof}[Proof of Theorem \ref{t:mainRestated}] The conclusion is clear for $t = 1$, so assume $t \in (1,2]$. We make a counter assumption: there exist $t \in (1,2]$, and a compact set $K \subset \R^{2}$ such that 
\begin{displaymath} \mathcal{H}_{\infty}^{t}(\{x \in K : \Hd E_{x}(K) > 2 - t\}) > 0. \end{displaymath}
Consequently, by the sub-additivity of $\mathcal{H}^{t}_{\infty}$, there exist $s > 2 - t$ and $\kappa > 0$ such that
\begin{equation}\label{form65} \mathcal{H}^{t}_{\infty}(\{x \in K : \mathcal{H}^{s}_{\infty}(E_{x}(K)) > \kappa\}) > \kappa. \end{equation}
For $r > 0$, let 
\begin{displaymath}	E_x^{r} := E^{r}_{x}(K) \coloneqq \{\theta \in [0,1] : \ell_{x,\theta} \cap K\subset B(x,r/2)\}, \end{displaymath}
and note that $E_{x}(K) \subset \bigcap_{r > 0} E_{x}^{r}$. In particular, \eqref{form65} implies
\begin{equation}\label{form66} \mathcal{H}^{t}_{\infty}(\{x \in K : \mathcal{H}^{s}_{\infty}(E^{r}_{x}) > \kappa\}) > \kappa, \qquad r > 0. \end{equation}
This may seem like a wasteful use of \eqref{form65}, but the idea is that we will obtain a contradiction for a fixed $r > 0$, depending only on $\kappa,s,t$. In fact, we can describe that "$r$" immediately: let $\mathbf{C} \geq 1$ be an absolute constant determined in a moment, and let  
	\begin{displaymath} r := r(\mathbf{C}\kappa^{-1},s,t) > 0 \end{displaymath}
be the constant given by Theorem \ref{t:configurations}. For every $\epsilon>0$, we further define
	\begin{equation}\label{form53}
	E_x^{r,\epsilon}\coloneqq \{\theta \in [0,1] : [\ell_{x,\theta}]_{\epsilon} \cap K\subset B(x,r/2)\},
	\end{equation}
	where $[\ell_{x,\theta}]_{\epsilon}$ is the $\epsilon$-neighbourhood of $\ell_{x,\theta}$. We claim that 
	\begin{equation}\label{eq:lines}
	E_x^{r}  = \bigcup_{\epsilon>0} E_x^{r,\epsilon}.
	\end{equation}
	The inclusion $E_{x}^{r,\epsilon} \subset E_{x}^{r}$ holds for all $\epsilon > 0$, so it suffices to prove the inclusion "$\subset$". Suppose that $\theta \in E_x^{r}$, so $\ell_{x,\theta} \cap K\subset B(x,r/2)$. If $\theta \notin \bigcup_{\epsilon>0} E_x^{r,\epsilon}$, then for every $\epsilon>0$ there exists $y_\epsilon\in [\ell_{x,\theta}]_{\epsilon} \cap K \, \setminus \, B(x,r/2)$. Since $K$ is compact, we may find a sequence $\epsilon_n$ such that $y_{\epsilon_n}\to y\in K\cap \ell_{x,\theta}$. Since $|x-y_{\epsilon_n}|\ge r/2$, we also have $|x-y|\ge r/2$. But this is a contradiction with the definition of $E_x^{r}$. This shows \eqref{eq:lines}.
	
	Note that as $\epsilon \searrow 0$, the sets $E_{x}^{r,\epsilon}$ increase to $E_{x}^{r}$ (by \eqref{eq:lines}). Therefore, whenever $x \in K$ and $\mathcal{H}^{s}_{\infty}(E_{x}^{r}) > \kappa$, Davies' increasing sets lemma \cite[Theorem 4]{MR260959} implies that also $\mathcal{H}^{s}_{\infty}(E_{x}^{r,\epsilon}) > \kappa$ for $\epsilon > 0$ sufficiently small. This observation implies that also the sets $\{x\in K \, :\, \mathcal{H}^{s}_{\infty}(E^{r,\epsilon}_{x}) > \kappa\}$ increase to $\{x \in K : \mathcal{H}^{s}_{\infty}(E_{x}^{r}(K)) > \kappa\}$ as $\epsilon \searrow 0$. By a second application of \cite[Theorem 4]{MR260959}, and recalling \eqref{form66}, we deduce that
	\begin{displaymath} \mathcal{H}^{t}_{\infty}(\{x \in K : \mathcal{H}^{s}_{\infty}(E^{r,\epsilon}_{x}) > \kappa\}) > \kappa \end{displaymath}
	for $\epsilon > 0$ sufficiently small. Fix such an $\epsilon > 0$, and write $K_{\epsilon} := \{x \in K : \mathcal{H}^{s}_{\infty}(E_{x}^{r,\epsilon}) > \kappa\}$.
	
	Fix $\delta \in 2^{-\N}$ with $\delta \leq c\min\{\epsilon,r\}$ for a suitable small constant $c > 0$. Let $\mathcal{P} \subset \mathcal{D}_{\delta}$ be a non-empty $(\delta,t,C\kappa^{-1})$-set with the property $K_{\epsilon} \cap p \neq \emptyset$ for all $p \in \mathcal{P}$ (such $\calP$ exists by Lemma 2.13 in \cite{FaO}). For each $p \in \mathcal{P}$, choose a distinguished point $x_{p} \in K_{\epsilon} \cap p$. Since $\mathcal{H}^{s}_{\infty}(E^{r,\epsilon}_{x_{p}}) > \kappa$, we may use Frostman's lemma to find a $(s,C\kappa^{-1})$-Frostman probability measure $\sigma_{x_{p}}$ supported on $E_{x_{p}}^{r,\epsilon}$. 
		
We now define the configuration $(\mu,\{\sigma_{x}\})$. First, we define $\mu$ to be the uniformly distributed probability measure on $\cup \mathcal{P}$. It is easy to check that $\mu$ is a $(t,\mathbf{C}\kappa^{-1})$-Frostman probability measure with $\mathbf{C} \lesssim C$.

Second, we need to define the measures $\sigma_{x}$ for $x \in \spt \mu$. We do this separately for every square $p \in \mathcal{P}$. For $p \in \mathcal{P}$ fixed, and $x \in \bar{p}$, let $\sigma_{x} := \sigma_{x_{p}}$. Then $\sigma_{x}$ is an $(s,C\kappa^{-1})$-Frostman probablity measure for every $x \in \spt \mu$. Then, for $g \in C(\R^{2} \times [0,1])$ the maps
\begin{displaymath} x \mapsto \int g(x,\theta) \, d\sigma_{x}(\theta), \end{displaymath}
are piecewise continuous and hence measurable. So $(\mu,\{\sigma_{x}\})$ is indeed a configuration.

We finally apply Theorem \ref{t:configurations} to the configuration $(\mu,\{\sigma_{x}\})$. In particular, there exist points $x,y \in \spt \mu$ such that 
\begin{equation}\label{form54} |x - y| \geq r \quad \text{and} \quad \dist(y,L_{x}) \leq \delta, \end{equation}
where $L_{x} = \cup \{\ell_{x,\theta} : \theta \in \spt \sigma_{x}\}$. Let $p,q \in \mathcal{P}$ be the squares such that $x \in \bar{p}$ and $y \in \bar{q}$. Let $x_{p} \in K_{\epsilon} \cap p$ and $y_{q} \in K_{\epsilon} \cap q$ be the associated "distinguished points". Now, $\dist(y,L_{x}) \leq \delta$ implies $\dist(y_{q},L_{x_{p}}) \lesssim \delta$, since $L_{x_{p}}$ is a translate of $L_{x}$ by $(x_{p} - x) \in [-\delta,\delta]^{2}$. In particular, there exists a $\theta \in E_{x_{p}}^{r,\epsilon}$ such that 
\begin{displaymath} y_{q} \in [\ell_{x_{p},\theta}]_{C\delta} \cap K_{\epsilon} \subset [\ell_{x_{p},\theta}]_{\epsilon} \cap K. \end{displaymath}
But since $\theta \in E_{x_{p}}^{r,\epsilon}$, this should mean by definition (recall \eqref{form53}) that $y_{q} \in B(x_{p},r/2)$, and therefore $|x - y| < r$. This contradicts \eqref{form54} and completes the proof of Theorem \ref{t:mainRestated}. \end{proof}


\section{Estimates for $X$-ray measures}

In Section \ref{s3} we reduced the main result (Theorem \ref{t:main}) to a measure theoretic version concerning configurations (Theorem \ref{t:configurations}). The purpose of this section is to prove Theorem \ref{t:configurations} under additional hypotheses on the behaviour of the $X$-ray measure (Theorem \ref{t2} below). Establishing the validity of these hypotheses will, roughly speaking, occupy the remainder of the paper. The proof of Theorem \ref{t:configurations} will be completed in Section \ref{s4}.

\begin{thm}\label{t2} For every $t \in (1,2]$ and $s \in (2 - t,1]$, there exist $\eta = \eta(s,t) > 0$ and $\Delta_{0} = \Delta_{0}(s,t) > 0$ such that the following holds for all $\Delta \in (0,\Delta_{0}]$. 
	
	Let $(\mu,\{\sigma_{x}\})$ be a configuration with the properties that $\mu$ is $(t,\Delta^{-\eta})$-Frostman, and each $\sigma_{x}$ is $(s,\Delta^{-\eta})$-Frostman. Assume additionally that there exist Borel sets $G_{1},G_{2} \subset \spt \mu$ such that:
	\begin{itemize}
		\item[(H1) \phantomsection \label{H1}] $\mu(G_{2}) \geq \Delta^{\eta}$.
		\item[(H2) \phantomsection \label{H2}] $G_{2} \subset \{z \in \R^{2} : X[\mu|_{G_{1}},\{\sigma_{x}\}]_{\Delta}(z) \geq \Delta^{\eta}\}$.
	\end{itemize} 
	Then, $\inf \{\dist(y,L_{x}) : x \in G_{1} \text{ and } y \in G_{2}\} = 0$, where $L_{x} = \cup \{\ell_{x,\theta} : \theta \in \spt \sigma_{x}\}$, and we recall that $\ell_{x,\theta} = \pi_{\theta}^{-1}\{\pi_{\theta}(x)\}$ for $x \in \R^{2}$ and $\theta \in [0,1]$. \end{thm} 

The remainder of this section is dedicated to the proof of Theorem \ref{t2}.
\subsection{Sobolev norms}
The proof of Theorem \ref{t2} relies on the theory of fractional Sobolev spaces. We deal with some preliminaries in this subsection. In this section, $\mathcal{M}_{\C}(X)$ refers to complex Borel measures on $X$, and $\mathcal{M}(X)$ refers to finite positive Borel measures on $X$ (where always $X \subset \R^{n}$).
\begin{definition}[Sobolev norms on $\R^2$]	
	For $\nu \in \mathcal{M}_{\C}(\R^{2})$, and $s>-1$ we define the homogeneous Sobolev norm
	\begin{equation*}
	\|\nu\|_{\dot{H}^s(\R^2)}^2 \coloneqq \int_{\R^2} |\hat{\nu}(\xi)|^2 |\xi|^{2s}\, d\xi.
	\end{equation*}
\end{definition}

It is well-known (see \cite[Proposition 1.36]{bahouri2011}) that if $\nu \in \mathcal{M}_{\mathbb{C}}(\R^{2})$, $|s|<1$, and $|\int f \, d\nu|$ is uniformly bounded for all $f \in \mathcal{S}(\R^{2})$ with $\|f\|_{\dot{H}^{-s}} \leq 1$, then in fact $\|\nu\|_{\dot{H}^{s}(\R^{2})} < \infty$, and
\begin{equation*}
\|\nu\|_{\dot{H}^s(\R^2)} = \sup_{f\in \calS(\R^2),\, \|f\|_{\dot{H}^{-s}(\R^2)\le 1}} \left| \int f \, d\nu \right|. \end{equation*}

\begin{definition}[Riesz energy]
	Given $0<t<2$ and $\mu \in \mathcal{M}(\R^{2})$, we define the Riesz $t$-energy of $\mu$ by
	\begin{equation*}
	I_t(\mu)=\iint \frac{d\mu(x)d\mu(y)}{|x-y|^t}.
	\end{equation*}
\end{definition}
By Theorem 3.10 in \cite{MR3617376}, the following relation holds between Riesz energies and homogeneous Sobolev norms:
\begin{equation}\label{eq:Riesz}
I_t(\mu)= c_t\|\mu\|_{\dot{H}^{t/2-1}(\R^2)}^2, \qquad \mu \in \mathcal{M}(\R^{2}), \, 0 < t < 2.
\end{equation}

We need a version of Sobolev norms for measures on $[0,1]\times\R$. We follow the conventions from \cite[Section 2]{MR4722034}. We abbreviate $\calA := [0,1] \times \R$ (the motivation being that $[0,1] \times \R$ should be viewed as a parametrisation for the space of lines $\mathcal{A}(2,1)$).

\begin{definition}[{Fourier transform on $\calA$}]
	Given $\nu \in \mathcal{M}_{\C}(\mathcal{A})$ we define for $n\in \Z$ and $\rho\in\R$
	\begin{equation*}
	\calF(\nu)(n,\rho)=\int_0^1\int_\R e^{-2\pi i (n\theta+\rho r)} \, d\nu(r,\theta).
	\end{equation*}
\end{definition}

We consider the following Sobolev norms on $\calA$.
\begin{definition}[Sobolev norms on {$\calA$}]	
	For any $g\in C_c^\infty(\calA)$ and $s>- \tfrac{1}{2}$ we define
	\begin{align*}
	\|g\|_{\dot{H}^s(\calA)}^2 &= \sum_{n\in\Z}\int_\R |\calF(g)(n,\rho)|^2 |(n,\rho)|^{2s}\, d\rho.
	\end{align*}
\end{definition}
Observe that by Plancherel's identity for any $f,g\in C_c^\infty(\calA)$ and any $s > -\tfrac{1}{2}$,
\begin{align}
\int_0^1\int_\R f(\theta,x)\cdot& \overline{g(\theta,x)}\, dx\,d\theta\notag
= \sum_{n\in\Z}\int_\R \calF(f)(n,\rho)\cdot \overline{\calF(g)(n,\rho)}\, d\rho\notag\\
&\le  \sum_{n\in\Z}\left(\int_\R |\calF(f)(n,\rho)|^2|(n,\rho)|^{2s}\, d\rho\cdot \int_\R|{\calF(g)(n,\rho)}|^2|(n,\rho)|^{-2s}\, d\rho	\right)^{1/2}\notag\\
&\le\|f\|_{\dot{H}^s(\calA)}	\|g\|_{\dot{H}^{-s}(\calA)}.\label{eq:dua}\notag
\end{align}

The Sobolev smoothing property of the $X$-ray transform is classical, see e.g. Chapter II, Theorem 5.1 in \cite{Natterer}. The variant below is best suited for our purpose.
\begin{thm}[{\cite[Theorem 2.16]{MR4722034}}]\label{thm:Sob-smooth}
	For every $\chi\in C_c^\infty(\R^2)$ there exists $C_{\chi}>1$ such that if $f\in \calS(\R^2)$ and $s\in [-1/2, 1/2]$ then
	\begin{equation*}
	\|X(f\chi)\|_{\dot{H}^{s+1/2}(\calA)}\le C_{\chi} \|f\|_{{\dot{H}}^{s}(\R^2)}.
	\end{equation*}
\end{thm}
We need the following estimate comparing the Sobolev norms on $\calA$ and $\R^2$.
\begin{lemma}\label{lem:sob-norm-est}
	Let $\nu \in \mathcal{M}_{\C}(\mathcal{A})$. Then, for $-1/2<s\le 0$
	\begin{equation*}
	\|\nu\|_{\dot{H}^s(\calA)}\lesssim_s  \|\nu\| + \|\nu\|_{\dot{H}^s(\R^2)},
	\end{equation*}
	where $\| \nu \| = \nu(\R^{2})$.
	\end{lemma}
\begin{proof}
	Observe that for any $n\in\Z$ and $\rho\in\R$ we have $\calF(\nu)(n,\rho)=\hat{\nu}(n,\rho)$, where $\hat{\nu}$ denotes the usual Fourier transform on $\R^2$. Thus,
	\begin{align*}
	\|\nu\|_{\dot{H}^s(\calA)} &= \sum_{n\in\Z}\int_\R |\calF(\nu)(n,\rho)|^2 |(n,\rho)|^{2s}\, d\rho = \sum_{n\in\Z}\int_\R |\hat{\nu}(n,\rho)|^2 |(n,\rho)|^{2s}\, d\rho.
	\end{align*}
	Let $\varphi\in\calS(\R)$ be such that $\|\varphi\|_{L^1}\lesssim 1$, $\varphi(\theta)= 1$ for $\theta\in [0,1]$. Since $\spt \nu \subset [0,1] \times \R$, we have $d\nu(\theta,r) = \varphi(\theta)d\nu(\theta,r)$, which leads to
	\begin{displaymath} \hat{\nu}(n,\rho) = \int \widehat{\varphi}(n - \eta)\hat{\nu}(\eta,\rho) \, d\eta, \qquad (n,\rho) \in \Z \times \R. \end{displaymath}
	Using $\|\widehat{\varphi}\|_{L^\infty}\le \|\varphi\|_{L^1}\lesssim 1$, and the rapid decay of $\widehat{\varphi}$, we further get 
	\begin{displaymath} |\hat{\nu}(n,\rho)|^2 \lesssim \int |\widehat{\varphi}(n - \eta)| |\hat{\nu}(\eta,\rho)|^2 d\eta \lesssim \sum_{j \in \N} 2^{-3j} \int_{|\eta - n| \leq 2^{j}} |\hat{\nu}(\eta,\rho)|^{2} \, d\eta. \end{displaymath}
 It follows that
	\begin{align}
	\|\nu\|_{\dot{H}^s(\calA)}^2 & =\sum_{n\in\Z}\int_\R |\hat{\nu}(n,\rho)|^2 |(n,\rho)|^{2s}\, d\rho \notag\\
	&\label{form84} \lesssim \sum_{j \in \N} 2^{-3j} \sum_{n\in\Z}\int_\R\int_{|\eta - n| \leq 2^{j}}  |\hat{\nu}(\eta,\rho)|^2 |(n,\rho)|^{2s}d\eta\, d\rho.
	\end{align}
	Fix $j \in \N$. Observe that, since $s\le 0$, whenever $|n|\ge 2^{j + 1}$, or $|n| < 2^{j + 1}$ and $|\rho|\ge 2^{j}$,
	\begin{equation*}
	|(n,\rho)|^{2s}\lesssim |(\eta,\rho)|^{2s} \quad\text{for $|\eta - n| \leq 2^{j}$.}
	\end{equation*}
	Thus,
	\begin{multline*}
	\sum_{n\in\Z}\int_\R\int_{|\eta - n| \leq 2^{j}} |\hat{\nu}(\eta,\rho)|^2 |(n,\rho)|^{2s}d\eta\, d\rho\\
	\lesssim \int_{-2^{j}}^{2^{j}} \int_{-2^{j}}^{2^{j}} |\hat{\nu}(\eta,\rho)|^2 |\rho|^{2s}d\eta\, d\rho+
	\sum_{n \in \Z} \int_\R\int_{|\eta - n| \leq 2^{j}} |\hat{\nu}(\eta,\rho)|^2 |(\eta,\rho)|^{2s} d\eta\, d\rho\\
	\lesssim_s 2^{2j} \|\hat{\nu}\|_{L^\infty}^2+2^{j} \|\nu\|_{\dot{H}^s(\R^2)}^2
	\le 2^{2j} \|\nu\|^{2} + 2^{j}\|\nu\|_{\dot{H}^s(\R^2)}^2,
	\end{multline*}
	where in the second inequality we used (a) our assumption $s>-1/2$ to make sure that $\rho \mapsto |\rho|^{2s}$ is locally integrable, and (b) that $\{n \in \N : |\eta - n| \leq 2^{j}\} \lesssim 2^{j}$ for $n \in \Z$. Taking into the factor $2^{-3j}$ in \eqref{form84}, the terms $2^{2j} \|\nu\| + 2^{j}\|\nu\|_{\dot{H}^s(\R^2)}^2$ are summable over $j \in \N$. \end{proof}

\subsection{Sobolev estimates for $X$-ray measures} Our main tool in the proof of Theorem \ref{t2} is the following (local) Sobolev estimate for $X$-ray measures.
\begin{proposition}\label{prop4} Let $t \in (1,2]$, $s \in (2 - t,1]$, and $C > 0$. Assume that $(\mu,\{\sigma_x\})$ is a configuration such that $\mu$ is a $(t,C)$-Frostman measure supported on $B(1)$, and that each $\sigma_x$ is an $(s,C)$-Frostman measure. 
	
	Then for every $0<\ve < s/2$ and $\chi\in C_c^\infty(\R^2)$ we have
	\begin{equation*}
	\|X[\mu,\{\sigma_x\}] \cdot \chi\|_{\dot{H}^{s/2-\ve}(\R^2)}^2 \lesssim_{s,t,\ve,\chi} C^3
	\end{equation*}	
\end{proposition}
We prove Proposition \ref{prop4} in this subsection. First, we reduce matters to an energy estimate for the measure $\mu(\sigma_x)$.

\begin{lemma}
	Suppose that $(\mu,\{\sigma_x\})$ is a configuration.	For $\sigma\in (0,1)$ and $\chi\in C_c^\infty(\R^2)$,
	\begin{equation}\label{eq:goo}
	\|X[\mu,\{\sigma_x\}]\cdot \chi\|^2_{\dot{H}^{\sigma/2}(\R^2)}\lesssim_{\sigma,\chi} I_{\sigma+1}(\mu(\sigma_x)).
	\end{equation}
\end{lemma}
\begin{proof}
	Recall that $X[\mu,\{\sigma_x\}] = X^*(\mu(\sigma_x))$ as measures, see Proposition \ref{prop:adjRad}. It is easy to see from Definition \ref{def:configuration} that $\spt(\mu(\sigma_x))\subset [0,1]\times [-2,2]$.
		
	By duality and the Sobolev smoothing of $X$-ray transform, Theorem \ref{thm:Sob-smooth}, we have
	\begin{align*}
	\| X^*(\mu(\sigma_x))\cdot\chi\|_{\dot{H}^{\sigma/2}(\R^2)} &= \sup_{\|f\|_{\dot{H}^{-\sigma/2}(\R^2)}\le 1} \left| \int_{\R^2} f\chi\, d X^*(\mu(\sigma_x)) \right| \\
	&= \sup_{\|f\|_{\dot{H}^{-\sigma/2}(\R^2)}\le 1} \left| \int_{\calA(2,1)} X(f\chi)\, d\mu(\sigma_x) \right|\\
	&\le \sup_{\|f\|_{\dot{H}^{-\sigma/2}(\R^2)}\le 1} \|X(f\chi)\|_{\dot{H}^{-\sigma/2+1/2}(\calA)}\|\mu(\sigma_x)\|_{\dot{H}^{\sigma/2-1/2}(\calA)}\\
	&\lesssim_\chi \|\mu(\sigma_x)\|_{\dot{H}^{\sigma/2-1/2}(\calA)}.
	\end{align*}
(The $\sup$ runs over $f\in \calS(\R^2)$.) Together with Lemma \ref{lem:sob-norm-est} this gives
	\begin{equation*}
	\| X^*(\mu(\sigma_x)) \cdot \chi\|^2_{\dot{H}^{\sigma/2}(\R^2)}\lesssim_{\sigma,\chi} \|\mu(\sigma_x)\|^{2} +\|\mu(\sigma_x)\|_{\dot{H}^{-\sigma/2+1/2}(\R^2)}^2\overset{\eqref{eq:Riesz}}{\sim_\sigma} I_{\sigma+1}(\mu(\sigma_x)),
	\end{equation*}
	using also $\mu(\sigma_{x})(\R^2)^2\lesssim_{\sigma} I_{\sigma + 1}(\mu(\sigma_{x}))$ by $\spt(\mu(\sigma_x))\subset [0,1]\times [-2,2]$.	\end{proof}

In light of \eqref{eq:goo}, Proposition \ref{prop4} is a consequence of the following result.
\begin{proposition}\label{prop:energyest}
	Suppose that $t\in [1,2]$ and $s\in [2-t,1]$. If $(\mu,\{\sigma_x\})$ is a configuration such that $\mu$ is $(t,C)$-Frostman, and each $\sigma_x$ is $(s,C)$-Frostman, then for every $\ve>0$
	\begin{equation}\label{eq:desiredest}
	I_{s+1-\ve}(\mu(\sigma_x))\lesssim_{s,t,\ve} C^3.
	\end{equation}
\end{proposition}

The proof of Proposition \ref{prop:energyest} is based on \cite[Proposition 6.18]{Orponen2024Jan}, restated below as Proposition \ref{lem:des}. 
\begin{definition}[$\delta$-measures]
	A collection of non-negative weights $\boldsymbol{\mu}=\{\mu(p)\}_{p\in\calD_\delta}$ with $\|\boldsymbol{\mu}\|=\sum_{p\in\calD_\delta}\mu(p)\le 1$ is called a $\delta$-measure. We say that a $\delta$-measure $\boldsymbol{\mu}$ is a $(\delta,s,C)$-measure if $\boldsymbol{\mu}(Q)\le C\ell(Q)^s$ for every $Q\in\calD_\Delta$ with $\delta\le\Delta\le 1$. 
\end{definition}

For every $T\in\calT^\delta$ let $p_T\in\calD_\delta$ denote the unique $\delta$-cube such that $T=\cup\mathbf{D}(p_T)$.
\begin{proposition}\label{lem:des}
	Suppose $t\in [1,2]$ and $s\in [2-t,1]$. Assume that $\boldsymbol{\mu}$ is a $(\delta,t,C_\mu)$-measure, and for every $T\in\calT^\delta$ there is a $(\delta,s,C_\sigma)$-measure $\boldsymbol{\sigma}_T$ supported on $\calD_\delta(T)$. Then, for every $\ve>0$
	\begin{equation}\label{eq:desiredest2}
	\int \bigg(\sum_{T\in\calT^\delta}\mu(p_T)\sum_{q\in \calD_\delta}\big(\sigma_T(q)\cdot \delta^{-2}\1_q\big)\bigg)^2\, dx\lesssim_{\ve} C_\mu C_\sigma^2 \delta^{s-1-\ve}.
	\end{equation}
\end{proposition}

\begin{proof} This is otherwise \cite[Proposition 6.18]{Orponen2024Jan}, except that \cite[Proposition 6.18]{Orponen2024Jan} assumed $s + t \leq 2$, and the right hand side of \eqref{eq:desiredest2} is replaced by
\begin{equation}\label{form82} \lesssim_{\epsilon} C_\mu C_\sigma^2 \delta^{2s + t - 3 -\ve}. \end{equation}
To derive Proposition \ref{lem:des} as stated, write $\tau := 2 - s \in [0,t]$, and note that $s + \tau \leq 2$. Further, $\boldsymbol{\mu}$ is automatically a $(\delta,\tau,C_{\mu})$-measure, since $\tau \leq t$. So, we may apply \cite[Proposition 6.18]{Orponen2024Jan} with parameters $\tau,s$. Then \eqref{form82} becomes $\lesssim_{\epsilon} C_{\mu}C_{\sigma}^{2}\delta^{s - 1 - \epsilon}$. \end{proof}

We then prove Proposition \ref{prop:energyest}. 
\begin{proof}[Proof of Proposition \ref{prop:energyest}] Recall that $\mu(\sigma_{x})$ is a measure on $[0,1]\times [-2,2]\subset [0,1]\times\R$, which can be identified with the space of lines $\calA(2,1)$ using the 2-to-1 map $(\theta,r)\mapsto\pi_{\theta}^{-1}\{r\}$. Set $\bF(\theta,r)=\pi_{\theta}^{-1}\{r\}$. At this point it will be more convenient for us to work with a different parametrization of $\calA(2,1)$ -- one that preserves the incidence relation. 
	
	Without loss of generality assume that $\spt(\sigma_x)\subset [1/8, 3/8]$ for each $x$, so that $\mu(\sigma_{x})$ is supported in $[1/8, 3/8]\times [-2,2]$. This can be achieved by decomposing the configuration $(\mu,\{\sigma_{x}\})$ into two configurations, each which satisfy this property up to a rotation, and then proving the desired estimate for each of them separately. Set $S=\spt(\mu(\sigma_x))\subset [1/8, 3/8]\times [-2,2]$.
	
	Note that $\bF|_{S}$ is bilipschitz onto its image, with absolute Lipschitz constant. Recalling that the duality map $\mathbf{D}:\R^2\to\calA(2,1)$ from Definition \ref{def:pointLineDuality} is locally bilipschitz, we get that for $\bar{S}\coloneqq \mathbf{D}^{-1}(\bF(S))$ the map $\mathbf{D}|_{\bar{S}}$ is bilipschitz onto $\bF(S)$. Let $\mathbf{D}^* := \mathbf{D}^{-1}|_{\mathbf{F}(S)}$.
	
	Consider the measure 
	\begin{equation}\label{form83}
	\bar{\mu}(\bar{\sigma}_x)\coloneqq \mathbf{D}^{\ast}_{\sharp}\mathbf{F}_{\sharp}(\mu(\sigma_{x})) = (\mathbf{D}^{\ast} \circ \mathbf{F})_{\sharp}(\mu(\sigma_{x})),	\end{equation}
	supported on $\bar{S}$. Since $\mathbf{D}^{\ast} \circ \mathbf{F}: S\to \bar{S}$ is bilipschitz, $I_{s+1-\ve}(\mu(\sigma_x))\lesssim I_{s+1-\ve}(\bar{\mu}(\bar{\sigma}_x))$. Thus, to get \eqref{eq:desiredest} it suffices to prove 
	\begin{equation}\label{eq:aa}
	I_{s+1-\ve}(\bar{\mu}(\bar{\sigma}_x))\lesssim_{s,t,\ve}C^3.
	\end{equation}
	
	By \eqref{eq:Riesz} we see that \eqref{eq:aa} follows from the $L^2$-estimate
	\begin{equation}\label{eq:des2}
	\|\bar{\mu}(\bar{\sigma}_x)\ast\psi_\delta\|^2_{L^2}\lesssim_{s,t,\ve} C^3\, \delta^{(s+1-\ve)-2} = C^3\, \delta^{s-1-\ve},
	\end{equation}
	where $\psi_\delta(x)=\delta^{-2}\psi(x/\delta)$ is a standard approximate identity with $\int\psi=1,$ smooth and radially decreasing, $\1_{B(1/2)}\le\psi\lesssim \1_{B(1)}$. Our goal now is \eqref{eq:des2}.
	
	For every $T\in\calT^\delta$ with $\mu(p_{T}) > 0$, and every $q \in \mathcal{D}_{\delta}(\R^{2})$, we define
	\begin{equation}\label{eq:sigmdef}
	\sigma_T(q) \coloneqq \frac{1}{{\mu}(p_T)}\int_{p_T}\int_{0}^{1}\1_{3q}(\mathbf{D}^*\circ\bF(\theta,\pi_{\theta}(x)))\, d\sigma_x(\theta)\,d\mu(x).
	\end{equation}
	The Frostman properties of $\mu$ and $\sigma_x$ imply that $\boldsymbol{\mu}=\{{\mu}(p)\}_{p\in\calD_\delta}$ is a $(\delta,t,C_{\mu}^{\prime})$-measure, and each $\boldsymbol{\sigma}_T = \{\sigma_T(q)\}_{q\in\calD_\delta}$ is a $(\delta,s,C_{\sigma}^{\prime})$-measure, with $C^{\prime}_{\sigma}\sim C$, $C^{\prime}_{\mu} \sim C_{\mu}$. 
	
	We claim that $\spt (\boldsymbol{\sigma}_{T}) \subset \mathcal{D}_{\delta}(10T)$. To see this, suppose $q \in \mathcal{D}_{\delta}(\R^{2})$ with $\sigma_T(q) \neq 0$, so in particular $\mathbf{D}^*\circ\bF(\theta,\pi_{\theta}(x))\in 3q$ for some  $x\in p_T$ . This is equivalent to 
	\begin{equation*}
	\bF(\theta,\pi_{\theta}(x))=\pi_{\theta}^{-1}\{\pi_\theta(x)\}\in\mathbf{D}(3q).
	\end{equation*}
	Since $x\in \pi_{\theta}^{-1}\{\pi_\theta(x)\}$, we get $x\in \cup\mathbf{D}(3q)$, and so $p_T\cap  \cup\mathbf{D}(3q)\neq\varnothing$. Since $\mathbf{D}$ preserves incidences, this implies $T\cap 3q\neq\varnothing$. Thus, $q \subset 10T$, completing the proof of the claim.
	
	Fix $y\in \spt(\bar{\mu}(\bar{\sigma}_{x}) \ast \psi_{\delta}) \subset 2\bar{S}$, and let $q\in\calD_\delta(2\bar{S})$ be the unique $\delta$-cube containing $y$. Since $\spt(\psi_\delta(\cdot-y))\subset 3q$ we have
	\begin{align*}
	(\bar{\mu}(\bar{\sigma}_x)\ast\psi_\delta)(y) &= \int \psi_\delta(z-y)\, d\bar{\mu}(\bar{\sigma}_x)(z) \lesssim \delta^{-2} \int \1_{3q}(z)\, d\bar{\mu}(\bar{\sigma}_x)(z)\\
	&\stackrel{\eqref{form83}}{=} \delta^{-2}\int_{[0,1]^2}\int_0^1 \1_{3q}(\mathbf{D}^*\circ\bF(\theta,\pi_{\theta}(x)))\, d\sigma_x(\theta)\,d\mu(x) \\
	&\le\delta^{-2}\sum_{T\in\calT^\delta} \int_{p_T}\int_{0}^{1}\1_{3q}(\mathbf{D}^*\circ\bF(\theta,\pi_{\theta}(x)))\, d\sigma_x(\theta)\,d\mu(x)\\
	&\stackrel{\eqref{eq:sigmdef}}{=} \delta^{-2}\sum_{T\in\calT^\delta}{\mu}(p_T)\sigma_{T}({q}).
	\end{align*}
	This estimate can be rewritten as the pointwise inequality
	\begin{equation*}
	\bar{\mu}(\bar{\sigma}_x)\ast\psi_\delta\lesssim \sum_{T\in\calT^\delta}\mu(p_T)\sum_{q\in\calD_\delta}\big(\sigma_T(q)\cdot\delta^{-2}\1_q\big).
	\end{equation*}	
	Since $\boldsymbol{\mu}$ and $\boldsymbol{\sigma}_T$ satisfy the assumptions of Proposition \ref{lem:des}, the estimate above together with \eqref{eq:desiredest2} gives \eqref{eq:des2}. (To be precise, the support of $\boldsymbol{\sigma}_{T}$ is potentially somewhat outside $[0,1]^{2}$, but nonetheless in some ball of absolute constant radius, and Proposition \ref{lem:des} has an easy extension to this situation.) \end{proof}

\subsection{Proof of Theorem \ref{t2}} In this subsection we use Proposition \ref{prop4} to prove Theorem \ref{t2}. We need the following proposition relating Sobolev estimates to Hausdorff content.

\begin{proposition}\label{prop3} Let $d \in \N$, $\epsilon_{1},\epsilon_{2},\epsilon_{3} > 0$, $s,t \in (0,d]$ such that 
\begin{displaymath} s + t > d + 2(\epsilon_{1} + \epsilon_{2} + \epsilon_{3}). \end{displaymath}
Let $\varphi \in C^{\infty}_{c}(\R^{d})$ be non-negative with $\int \varphi = 1$. Then, there exists $\Delta_{0} = \Delta_{0}(d,\epsilon_{1},\epsilon_{2},\epsilon_{3},\varphi,s,t) > 0$ such that the following holds for all $\Delta \in (0,\Delta_{0}]$. Let $X$ be a Radon measure on $\R^{d}$ such that $\|X\|_{\dot{H}^{s/2}} \leq \Delta^{-\epsilon_{3}}$. Then,
\begin{displaymath} \mathcal{H}^{t}_{\infty}(X(\Delta,{\epsilon_{1}}) \, \setminus \, \spt X) \leq \Delta^{\epsilon_{2}}, \end{displaymath}
where $X(\Delta, \epsilon) = \{x \in B(1) : X_{\Delta}(x) \geq \Delta^{\epsilon}\}$, and $X_{\Delta} = X \ast \varphi_{\Delta}$.  \end{proposition}

\begin{proof} We make a counter assumption: $\mathcal{H}^{t}_{\infty}(X(\Delta,{\epsilon_{1}}) \, \setminus \, \spt X) > \Delta^{\epsilon_{2}}$. Then Frostman's lemma produces a measure $\nu$ with $\spt \nu \subset X(\Delta,{\epsilon_{1}}) \, \setminus \, \spt X \subset B(1)$, and satisfying
\begin{displaymath} \nu(B(1))\sim_d \cH^{t}_{\infty}(X(\Delta,\epsilon_{1}) \setminus \spt X) \geq \Delta^{\epsilon_{2}} \quad \text{and} \quad \nu(B(x,r)) \le r^{t} \text{ for all } x \in \R^{d}, \, r > 0. \end{displaymath}
 Since $s + t > d + 2(\epsilon_{1} + \epsilon_{2} + \epsilon_{3})$, there exists $\epsilon_{4} > \epsilon_{1} + \epsilon_{2} + \epsilon_{3}$ such that $t > d -s + 2\epsilon_{4}$. Then, using the higher dimensional version of \eqref{eq:Riesz} (see Theorem 3.10 in \cite{MR3617376})
\begin{equation}\label{form52} \int_{\R^{d}} |\hat{\nu}(\xi)|^{2}|\xi|^{- s + 2\epsilon_{4}} \, d\xi \lesssim I_{d - s + 2\epsilon_{4}}(\nu) \le \frac{d-s+2 \epsilon_{4}}{t-(d-s+2 \epsilon_{4})} \nu(B(1)). \end{equation}
The implicit constants above, depend on $d, s, \epsilon_{4}$ or equivalently, $d,s,t,\epsilon_{1},\epsilon_{2}, \epsilon_{3}$.  We claim that 
$$
0 = \int \widehat{X}(\xi) \bar{\hat{\nu}}(\xi) d \xi.
$$
Indeed, since $\spt \nu , \spt X$ are disjoint compact sets, $\rho \vcentcolon = \dist(\spt \nu, \spt X) >0$. Consequently, whenever $\delta < \frac{\rho}{2 \diam \spt \varphi}$,
\begin{align*}
0 &= \int X_{\delta} \nu_{\delta} = \int \widehat{X_{\delta}}(\xi) \widehat{\nu_{\delta}}(\xi) d \xi = \int \widehat{\varphi}(\delta \xi)^{2} \widehat{X}(\xi) \hat{\nu}(\xi) d \xi \\
& = \int \widehat{X}(\xi) \hat{\nu}(\xi) d \xi + \int(1 - \widehat{\varphi}(\delta \xi)^{2}) \widehat{X}(\xi) \hat{\nu}(\xi) d \xi = \vcentcolon A_{1} + A_{2}.
\end{align*}
In particular, $A_{2} = - A_{1}$ so the claim follows from showing $|A_{2}| = o(\delta)$. To this end, note $|1 - \hat{\varphi}(\delta \xi)| = |\hat{\varphi}(0) - \hat{\varphi} (\delta \xi)| \lesssim \min \{ |\delta \xi|,1\} \le \min\{ |\delta \xi|,1\}^{\epsilon_{4}}$ which in turn implies $(1 - \varphi(\delta \xi)^{2}) \lesssim |\delta \xi|^{\epsilon_{4}}$.
Therefore, 
\begin{align*}
|A_{2}| &\lesssim \delta^{\epsilon_{4}}  \left| \int_{\R^{d}} |\xi|^{\epsilon_{4}} \widehat{X}(\xi) \bar{\hat{\nu}}(\xi) d \xi \right| \lesssim \delta^{\epsilon_{4}}  \Big( \int_{\R^{d}} |\widehat{X}(\xi)|^{2}|\xi|^{s} \, d\xi \Big)^{1/2} \Big( \int_{\R^{d}} |\hat{\nu}(\xi)|^{2}|\xi|^{-s + 2\epsilon_{4}} \, d\xi \Big)^{1/2}  \\
& = \delta^{\epsilon_{4}} \|X\|_{\dot{H}^{s/2}} \|\nu\|_{\dot{H}^{-s/2+\epsilon_{4}}} = o(\delta),
\end{align*}
completing the claim. On the other hand,
\begin{displaymath} 0 = \left| \int \widehat{X}(\xi)\bar{\hat{\nu}}(\xi) \, d\xi \right| \geq \left| \int \widehat{\varphi}(\Delta \xi)\widehat{X}(\xi)\bar{\hat{\nu}}(\xi) \, d\xi \right| - \left| \int [1 - \widehat{\varphi}(\Delta \xi)]\widehat{X}(\xi)\bar{\hat{\nu}}(\xi) \right| =: I_{1} - I_{2}. \end{displaymath}
Here,
\begin{displaymath} I_{1} = \int (X \ast \varphi_{\Delta})(x) \, d\nu(x) \geq \Delta^{\epsilon_{1} + \epsilon_{2}}, \end{displaymath} 
because $(X \ast \varphi_{\Delta})(x) = X_{\Delta}(x) \geq \Delta^{\epsilon_{1}}$ for $x \in \spt \nu \subset X(\Delta,\epsilon_{1})$. $I_{2}$ can be estimated similarly to $|A_{2}|$, to conclude 
$$
I_{2} \lesssim \Delta^{\epsilon_{4}} \|X\|_{\dot{H}^{s}} \|\nu\|_{\dot{H}^{-s+\epsilon_{4}}} \lesssim \Delta^{\epsilon_{4} - \epsilon_{3} }  \cH^{t}_{\infty}(X(\Delta, \epsilon_{1}) \setminus \spt X),
$$
since $\|X\|_{\dot{H}^{s/2}} \lesssim \Delta^{-\epsilon_{3}}$ by hypothesis and $\|\nu\|_{\dot{H}^{- s/2 + \epsilon_{4}}} \lesssim \cH^{t}_{\infty}(X(\Delta, \epsilon_{1}) \setminus \spt X)$.
This yields the desired contradiction $I_{1} - I_{2} \geq \Delta^{\epsilon_{1} + \epsilon_{2}} - O(\Delta^{\epsilon_{4} - \epsilon_{3}}) > 0$ for all $\Delta > 0$ sufficiently small since $\epsilon_{4} > \epsilon_{1} + \epsilon_{2} + \epsilon_{3}$.\end{proof}

\begin{proof}[Proof of Theorem \ref{t2}] Write $\sigma := \sigma(s,t) := \tfrac{1}{2}[(2 - t) + s] \in (2 - t,s)$. We claim that the conclusion of Theorem \ref{t2} is valid for any
\begin{displaymath} 0 < \eta < \frac{\sigma + t - 2}{8}, \end{displaymath} 
provided that $\Delta > 0$ is small enough. Fix $\eta > 0$ as above and let $\chi\in C_c^\infty(\R^2)$ satisfy $\1_{B(1)}\le \chi \le \1_{B(2)}$. Then, under the hypotheses of Theorem \ref{t2}, by Proposition \ref{prop4}, and assuming $\Delta > 0$ small enough,
\begin{displaymath} X := X[\mu|_{G_{1}},\{\mathcal{L}_{x}\}]\chi \in \dot{H}^{\sigma/2} \quad \text{ with } \quad \|X\|_{\dot{H}^{\sigma/2}} \leq \Delta^{-4\eta}. \end{displaymath}
Now, since $\sigma + t > 2 + 8\eta$, we may infer from Proposition \ref{prop3} (with $\epsilon_{1} = \eta$, $\epsilon_{2} = 3\eta$ and $\epsilon_{3} = 4\eta$) that
\begin{equation}\label{form55} \mathcal{H}^{t}_{\infty}(X(\Delta, \eta) \, \setminus \, \spt X) \leq \Delta^{3\eta}. \end{equation}
Note that the hypothesis \nref{H2} can be rephrased as $G_{2} \subset X(\Delta, \eta)$. On the other hand, \nref{H1} combined with the $(t,\Delta^{-\eta})$-Frostman property of $\mu$ implies $\mathcal{H}^{t}_{\infty}(G_{2}) \gtrsim \Delta^{2\eta}$. Combining this information with \eqref{form55}, we deduce that 
\begin{equation}\label{form56} G_{2} \cap (\spt X) \neq \emptyset. \end{equation}
This implies $\inf \{\dist(y,L_{x}) : x \in G_{1} \text{ and } y \in G_{2}\} = 0$, because
\begin{displaymath} \spt X \subset \overline{\bigcup_{x \in \bar{G}_{1}} L_{x}} \end{displaymath}
by Remark \ref{rem5}. This completes the proof of Theorem \ref{t2}. \end{proof}


\section{A $\delta$-discretised Marstrand slicing theorem}

As a technical tool in later sections, we will need a $\delta$-discretised version of \cite[Theorem 2.4]{MR3145914}, stated in Proposition \ref{prop5} below. Fortunately, a $\delta$-discretised version of a stronger result (concerning radial projections) was recently proven in \cite[Theorem 3.1]{MR4722034}. We follow that argument closely, but nonetheless give all the details for the reader's convenience.

\begin{proposition}\label{prop5} For every $t \in (1,2]$, $s \in (2 - t,1]$, and $\eta > 0$, there exist $\delta_{0} = \delta_{0}(\eta,s,t) > 0$ and $\epsilon = \epsilon(\eta) > 0$ such that the following holds for all $\delta \in 2^{-\N} \cap (0,\delta_{0}]$.

Let $\mathcal{P} \subset \mathcal{D}_{\delta}$ be a non-empty $(\delta,t,\delta^{-\epsilon})$-set, and let $\Theta \subset S^{1}$ be a non-empty $\delta$-separated $(\delta,s,\delta^{-\epsilon})$-set. Then, there exists a subset $\bar{\Theta} \subset \Theta$ with $|\bar{\Theta}| \geq (1 - \delta^{\epsilon})|\Theta|$, and for every $\theta \in \bar{\Theta}$ a subset $\mathcal{P}_{\theta} \subset \mathcal{P}$ with $|\mathcal{P}_{\theta}| \geq (1 - \delta^{\epsilon})|\mathcal{P}|$ satisfying the following "rectangular" Frostman condition: if $R \subset \R^{2}$ is a $(\delta \times \Delta)$-rectangle with $\delta \leq \Delta \leq 1$, and the $\Delta$-side parallel to $\theta$, then 
\begin{equation}\label{rectangularFrostman} |\mathcal{P}_{\theta} \cap R| \leq \delta^{1 - \eta}\Delta^{t - 1}|\mathcal{P}|. \end{equation}
\end{proposition} 

\begin{remark} One may take $\epsilon = c\eta$ for a sufficiently small absolute constant $c > 0$. \end{remark}

To prove Proposition \ref{prop5}, we need a quantitative Furstenberg set estimate \cite[Theorem 4.9]{MR4722034}, stated as Theorem \ref{t:orponen}. In the statement, an \emph{ordinary $\delta$-tube} is any rectangle of dimensions $\delta \times 1$. A family $\mathcal{T}$ of ordinary $\delta$-tubes is called a \emph{Katz-Tao $(\delta,s,C)$-set} if for all $\delta \leq r \leq 1$, an arbitrary rectangle of dimensions $(r \times 2)$ contains at most $C(r/\delta)^{s}$ elements of $\mathcal{T}$. If the constant $C$ is absolute, a Katz-Tao $(\delta,s,C)$-set is called a Katz-Tao $(\delta,s)$-set. This definition is, in particular, used in \cite{fu2022incidence}, whose results we plan to apply in a moment. 

Similarly, a family $\mathcal{T}$ of ordinary $\delta$-tubes is called a \emph{$(\delta,s,C)$-set} if every rectangle of dimensions $(r \times 2)$, $\delta \leq r \leq 1$, contains at most $Cr^{s}|\mathcal{T}|$ elements of $\mathcal{T}$. These definitions are the analogues of Definition \ref{def:deltaSTubes} for ordinary tubes.

\begin{thm}\label{t:orponen} Let $t \in (1,2]$, $s \in (2 - t,1]$, $C \geq 1$, and $\sigma < s$. Fix $\delta \in 2^{-\N}$. Assume that $\mu$ is an $(s,C)$-Frostman measure on $B(1) \subset \R^{2}$. For every $p \in \mathcal{D}_{\delta}(\spt \mu)$, let $\mathcal{T}_{p}$ be a $(\delta,s,C)$-set of ordinary $\delta$-tubes such that $T \cap p \neq \emptyset$ for all $T \in \mathcal{T}_{p}$. Then, $\mathcal{T}$ contains a Katz-Tao $(\delta,\sigma + 1)$-set $\mathcal{T}'$ of cardinality $|\mathcal{T}'| \gtrsim_{\sigma,s,t} \mu(\R^{2})\delta^{-(\sigma + 1)}/C^{3}$. \end{thm}

In the original formulation \cite[Theorem 4.9]{MR4722034} the elements of $\mathcal{T}_{p}$ are assumed to be dyadic $\delta$-tubes, but the two variants of the theorem are easily seen to be equivalent. 

We also need Fu and Ren's incidence bound \cite[Theorem 1.5]{fu2022incidence}:

\begin{thm}\label{t:fuRen} Let $\epsilon > 0$, and $s,t \in (0,2)$ with $s + t \leq 3$. Then, there exists $\delta_{0} = \delta_{0}(\epsilon,s,t) > 0$ such that the following holds for all $\delta \in (0,\delta_{0}]$. Assume that $\mathcal{P} \subset \mathcal{D}_{\delta}([0,10)^{2})$ is a Katz-Tao $(\delta,t,C_{\mathcal{P}})$-set, and $\mathcal{T}$ is a Katz-Tao $(\delta,s,C_{\mathcal{T}})$-set of ordinary $\delta$-tubes. Then,
\begin{displaymath} |\mathcal{I}(\mathcal{P},\mathcal{T})| \stackrel{\mathrm{def.}}{=} |\{(p,T) \in \mathcal{P} \times \mathcal{T} : p \cap T \neq \emptyset\}| \leq \delta^{-1/2 - \epsilon}(C_{\mathcal{P}}C_{\mathcal{T}})^{1/2}|P|^{1/2}|\mathcal{T}|^{1/2}. \end{displaymath} \end{thm} 

We are then equipped to prove Proposition \ref{prop5}.

\begin{proof}[Proof of Proposition \ref{prop5}] We claim that the statement holds if $\epsilon = c\eta$ for a sufficiently small absolute constant $c \in (0,\tfrac{1}{2}]$, determined on the last line of the proof. We may assume that $t < 2$, since in the case $t = 2$ we may use $|\mathcal{P}| \geq \delta^{\epsilon - 2}$ to deduce the trivial upper bound 
\begin{displaymath} |\mathcal{P} \cap R| \lesssim \delta^{-1} \Delta \leq \delta^{1 - \epsilon}\Delta^{2 - 1}|\mathcal{P}| \leq \delta^{1 - \eta}\Delta^{t - 1}|\mathcal{P}|   \end{displaymath}
for all $(\delta \times \Delta)$-rectangles $R \subset \R^{2}$ with $\delta \leq \Delta \leq 1$, verifying \eqref{rectangularFrostman} when $t=2$. 

Assuming $t \in (1,2)$, we make a counter assumption: there exists a subset $\bar{\Theta} \subset \Theta$ of cardinality $|\bar{\Theta}| \geq \delta^{\epsilon}|\Theta|$, and for each $\theta \in \overline{\Theta}$, a "bad" subset $\mathcal{B}_{\theta} \subset \mathcal{P}$ with $|\mathcal{B}_{\theta}| \geq \delta^{\epsilon}|\mathcal{P}|$, all squares of which fail the rectangular Frostman condition \eqref{rectangularFrostman} for some $(\delta \times \Delta)$-rectangle $R \subset \R^{2}$ parallel to $\theta$ -- which may of course depend on the square.

For the duration of this proof, the notation $A \lessapprox B$ will mean that $A \leq C_{s,t}\delta^{-C\epsilon}B$, where $C_{s,t} > 0$ may depend on $s,t$, but $C > 0$ is absolute.

By pigeonholing, reducing both $\bar{\Theta}$ and $\mathcal{B}_{\theta}$ somewhat, the (longer) side-length $\Delta$ of these rectangles can roughly be assumed to be independent of both $\theta$ and $p \in \mathcal{B}_{\theta}$. More precisely, the following objects can be found (we recycle the notation $\bar{\Theta}$ and $\mathcal{B}_{\theta}$):
\begin{enumerate}
\item A fixed number $\Delta \in 2^{-\N} \cap [\delta,1]$.
\item A subset $\bar{\Theta} \subset \Theta$ with $|\bar{\Theta}| \gtrapprox |\Theta|$.
\item For each $\theta \in \bar{\Theta}$ a family of disjoint $(\delta \times \Delta)$-rectangles $\mathcal{R}_{\theta}$ parallel to $\theta$ which are heavy in the sense $|\mathcal{P} \cap R| \geq \delta^{1 - \eta}\Delta^{t - 1}|\mathcal{P}|$, $R \in \mathcal{R}_{\theta}$, and cover a large part of $\mathcal{P}$:
\begin{displaymath} |\{p \in \mathcal{P} : p \cap R \neq \emptyset \text{ for some } R \in \mathcal{R}_{\theta}\}| \gtrapprox |\mathcal{P}|. \end{displaymath}
\end{enumerate}
Since $|\mathcal{P}| \geq \delta^{\epsilon - t}$ by the $(\delta,t,\delta^{-\epsilon})$-set hypothesis, we note that 
\begin{displaymath} \Delta/\delta \gtrsim |\mathcal{P} \cap R| \geq \delta^{1 - t - \eta + \epsilon}\Delta^{t - 1} \geq \delta^{1 - t - \eta/2}\Delta^{t - 1}, \end{displaymath}
which can be rearranged to $\delta/\Delta \lesssim \delta^{\eta/(4 - 2t)}$ (since $t < 2$). In particular, we may assume that $\delta/\Delta$ is arbitrarily small by choosing $\delta > 0$ small enough in terms of $\eta,t$. 

Inspired by (3), we define $\mathcal{B}_{\theta} := \{p \in \mathcal{P} : p \cap R \neq \emptyset \text{ for some } R \in \mathcal{R}_{\theta}\}$. Observe that
\begin{displaymath} \sum_{p \in \mathcal{P}} |\{\theta \in \bar{\Theta} : p \in \mathcal{B}_{\theta}\}| = \sum_{\theta \in \bar{\Theta}} |\mathcal{B}_{\theta}| \gtrapprox |\mathcal{P}||\Theta|. \end{displaymath}
 Consequently, there exists a fixed set $\mathcal{B} \subset \mathcal{P}$ with $|\mathcal{B}| \gtrapprox |\mathcal{P}|$ with the property
\begin{equation}\label{form37} |\{\theta \in \bar{\Theta} : p \in \mathcal{B}_{\theta}\}| \gtrapprox |\Theta|, \qquad p \in \mathcal{B}. \end{equation}
Write $\Theta_{p} := \{\theta \in \bar{\Theta} : p \in \mathcal{B}_{\theta}\}$ for $p \in \mathcal{B}$.

We next claim that there exists a distinguished square $Q \in \mathcal{D}_{\Delta}(\mathcal{P})$ with the property
\begin{equation}\label{form38} |\mathcal{B} \cap Q| \gtrapprox |\mathcal{P} \cap 10Q| > 0. \end{equation}
Indeed, if this failed for all $Q \in \mathcal{D}_{\Delta}$, then the bounded overlap of the squares $10Q$ would contradict $|\mathcal{B}| \gtrapprox |\mathcal{P}|$. We fix a square $Q \in \mathcal{D}_{\Delta}(\mathcal{P})$ satisfying \eqref{form38} for the remaining proof.

Fix $p \in \mathcal{B} \cap Q$, and recall that $|\Theta_{p}| \gtrapprox |\Theta|$ by \eqref{form37}. For every $\theta \in \Theta_{p}$, we have $p \in \mathcal{B}_{\theta}$ by definition. This means that there exists a $(\delta \times \Delta)$-rectangle $R = R(p,\theta)$ parallel to $\theta$ which intersects $p$, and satisfies $|\mathcal{P} \cap R| \geq \delta^{1 - \eta}\Delta^{t - 1}|\mathcal{P}|$. \emph{A fortiori},
\begin{equation}\label{form39} |(\mathcal{P} \cap 10Q) \cap R| \geq \delta^{1 - \eta}\Delta^{t - 1}|\mathcal{P}|, \end{equation}
since $R \subset B(p,2\Delta) \subset 10Q$. Let 
\begin{displaymath} \mathcal{R}_{p} := \{R(p,\theta) : \theta \in \Theta_{p}\}, \qquad p \in \mathcal{B} \cap Q. \end{displaymath}
(The notation $\mathcal{R}_{\theta}$ will no longer appear to cause confusion.) Since $|\Theta_{p}| \gtrapprox |\Theta|$, and $\Theta$ is a non-empty $(\delta,s,\delta^{-\epsilon})$-set, one could informally say that $\mathcal{R}_{p}$ is a non-empty $(\delta,s,\delta^{-\epsilon})$-set of $(\delta \times \Delta)$-rectangles incident to $p$. To make this more formal, we rescale by $\sim \Delta^{-1}$. Let $S_{10Q}$ be the $\Delta^{-1}$-rescaling which sends $10Q$ to $[0,10)^{2}$. Let 
\begin{equation}\label{form41} \mathcal{Q} := S_{10Q}(\mathcal{P} \cap 10Q) \quad \text{and} \quad \overline{\mathcal{B}} := S_{10Q}(\mathcal{B} \cap Q). \end{equation}
Thus, $\mathcal{Q}$ and $\overline{\mathcal{B}}$ consist of $\delta/\Delta$-squares contained in $[0,10)^{2}$.

Now, for $q = S_{10Q}(p) \in \overline{\mathcal{B}}$, consider the rescaled family of rectangles $S_{10Q}(\mathcal{R}_{p})$. This family consists of $((\delta/\Delta) \times 1)$-tubes intersecting $q$ and contained in $[0,1)^{2}$. The homothety $S_{10Q}$ preserves directions: therefore the $(\delta/\Delta)$-sides of the elements in $S_{10Q}(\mathcal{R}_{p})$ are still parallel to the elements in the $(\delta,s,\delta^{-\epsilon})$-set $\Theta_{p}$.

The $\delta$-separation of $\Theta_{p} \subset \Theta$ is no longer natural at the relevant scale $\delta/\Delta$, so we pass to a subset: applying \cite[Proposition A.1]{FaO}, we extract a non-empty $(\delta/\Delta,s,\delta^{-\epsilon})$-set $\bar{\Theta}_{p} \subset \Theta_{p}$ consisting of $(\delta/\Delta)$-separated elements. Then, for each $\theta \in \bar{\Theta}_{p}$, we select a tube in $S_{10Q}(\mathcal{R}_{p})$ with slope $\theta$, and we denote the ensuing collection $\mathcal{T}_{q}$ (recall: $q = S_{10Q}(p))$. Summary: for each $q \in \overline{\mathcal{B}}$, the family $\mathcal{T}_{q}$ is a non-empty $(\delta/\Delta,s,\delta^{-\epsilon})$-set of ordinary $(\delta/\Delta)$-tubes intersecting $q$.

We aim to apply Theorem \ref{t:orponen} to the set $\overline{\mathcal{B}}$ and the families $\mathcal{T}_{q}$, at scale $\delta/\Delta$. To do this, we need to introduce an appropriate measure $\mu$ associated to $\overline{\mathcal{B}}$. Recall that $\mathcal{P} \subset \mathcal{D}_{\delta}$ is a $(\delta,t,\delta^{-\epsilon})$-set, that is,
\begin{equation}\label{form40} |\mathcal{P} \cap B(x,r)| \lesssim \delta^{-\epsilon}r^{t}|\mathcal{P}|, \qquad x \in \R^{2}, \, r \geq \delta. \end{equation}
Let $P \subset \mathcal{P}$ be a set containing one point from each square $p \in \mathcal{P}$, and let $\mu_{0} := |P|^{-1}\mathcal{H}^{0}|_{P}$. Then \eqref{form40} shows that $\mu_{0}(B(x,r)) \lesssim \delta^{-\epsilon}r^{t}$ for all $x \in \R^{2}$ and $r \geq \delta$. Consider
\begin{displaymath} \mu := \Delta^{-t}S_{10Q}(\mu_{0}|_{10Q}). \end{displaymath}
We remark that $\mu$ may not be a probability measure: the best one can say is that $\mu(\R^{2}) = \Delta^{-t} \cdot |\mathcal{Q}|/|\mathcal{P}| \lesssim \delta^{-\epsilon}$. In fact, more generally for $x \in \R^{2}$, $r \geq \delta/\Delta$, and $y := S_{10Q}^{-1}(x)$, 
\begin{displaymath} \mu(B(x,r)) = \Delta^{-t} \mu_{0}(S_{10Q}^{-1}(B(x,r))) = \Delta^{-t}(\mu_{0}(B(y,\Delta r)) \lesssim \delta^{-\epsilon}\Delta^{-t}(\Delta r)^{t} = \delta^{-\epsilon}r^{t}. \end{displaymath}
In other words, $\mu$ satisfies the Frostman condition in Theorem \ref{t:orponen} with $C \sim \delta^{-\epsilon}$. 

Another key hypothesis of Theorem \ref{t:orponen} is that the $(\delta/\Delta,s,\delta^{-\epsilon})$-sets $\mathcal{T}_{q}$ exist for all $q \in \mathcal{D}_{\delta/\Delta}(\spt \mu)$, but this is only the case for $q \in \overline{\mathcal{B}}$. This is no problem, however, because $\cup \overline{\mathcal{B}}$ has nearly full $\mu$ measure: recalling the definition \eqref{form41}, and then the lower bound \eqref{form38}, we infer
\begin{equation}\label{form42} \mu(\cup \overline{\mathcal{B}}) \gtrapprox \mu(\R^{2}) = \Delta^{-t} \cdot |\mathcal{Q}|/|\mathcal{P}|. \end{equation}
Now, the measure $\bar{\mu} := \mu|_{\cup \overline{\mathcal{B}}}$ satisfies the hypotheses of Theorem \ref{t:orponen} at scale $\delta/\Delta$, with constant $C \sim \delta^{-\epsilon}$, and  
\begin{equation}\label{form46} \sigma := 2 - t < s. \end{equation} The conclusion is that $\bigcup_{q \in \overline{\mathcal{B}}} \mathcal{T}_{q}$ contains a Katz-Tao $(\delta/\Delta,\sigma + 1)$-set $\mathcal{T}$ of cardinality
\begin{equation}\label{form45} |\mathcal{T}| \gtrapprox \bar{\mu}(\R^{2})\left(\tfrac{\delta}{\Delta}\right)^{-(\sigma + 1)} \stackrel{\eqref{form42}}{\gtrapprox} \Delta^{-t} \cdot \tfrac{|\mathcal{Q}|}{|\mathcal{P}|} \cdot \left(\tfrac{\delta}{\Delta} \right)^{-(\sigma + 1)}. \end{equation}
In the remainder of the proof, we will check that, thanks to \eqref{form39}, the cardinality of incidences $\mathcal{I}(\mathcal{Q},\mathcal{T})$ is high. We will then compare this with the upper bound from Fu and Ren's result, Theorem \ref{t:fuRen}, at scale $\delta/\Delta$ to derive a contradiction. 

Recall that $\mathcal{Q} = S_{10Q}(\mathcal{P} \cap 10Q)$. By  \eqref{form39}, and the fact that the family $\mathcal{T}$ consists of $S_{10Q}$-rescaled versions of (some of) the rectangles $R = R(p,\theta)$ considered at \eqref{form39}, 
\begin{displaymath} |\{q \in \mathcal{Q} : q \cap T \neq \emptyset\}| \geq \delta^{1 - \eta}\Delta^{t - 1}|\mathcal{P}|, \qquad T \in \mathcal{T}. \end{displaymath}
Summing this over $T \in \mathcal{T}$,
\begin{equation}\label{form44} |\mathcal{I}(\mathcal{Q},\mathcal{T})| \geq \delta^{1 - \eta}\Delta^{t - 1}|\mathcal{P}||\mathcal{T}| \end{equation}
On the other hand, we may apply Theorem \ref{t:fuRen}, at scale $\delta/\Delta$ to find a strong upper bound for $|\mathcal{I}(\mathcal{Q},\mathcal{T})|$. For this purpose, we need to note or recall that
\begin{itemize}
\item $\mathcal{Q} = S_{Q}(\mathcal{P} \cap 10Q)$ is a Katz-Tao $(\delta/\Delta,t,C_{\mathcal{Q}})$-set of cardinality $|\mathcal{Q}| = |\mathcal{P} \cap 10Q|$, and constant $C_{\mathcal{Q}} \lessapprox \delta^{t}|\mathcal{P}|$. The latter claim follows immediately from the $(\delta,t,\delta^{-\epsilon})$-set property of $\mathcal{P}$, and the calculation
\begin{displaymath} |\mathcal{Q} \cap B(x,r)| = |\mathcal{P} \cap B(y,\Delta r)| \leq \delta^{-\epsilon}(\Delta r)^{t}|\mathcal{P}| = (\delta^{t - \epsilon}|\mathcal{P}|) \cdot \left(\tfrac{r}{\delta/\Delta} \right)^{t}, \quad r \geq \delta/\Delta. \end{displaymath}

\item $\mathcal{T}$ is a Katz-Tao $(\delta,\sigma + 1)$-set.
\end{itemize} 
With this information in hand, Theorem \ref{t:fuRen} applied with exponents $t$ and $\sigma + 1 = 3 - t$  implies
\begin{displaymath} |\mathcal{I}(\mathcal{Q},\mathcal{T})| \lessapprox \left(\tfrac{\delta}{\Delta} \right)^{-1/2}(\delta^{t}|\mathcal{P}|)^{1/2}|\mathcal{Q}|^{1/2}|\mathcal{T}|^{1/2}. \end{displaymath} 
Combining this upper bound with the lower bound in \eqref{form44} leads to
\begin{displaymath} |\mathcal{T}| \lessapprox \delta^{t - 3 + 2\eta} \cdot \Delta^{3 - 2t} \cdot \tfrac{|\mathcal{Q}|}{|\mathcal{P}|}. \end{displaymath}
Comparing this upper bound further with the lower bound \eqref{form45} yields
\begin{displaymath} \Delta^{-t} \cdot \tfrac{|\mathcal{Q}|}{|\mathcal{P}|} \cdot \left(\tfrac{\delta}{\Delta} \right)^{-(\sigma + 1)} \lessapprox \delta^{t - 3 + 2\eta} \cdot \Delta^{3 - 2t} \cdot \tfrac{|\mathcal{Q}|}{|\mathcal{P}|}, \end{displaymath}
which can be rearranged to $\delta^{2 - t - \sigma - 2\eta} \lessapprox \Delta^{2 - t - \sigma}$. Recall now that the notation $A \lessapprox B$ is an abbreviation for $A \leq C_{s,t}\delta^{-C\epsilon}B$. Since $\sigma + t = 2$ by \eqref{form46}, and $\Delta \geq \delta$, the aligned inequality above produces a contradiction if $\epsilon = c\eta$ for a sufficiently small absolute constant $c > 0$, and $\delta > 0$ small enough depending on $\eta,s,t$. \end{proof}


\section{Main lemma}\label{s6}

Morally, Theorem \ref{t2} imply would Theorem \ref{t:configurations} (and therefore Theorem \ref{t:main}) if we could ensure the validity of the extra hypotheses \nref{H1}-\nref{H2}. In this section we give a sufficient condition (Lemma \ref{l:main})  for achieving the hypotheses \nref{H1}-\nref{H2} of Theorem \ref{t2}. Unfortunately, the "original" configuration $(\mu,\{\sigma_{x}\})$ naturally associated to Theorem \ref{t:configurations} need not satisfy these conditions, but we will show in the next section that a suitable "renormalised" configuration does. Eventually, in Section \ref{s4}, Theorem \ref{t:configurations} will be proven by applying Lemma \ref{l:main} and Theorem \ref{t2} to that renormalised configuration.

\begin{definition}[$\mu \times \sigma_{x}$] Given a configuration $(\mu,\{\sigma_{x}\})$, we define the measure $\mu \times \sigma_{x}$ on $\R^{2} \times [0,1]$ as the Radon measure produced by the Riesz representation theorem applied to the positive linear functional determined by
\begin{equation}\label{form60} \int g \, d(\mu \times \sigma_{x}):= \int \int g(x,\theta) \, d\sigma_{x}(\theta) \, d\mu(x), \qquad g \in C_{c}(\R^{2} \times [0,1]). \end{equation}
\end{definition}

\begin{remark} The right hand side of \eqref{form60} makes sense by the measurability hypothesis in Definition \ref{def:configuration}. Let us also emphasise that the measure $\mu \times \sigma_{x}$ on $\R^{2} \times [0,1]$ should not be confused with the measure $\mu(\sigma_{x})$ on $[0,1] \times \R$ from Definition \ref{def:productMeasure}. In fact, whereas $\mu(\sigma_{x})$ is best interpreted as a measure on the space $\mathcal{A}(2,1)$ of all affine lines, $\mu \times \sigma_{x}$ is best interpreted as a measure on $\R^{2} \times \mathcal{A}(2,1)$. The next \textbf{Notation} makes this more precise.  \end{remark}

\begin{notation} For a Borel set $B \subset \R^{2} \times \mathcal{A}(2,1)$, we denote 
\begin{equation}\label{form81} (\mu \times \sigma_{x})(B) := (\mu \times \sigma_{x})(\{(x,\theta) \in \R^{2} \times [0,1] : (x,\ell_{x,\theta}) \in B\}). \end{equation} 
In other words, we will use the identification above to consider $\mu \times \sigma_{x}$ either as a measure on $\R^{2} \times [0,1]$, or as a measure on $\R^{2} \times \mathcal{A}(2,1)$ depending on which is more convenient. \end{notation}

\begin{definition}[Tight configuration]\label{def:tightness} Let $C \geq 1$. A configuration $(\mu,\{\sigma_{x}\})$ is \emph{$C$-tight at scale $\Delta \in 2^{-\N}$ with data $(\mathcal{Q},\mathbb{T}$)} if there exist families
\begin{displaymath} \mathcal{Q} \subset \mathcal{D}_{\Delta} \quad \text{and} \quad \mathbb{T} = \bigcup_{Q \in \mathcal{Q}} \mathbb{T}_{Q} \subset \mathcal{T}^{\Delta} \end{displaymath}
with the following properties.
\begin{itemize}
\item[(T1) \phantomsection \label{T1}] The families $\mathbb{T}_{Q}$ have constant cardinality $M \geq 1$.
\item[(T2) \phantomsection \label{T2}] $Q \mapsto \mu(Q)$ is constant up to a factor of $2$ on $\mathcal{Q}$, and $\mu(\cup \mathcal{Q}) \geq C^{-1}$.
\item[(T3) \phantomsection \label{T3}]  $(\mu \times \sigma_{x})(Q \times \mathbf{T}) \geq C^{-1}\mu(Q)/M > 0$ for all $\mathbf{T} \in \mathbb{T}_{Q}$.
\item[(T4) \phantomsection \label{T4}] The slope set $\sigma(\mathbb{T})$ has cardinality $|\sigma(\mathbb{T})| \leq CM$.
\end{itemize}
\end{definition}

\begin{remark} \nref{T3} ensures that all the tubes $\mathbf{T} \in \mathbb{T}_{Q}$ intersect $Q$. In particular, \nref{T1} implies
\begin{equation}\label{form48} |\sigma(\mathbb{T})| \geq |\sigma(\mathbb{T}_{Q})| \gtrsim M, \qquad Q \in \mathcal{Q}. \end{equation}
  \end{remark} 

\begin{lemma}\label{l:main} For all $t \in (1,2]$, $s \in (2 - t,1]$, and $\eta > 0$ there exist $\epsilon = \epsilon(\eta,s,t) > 0$ and $\Delta_{0} = \Delta_{0}(\eta,s,t) > 0$ such that the following holds for all $\Delta \in 2^{-\N} \cap (0,\Delta_{0}]$. 

Let $(\mu,\{\sigma_{x}\})$ be a $\Delta^{-\epsilon}$-tight configuration at scale $\Delta$ with data $(\mathcal{Q},\mathbb{T})$, where $\Theta = \sigma(\mathbb{T})$ is a non-empty $(\Delta,s,\Delta^{-\epsilon})$-set, and $\mathcal{Q}$ is a $(\Delta,t,\Delta^{-\epsilon})$-set. Then, there exist $\mathcal{G}_{1},\mathcal{G}_{2} \subset \mathcal{D}_{\Delta}$ with $\dist(\mathcal{G}_{1},\mathcal{G}_{2}) \geq \Delta$, and the following properties:
\begin{itemize}
\item[(M1) \phantomsection \label{M1}] $\min\{\mu(G_{1}),\mu(G_{2})\} \geq \Delta^{\eta}$, where $G_{j} := (\cup \mathcal{G}_{j}) \cap \spt \mu$.
\item[(M2) \phantomsection \label{M2}] $X[\mu|_{G_{1}},\{\sigma_{x}\}]_{\Delta}(y) \geq \Delta^{\eta}$ for all $y \in G_{2}$.
\end{itemize}
Here $X[\ldots]_{\Delta} = X[\ldots] \ast \varphi_{\Delta}$, where $\varphi \in C^{\infty}_{c}(\R^{2})$ is any function satisfying 
\begin{displaymath} \mathbf{1}_{B(100)} \leq \varphi \leq \mathbf{1}_{B(200)}, \end{displaymath}
and $\varphi_{r}(x) = r^{-2}\varphi(x/r)$ for $r > 0$.
\end{lemma} 

\begin{proof} Fix $\eta > 0$, and let $\bar{\eta} > 0$ be so small that $4\bar{\eta}/(t - 1) \leq \eta$. Let $\epsilon_{\ref{prop5}} = \epsilon_{\ref{prop5}}(s,t,\bar{\eta}) > 0$ be the constant given by Proposition \ref{prop5} with parameters $s,t,\bar{\eta}$. Finally, assume that 
\begin{equation}\label{form51} 0 < \epsilon \leq \tfrac{1}{8}\min\{\bar{\eta},\epsilon_{\ref{prop5}}\}. \end{equation}
We claim that the conclusions of Lemma \ref{l:main} are valid for any such choice of "$\epsilon$", provided that $\Delta > 0$ is sufficiently small in terms of $\eta,s,t$ (we do not explicitly track the required upper bound for $\Delta$, and we also omit constantly writing "if $\Delta$ is small enough").

The squares in $\mathcal{Q}$ are disjoint, but we desire them to be $\Delta$-separated to eventually guarantee the condition $\dist(\mathcal{G}_{1},\mathcal{G}_{2}) \geq \Delta$. This can be achieved by replacing $\mathcal{Q}$ by a subset of cardinality $\gtrsim |\mathcal{Q}|$. This replacement has no noticeable effect on the validity of the hypotheses of Lemma \ref{l:main}, so we assume that $\mathcal{Q}$ is $\Delta$-separated to begin with.

We record the following consequence of \nref{T2} for our $\Delta^{-\epsilon}$-tight configuration
\begin{equation}\label{form11} \mu(Q)|\mathcal{Q}| \sim \mu(\cup \mathcal{Q}) \geq \Delta^{\epsilon}, \qquad Q \in \mathcal{Q}. \end{equation}

For all $\mathcal{Q}' \subset \mathcal{Q}$ and $\mathbb{T}' \subset \mathbb{T}$, write $\mathcal{I}(\mathcal{Q}',\mathbb{T}') := \{(Q,\mathbf{T}) \in \mathcal{Q}' \times \mathbb{T}' : \mathbf{T} \in \mathbb{T}_{Q}\}$. Observe that
\begin{displaymath} |\mathcal{I}(\mathcal{Q},\mathbb{T})| = \sum_{Q \in \mathcal{Q}} |\mathbb{T}_{Q}| = M|\mathcal{Q}|. \end{displaymath}
A useful alternative way to count the incidences $\mathcal{I}(\mathcal{Q},\mathbb{T})$ is the following. For $\theta \in \Theta = \sigma(\mathbb{T})$, let $\mathbb{T}_{\theta} := \{T \in \mathbb{T} : \sigma(T) = \theta\}$ and $\mathcal{Q}_{\theta} := \{Q \in \mathcal{Q} : \mathbb{T}_{Q} \cap \mathbb{T}_{\theta} \neq \emptyset \}$. Since for each $Q \in \mathcal{Q}_{\theta}, |\mathbb{T}_{Q} \cap \mathbb{T}_{\theta}| \lesssim 1$,
\begin{equation} \label{e:altcount}M|\mathcal{Q}| = |\mathcal{I}(\mathcal{Q},\mathbb{T})| = \sum_{\theta \in \Theta} |\mathcal{I}(\mathcal{Q},\mathbb{T}_{\theta})| \sim \sum_{\theta \in \Theta} |\mathcal{Q}_{\theta}|. \end{equation}
 On the other hand, $|\mathcal{Q}_{\theta}| \leq |\mathcal{Q}|$  and \nref{T4} ensures $|\Theta| \leq \Delta^{-\epsilon}M$. So, combined with \eqref{e:altcount} there exists a subset $\Theta' \subset \Theta$ of cardinality $|\Theta'| \gtrsim \Delta^{\epsilon}|\Theta|$ such that $|\mathcal{Q}_{\theta}| \gtrsim \Delta^{2 \epsilon}|\mathcal{Q}|$ for all $\theta \in \Theta'$. In particular, $\Theta'$ is a $(\Delta,s,\Delta^{-\epsilon_{\ref{prop5}}})$-set. This places us in a position to apply Proposition \ref{prop5} with parameters $s,t$, at scale $\Delta$, to the $(\Delta,t,\Delta^{-\epsilon_{\ref{prop5}}})$-set $\mathcal{Q}$, and to the set of slopes $\Theta'$ found just above.

The conclusion is that there exists a further subset $\bar{\Theta} \subset \Theta'$ with $|\bar{\Theta}| \geq \tfrac{1}{2}|\Theta'| \gtrsim \Delta^{2 \epsilon} |\Theta|$, and for each $\theta \in \bar{\Theta}$ a subset $\mathcal{Q}_{\theta}' \subset \mathcal{Q}$ of cardinality $|\mathcal{Q}_{\theta}'| \geq (1 - \Delta^{\epsilon_{\ref{prop5}}})|\mathcal{Q}|$ such that
\begin{equation}\label{form13} |\mathcal{Q}_{\theta}' \cap R| \lesssim \Delta^{1 - \bar{\eta}}\bar{\Delta}^{t - 1}|\mathcal{Q}|, \end{equation} 
whenever $R \subset \R^{2}$ is a rectangle of dimensions $(\Delta \times \bar{\Delta})$ parallel to $\theta$ (here $\mathcal{Q}_{\theta}' \cap R = \{Q' \in \mathcal{Q}_{\theta}' : Q' \cap R \neq \emptyset\}$). Since $|\mathcal{Q} \, \setminus \, \mathcal{Q}_{\theta}'| \leq \Delta^{\epsilon_{\ref{prop5}}}|\mathcal{Q}| \leq \tfrac{1}{2}|\mathcal{Q}_{\theta}|$, the sets $\mathcal{Q}_{\theta}$ and $\mathcal{Q}_{\theta}'$ have large intersection. In fact, for $\Delta$ small enough, the cardinality of $Q_{\theta}^{\prime}$ is so large that, writing 
\begin{displaymath} \mathcal{Q}_{\theta}'' := \mathcal{Q}_{\theta} \cap \mathcal{Q}_{\theta}', \end{displaymath}
the definition of $\Theta^{\prime} \supset \bar \Theta$ ensures we have
\begin{displaymath} |\mathcal{Q}_{\theta}''| \geq \tfrac{1}{2}|\mathcal{Q}_{\theta}| \gtrsim \Delta^{2 \epsilon}|\mathcal{Q}|, \qquad \theta \in \bar{\Theta}. \end{displaymath}
Of course the estimate \eqref{form13} persists for $\mathcal{Q}_{\theta}'' \subset \mathcal{Q}_{\theta}'$.

To keep track of this information, we define  the \emph{restricted incidences}
\begin{displaymath} \overline{\mathcal{I}}(\mathcal{Q}',\mathbb{T}') := \{(Q,\mathbf{T}) \in \mathcal{Q}' \times \mathbb{T}' : \mathbf{T} \in \mathbb{T}_{Q} \text{ and } Q \in \mathcal{Q}_{\sigma(\mathbf{T})}''\},  \end{displaymath} 
for arbitrary $\mathcal{Q}' \subset \mathcal{Q}$ and $\mathbb{T}' \subset \mathbb{T}$. From the cardinality lower bounds for the sets $\bar{\Theta}$ and $\mathcal{Q}_{\theta}''$, we will soon deduce that the number of restricted incidences remains nearly maximal. We first record that if $\theta \in \bar{\Theta}$ and $Q \in \mathcal{Q}_{\theta}''$, then in particular $Q \in \mathcal{Q}_{\theta}$, and thus $\mathbf{T} \in \mathbb{T}_{Q}$ for some $\mathbf{T} \in \mathbb{T}_{\theta}$. Then $(Q,\mathbf{T}) \in \overline{\mathcal{I}}(\mathcal{Q},\mathbb{T}_{\theta})$ by definition. This shows that $|\mathcal{Q}_{\theta}''| \leq |\overline{\mathcal{I}}(\mathcal{Q},\mathbb{T}_{\theta})|$, and consequently
\begin{displaymath} |\overline{\mathcal{I}}(\mathcal{Q},\mathbb{T})| \geq \sum_{\theta \in \bar{\Theta}} |\overline{\mathcal{I}}(\mathcal{Q},\mathbb{T}_{\theta})| \geq \sum_{\theta \in \bar{\Theta}} |\mathcal{Q}_{\theta}''| \gtrsim \Delta^{2\epsilon}|\bar{\Theta}||\mathcal{Q}| \stackrel{\eqref{form48}}{\gtrsim} \Delta^{3\epsilon}M|\mathcal{Q}|. \end{displaymath}

Next, we pass to a subset of $\mathbb{T}$ with a roughly constant number of restricted incidences.  By the pigeonhole principle, choose a subset $\overline{\mathbb{T}} \subset \mathbb{T}$ and a number $N \geq 1$ such that
\begin{itemize}
\item[(a) \phantomsection \label{a}] $|\{Q \in \mathcal{Q}_{\sigma(\mathbf{T})}'' : \mathbf{T} \in \mathbb{T}_{Q}\}| \in [N,2N]$ for all $\mathbf{T} \in \overline{\mathbb{T}}$, and
\item[(b) \phantomsection \label{b}] $N|\overline{\mathbb{T}}| \sim |\overline{\mathcal{I}}(\mathcal{Q},\overline{\mathbb{T}})| \gtrapprox_{\Delta} \Delta^{3\epsilon}M|\mathcal{Q}|$.
\end{itemize}
Notice that $|\overline{\mathbb{T}}| \leq |\mathbb{T}| \lesssim \Delta^{-1 - \epsilon}M$ by the tightness hypothesis \nref{T4}, and since all the tubes in $\mathbb{T}$ intersect $\cup \mathcal{Q} \subset [0,1]^{2}$. Consequently,
\begin{equation}\label{form6a} N \gtrapprox_{\Delta} \frac{\Delta^{3\epsilon}M|\mathcal{Q}|}{|\overline{\mathbb{T}}|} \gtrsim \Delta^{1 + 4\epsilon}|\mathcal{Q}|. \end{equation}

We next replace $\mathcal{Q}$ by a subset $\overline{\mathcal{Q}}$ by another pigeonholing argument. Note that
\begin{displaymath} \Delta^{3\epsilon}M|\mathcal{Q}| \stackrel{\textup{\nref{b}}}{\lessapprox_{\Delta}} |\overline{\mathcal{I}}(\mathcal{Q},\overline{\mathbb{T}})| = \sum_{Q \in \mathcal{Q}} |\{\mathbf{T} \in \overline{\mathbb{T}} \cap \mathbb{T}_{Q} : Q \in \mathcal{Q}_{\sigma(\mathbf{T})}''\}| =: \sum_{Q \in \mathcal{Q}} |\overline{\mathbb{T}}_{Q}|. \end{displaymath}
Since $|\overline{\mathbb{T}}_{Q}| \leq |\mathbb{T}_{Q}| = M$, there exists a subset $\overline{\mathcal{Q}} \subset \mathcal{Q}$ with $|\overline{\mathcal{Q}}| \gtrapprox_{\Delta} \Delta^{4\epsilon}|\mathcal{Q}|$ such that 
\begin{displaymath} |\overline{\mathbb{T}}_{Q}| \gtrapprox_{\Delta} \Delta^{4\epsilon}M, \qquad Q \in \overline{\mathcal{Q}}. \end{displaymath}

We then proceed to define a graph with vertex set $\mathcal{Q}$, and edge set $\mathcal{E}$, as follows. For $Q,Q'$ distinct, we set $(Q,Q') \in \mathcal{E}$ if there exists a tube $\mathbf{T} \in \overline{\mathbb{T}}_{Q} \cap \overline{\mathbb{T}}_{Q'}$. (Note that $(Q,Q') \in \mathcal{E}$ if and only if $(Q',Q) \in \mathcal{E}$, so $(\mathcal{Q},\mathcal{E})$ is an undirected graph.) We claim that
\begin{equation}\label{form1a} |\mathcal{E}| \geq \Delta^{\eta}MN|\mathcal{Q}| \end{equation}
for small enough $\Delta > 0$. To see this, fix $Q \in \overline{\mathcal{Q}}$, so $|\overline{\mathbb{T}}_{Q}| \gtrapprox_{\Delta} \Delta^{4\epsilon}M$. Moreover, for every $\mathbf{T} \in \overline{\mathbb{T}}_{Q} \subset \overline{\mathbb{T}}$,
\begin{equation}\label{form50} |\{Q' \in \mathcal{Q}_{\sigma(\mathbf{T})}'' : \mathbf{T} \in \overline{\mathbb{T}}_{Q'}\}| = |\{Q' \in \mathcal{Q}_{\sigma(\mathbf{T})}'' : \mathbf{T} \in \mathbb{T}_{Q'}\}| \stackrel{\textup{\nref{a}}}{\geq} N \stackrel{\eqref{form6a}}{\gtrapprox_{\Delta}} \Delta^{1 + 3\epsilon}|\mathcal{Q}|. \end{equation}
All the $N$ squares $Q' \in \mathcal{Q}$ counted here contribute an edge $(Q',Q) \in \mathcal{E}$. Since $\mathbf{T}$ can be selected in $\gtrapprox_{\Delta} \Delta^{4\epsilon}M$ different ways for each $Q \in \overline{\mathcal{Q}}$, this might seem to give something even stronger than \eqref{form1a} (since $\epsilon < \eta/5$). But there is a catch: some of the $N$ squares $Q'$ corresponding to different tubes in $\overline{\mathbb{T}}_{Q}$ may coincide if $\dist(Q',Q) \ll 1$. The non-concentration condition \eqref{form13} is needed to fix this.

For $Q \in \overline{\mathcal{Q}}$ and $\mathbf{T} \in \overline{\mathbb{T}}_{Q}$ still fixed, and $\bar{\Delta} \in [\Delta,1]$, we claim that
\begin{equation}\label{form49} I_{\bar{\Delta}}(Q, \mathbf{T}) \vcentcolon = |\{Q' \in \mathcal{Q}_{\sigma(\mathbf{T})}'' : \mathbf{T} \in \mathbb{T}_{Q'} \text{ and } \dist(Q,Q') \leq \bar{\Delta}\}| \stackrel{\eqref{form13}}{\lesssim} \Delta^{1 - \bar{\eta}}\bar{\Delta}^{t - 1}|\mathcal{Q}|. \end{equation}
This is because $\mathcal{Q}_{\sigma(\mathbf{T})}'' \subset \mathcal{Q}_{\sigma(\mathbf{T})}'$, and because the conditions $\dist(Q',Q) \leq \bar{\Delta}$ and $\mathbf{T} \in \mathbb{T}_{Q'}$ imply that $Q'$ intersects a rectangle $R$ of dimensions $\sim (\Delta \times \bar{\Delta})$ parallel to $\mathbf{T}$.

In particular, choosing
\begin{displaymath} \bar{\Delta} := \Delta^{(4\epsilon + \bar{\eta})/(t - 1)} \geq \Delta^{2\bar{\eta}/(t - 1)} \geq \Delta^{\eta/2}, \end{displaymath}
implies $I_{\bar{\Delta}}(Q,\mathbf{T}) \lesssim  \Delta^{1 + 4\epsilon}|\mathcal{Q}| \lessapprox_{\Delta} \Delta^{\epsilon} N$.
For this choice of $\bar{\Delta}$ we deduce from \eqref{form50} that for $\Delta$ small enough,
\begin{displaymath} |\{Q' \in \mathcal{Q}_{\sigma(\mathbf{T})}'' : \mathbf{T} \in \overline{\mathbb{T}}_{Q'} \text{ and } \dist(Q',Q) \geq \Delta^{\eta/2}\}| \geq N - I(Q, \mathbf{T}) \ge \tfrac{1}{2}N, \qquad Q \in \overline{\mathcal{Q}}, \, \mathbf{T} \in \overline{\mathbb{T}}_{Q}. \end{displaymath}
For $Q \in \overline{\mathcal{Q}}$ fixed, the families 
\begin{displaymath} \{Q' \in \mathcal{Q}_{\sigma(\mathbf{T})}'' : \mathbf{T} \in \overline{\mathbb{T}}_{Q'} \text{ and } \dist(Q',Q) \geq \Delta^{\eta/2}\} \end{displaymath}
have overlap bounded by $\lesssim \Delta^{-\eta/2}$ as $\mathbf{T} \in \overline{\mathbb{T}}_{Q}$ varies. Since $4 \epsilon \le \eta /2$ and $|\overline{\mathbb{T}}_{Q}| \gtrapprox_{\Delta} \Delta^{4\epsilon}M$, we get \eqref{form1a}.

Erd{\H o}s \cite[Lemma 1]{MR190027} has shown that every (undirected) graph $(\mathcal{Q},\mathcal{E})$ contains a bi-partite sub-graph $(\mathcal{G}_{1} \dot{\cup} \mathcal{G}_{2},\overline{\mathcal{E}})$ (here $\mathcal{G}_{1},\mathcal{G}_{2} \subset \mathcal{Q}$ are disjoint, and $\overline{\mathcal{E}} \subset \mathcal{E}$ consists of edges between $\mathcal{G}_{1},\mathcal{G}_{2}$) with 
\begin{equation}\label{form3a} |\overline{\mathcal{E}}| \geq \tfrac{1}{2}|\mathcal{E}| \geq \tfrac{1}{2}\Delta^{\eta}MN|\mathcal{Q}|. \end{equation}
For $Q \in \mathcal{G}_{j}$, $j \in \{1,2\}$, let $d(Q) := |\{Q' \in \mathcal{G}_{3 - j} : (Q',Q) \in \overline{\mathcal{E}}\}|$ be the degree of a vertex in $\mathcal{G}_{j}$. Let
\begin{displaymath} \overline{\mathcal{G}}_{j} := \{Q \in \mathcal{G}_{j} : d(Q) \geq \tfrac{1}{10}\Delta^{\eta}MN \}, \qquad j \in \{1,2\}, \end{displaymath}
be the "high degree" squares in $\mathcal{G}_{j}$. We claim that
\begin{equation}\label{form2a} \min\{|\mathcal{G}_{1}|,|\mathcal{G}_{2}|\} \geq \min\{|\overline{\mathcal{G}}_{1}|,|\overline{\mathcal{G}}_{2}|\} \geq \tfrac{1}{10} \Delta^{\eta}|\mathcal{Q}|. \end{equation}
This follows from the uniform upper bound
\begin{displaymath} |\{Q' \in \mathcal{Q} : (Q',Q) \in \mathcal{E}\}| \leq \sum_{\mathbf{T} \in \overline{\mathbb{T}}_{Q}} |\{Q' \in \mathcal{Q}_{\sigma(\mathbf{T})}'' : \mathbf{T} \in \overline{\mathbb{T}}_{Q'}\}| \stackrel{\textup{\nref{a}}}{\leq} 2MN, \qquad Q \in \mathcal{Q}, \end{displaymath}
so in particular $d(Q) \leq 2MN$ for $Q \in \mathcal{G}_{1} \cup \mathcal{G}_{2}$. Thus, if \eqref{form2a} failed, we could estimate the number of edges in $\overline{\mathcal{E}}$ (recall that $\mathcal{E}$ is an undirected graph) as
\begin{displaymath} |\overline{\mathcal{E}}| = \sum_{Q \in \mathcal{G}_{j} \, \setminus \, \overline{\mathcal{G}}_{j}} d(Q) + \sum_{Q \in \overline{\mathcal{G}}_{j}} d(Q) \leq \tfrac{3}{10}\Delta^{\eta} MN|\mathcal{Q}|, \qquad j \in \{1,2\}, \end{displaymath}
violating \eqref{form3a}.

We will now verify that the claims Lemma \ref{l:main}\nref{M1}-\nref{M2} are satisfied by the families $\mathcal{G}_{1}$ and $\overline{\mathcal{G}}_{2}$. The measure lower bounds in \nref{M1} are clear from \eqref{form2a} and the constancy of $Q \mapsto \mu(Q)$ on $\mathcal{Q}$, recall \eqref{form11}. It remains to prove the $X$-ray measure lower bound \nref{M2}.

Fix $y \in G_{2} = (\cup \overline{\mathcal{G}}_{2}) \cap \spt \mu$, and $Q \in \overline{\mathcal{G}}_{2}$ such that $y \in Q$. Recalling that $X[\ldots]_{\Delta} = X[\ldots] \ast \varphi_{\Delta}$, and that $\mathbf{1}_{B(100)} \leq \varphi \leq \mathbf{1}_{B(200)}$,
\begin{align} X[\mu|_{G_{1}},\{\sigma_{x}\}]_{\Delta}(y) & \gtrsim \frac{X[\mu|_{G_{1}},\{\sigma_{x}\}](10Q)}{\Delta^{2}} \notag\\
&\label{form10a} = \Delta^{-2} \int_{G_{1}} \int X(\mathbf{1}_{10Q})(\theta,\pi_{\theta}(x)) \, d\sigma_{x}(\theta) \, d\mu(x). \end{align}
Note that if $\ell_{x,\theta} \cap 5Q \neq \emptyset$, then $X(\mathbf{1}_{10Q})(\theta,\pi_{\theta}(x)) \sim \Delta$. This motivates studying
\begin{displaymath} A := (\mu \times \sigma_{x})(\{(x,\theta) \in G_{1} \times [0,1] : \ell_{x,\theta} \cap 5Q \neq \emptyset\}). \end{displaymath}
Here is a crucial observation: if $x \in Q' \in \mathcal{G}_{1}$, and $(Q,Q') \in \overline{\mathcal{E}}$, then there exists a tube $\mathbf{T} \in \mathbb{T}_{Q} \cap \mathbb{T}_{Q'}$ (thus $Q \cap \mathbf{T} \neq \emptyset \neq Q' \cap \mathbf{T}$). Now, $\ell_{x,\theta} \cap 5Q \neq \emptyset$ for all lines $\ell_{x,\theta} \subset \mathbf{T}$. Therefore, recalling the notation \eqref{form81},
\begin{align*} (Q,Q') \in \overline{\mathcal{E}} \quad \Longrightarrow \quad   (\mu \times \sigma_{x}) (\{(x,\theta) \in Q' \times [0,1] : \ell_{x,\theta} \cap 5Q \neq \emptyset\}) & \geq (\mu \times \sigma_{x})(Q' \times \mathbf{T})\\
&\geq \Delta^{\epsilon}\mu(Q)/M, \end{align*}
using the $\Delta^{-\epsilon}$-tightness hypothesis \nref{T3} in the final inequality. Since $|\{Q' \in \mathcal{G}_{1} : (Q,Q') \in \overline{\mathcal{E}}\}| = d(Q) \gtrsim \Delta^{\eta}MN$ by the definition of $Q \in \overline{\mathcal{G}}_{2}$, we find

\begin{align*} A & \geq \mathop{\sum_{Q' \in \mathcal{G}_{1}}}_{(Q,Q') \in \overline{\mathcal{E}}} (\mu \times \sigma_{x})(\{(x,\theta) \in Q' \times [0,1] : \ell_{x,\theta} \cap 5Q \neq \emptyset\}) \gtrsim \Delta^{\eta}MN \cdot \Delta^{\epsilon}\mu(Q)/M \\
& \stackrel{\eqref{form6a}}{\gtrapprox_{\Delta}} \Delta^{1 + 4\epsilon + \eta}\mu(Q)|\mathcal{Q}| \stackrel{\eqref{form11}}{\geq} \Delta^{1 + 5\epsilon + \eta} \stackrel{\eqref{form51}}{\geq} \Delta^{1 + 2\eta}. \end{align*} 
Taking into account the factor $\Delta^{-2}$ in \eqref{form10a}, and the lower bound $X(\mathbf{1}_{10Q})(\theta,\pi_{\theta}(x)) \gtrsim \Delta$ whenever $\ell_{x,\theta} \cap 5Q \neq \emptyset$, this concludes the proof of the proposition. \end{proof}


\section{Finding a tight sub-configuration}\label{s7}

In this section we prove that a (finitary) "blow-up" of $(\mu,\{\sigma_{x}\})$ satisfies the tightness hypotheses required to apply Lemma \ref{l:main}. Here is what we mean by "blow-up":
\begin{definition}[Renormalised configuration]\label{def:renormalisation} Let $(\mu,\{\sigma_{x}\})$ be a configuration, and let $Q \subset [0,1)^{2}$ be a dyadic square with $\mu(Q) > 0$. The \emph{$Q$-renormalised configuration} is $(\mu^{Q},\{\sigma_{y}^{Q}\})$, where
\begin{displaymath} \mu^{Q} := \tfrac{1}{\mu(Q)}S_{Q}(\mu|_{Q}) \quad \text{and} \quad \sigma_{y}^{Q} := \sigma_{S_{Q}^{-1}(y)}. \end{displaymath}
Here $S_{Q}$ is the homothety $Q \to [0,1)^{2}$, and $S_{Q}(\mu|_{Q})$ is the push-forward of $\mu|_{Q}$ by $S_{Q}$. \end{definition}

Here is the main result of the section:
\begin{thm}\label{t:renormalisation} Let $s,t \in (0,2]$, $\tau \in (0,t)$, $C > 0$, and $\epsilon > 0$. Then, there exist $\Delta_{0} = \Delta_{0}(C,\epsilon,t,\tau) > 0$ and $n = n(\epsilon,t,\tau) \in \N$ such that the following holds for all $\Delta_{1} \in 2^{-\N} \cap (0,\Delta_{0}]$. Let $(\mu,\{\sigma_{x}\})$ be a configuration, where $\mu$ is a $(t,C)$-Frostman probability measure, and $\sigma_{x}$ is an $(s,C)$-Frostman probability measure for $\mu$ almost all $x \in \R^{2}$.

Then, there exist 
\begin{itemize}
\item dyadic scales $\underline{\Delta},\Delta \in [\Delta_{1}^{n},\Delta_{1}]$ with $\underline{\Delta} \leq \Delta$
\item a measure $\bar{\mu} = \mu|_{B}$, where $B \subset [0,1)^{2}$ is Borel, and 
\item a square $Q \in \mathcal{D}_{\underline{\Delta}}$
\end{itemize}
such that the $Q$-renormalised configuration $(\bar{\mu}^{Q},\{\sigma^{Q}_{y}\})$ is $\Delta^{-\epsilon}$-tight at scale $\Delta$ with data $(\mathcal{Q},\mathbb{T}) \subset \mathcal{D}_{\Delta} \times \mathcal{T}^{\Delta}$. The square $Q$ can be selected so that $\bar{\mu}^{Q}$ is $(\tau,\Delta^{-\epsilon})$-Frostman, and $\mathbb{T}$ can be selected so that $\sigma(\mathbb{T})$ is a non-empty $(\Delta,s,\Delta^{-\epsilon})$-set.
\end{thm}

\subsection{Preliminaries} Most of the work in the proof of Theorem \ref{t:renormalisation} has nothing to do with renormalisations and Frostman conditions. The main technical tool is Proposition \ref{prop6} below, which works for all (probability) measures. Proposition \ref{prop6} is modelled on ideas which appeared in the proof of \cite[Theorem 5.7]{2023arXiv230110199O}. Fortunately, we do not need all the components of the proof of \cite[Theorem 5.7]{2023arXiv230110199O}, so Proposition \ref{prop6} is somewhat simpler.

Below, the expression "a (positive) function $f$ is roughly constant" means that there exists $C > 0$ such that $C \leq f \leq 2C$. We also recall from \eqref{form81} that if $(\mu,\{\sigma_{x}\})$ is a configuration, we denote (for $\delta,\Delta \in 2^{-\N}$)
\begin{displaymath} (\mu \times \sigma_{x})(p \times T) = \int_{Q} \sigma_{x}(\{\theta \in [0,1] : \ell_{x,\theta} \subset T\}) \, d\mu(x), \quad p \in \mathcal{D}_{\delta}, \, T \in \mathcal{T}^{\Delta}. \end{displaymath}

\begin{lemma}\label{lemma1} Let $A \geq 1$, $\delta,\Delta \in 2^{-\N}$ with $\delta \leq \Delta$, and let $(\mu,\{\sigma_{x}\})$ be a configuration satisfying $\mu(\R^{2}) \in [\delta^{A},\delta^{-A}]$ and 
\begin{displaymath} \delta^{A}\mu(p) \leq (\mu \times \sigma_{x})(p \times [0,1]) \leq \mu(p)\delta^{-A}, \qquad p \in \mathcal{D}_{\delta}(\spt \mu). \end{displaymath}
Then, there exist $\mathcal{P} \subset \mathcal{D}_{\delta}$, and for each $p \in \mathcal{P}$ a family $\mathcal{T}_{p} \subset \mathcal{T}^{\Delta}$ with the following properties:
\begin{itemize}
\item[(B1) \phantomsection \label{B1}] $p \mapsto \mu(p)$ is roughly constant on $\mathcal{P}$, and $\mu(\cup \mathcal{P}) \gtrapprox_{\delta} \mu(\R^{2})$.
\item[(B2) \phantomsection \label{B2}] $(p,\mathbf{T}) \mapsto (\mu \times \sigma_{x})(p \times \mathbf{T})$ is roughly constant on $\{(p,\mathbf{T}) \in \mathcal{P} \times \mathcal{T}^{\Delta} : T \in \mathcal{T}_{p}\}$, and $(\mu \times \sigma_{x})(p \times (\cup \mathcal{T}_{p})) \approx_{\delta} (\mu \times \sigma_{x})(p \times [0,1])$ for all $p \in \mathcal{P}$.
 \end{itemize}
The implicit constants in "$\approx_{\delta}$" are here allowed to depend on $A$. \end{lemma}
 
 \begin{remark} The reader may think that $A = 1$, since we do not need any other cases. \end{remark}
 
 \begin{proof}[Proof of Lemma \ref{lemma1}] We start with property \nref{B1}, and then refine further to obtain also \nref{B2}. Let 
 \begin{displaymath} \mathcal{P}_{j} := \{p \in \mathcal{D}_{\delta}(\spt \mu) : 2^{-j - 1} < \mu(p) \leq 2^{-j}\}, \qquad j \in \Z. \end{displaymath}
 By the pigeonhole principle, choose an index $j \in \Z$ such that $\mu(\cup \mathcal{P}_{j}) \gtrapprox_{\delta} \mu(\R^{2})$. This is possible by the hypothesis $\mu(\R^{2}) \in [\delta^{A},\delta^{-A}]$. Now $p \mapsto \mu(p)$ is roughly constant on $\mathcal{P}_{j}$, with $\mu(p) \sim 2^{-j} =: m$.  Note that $m \in [c\delta^{A + 3},\delta^{-A}]$ for an absolute constant $c > 0$, as follows by combining $|\mathcal{P}_{j}| \leq \delta^{-2}$ and $\mu(\cup \mathcal{P}_{j}) \gtrapprox_{\delta} \delta^{A}$.
 
 Next, for $i \in \Z$ and $p \in \mathcal{P}_{j}$, let
 \begin{displaymath} \mathcal{T}_{p}^{i} := \{\mathbf{T} \in \mathcal{T}^{\Delta} : 2^{-i - 1} < (\mu \times \sigma_{x})(p \times \mathbf{T}) \leq 2^{-i}\}. \end{displaymath}
 By another application of the pigeonhole principle, pick $i = i(p) \in \Z$ such that 
 \begin{equation}\label{form61} (\mu \times \sigma_{x})(p \times (\cup \mathcal{T}_{p}^{i(p)})) \approx_{\delta} (\mu \times \sigma_{x})(p \times [0,1]), \qquad p \in \mathcal{P}_{j}. \end{equation}
 This is possible, because 
 \begin{equation}\label{form62} \delta^{-2A} \geq (\mu \times \sigma_{x})(p \times [0,1]) \geq \mu(p)\delta^{A}  \gtrsim \delta^{2A + 3}, \qquad p \in \mathcal{P}_{j}, \end{equation}
Finally, we want to remove the dependence of $i(p)$ on the choice of $p \in \mathcal{P}_{j}$. Note that $2^{-i(p)} \in [c\delta^{2A + 5},\delta^{-2A}]$ for each $p \in \mathcal{P}_{j}$, which follows by combining \eqref{form61}, \eqref{form62}, and $|\mathcal{T}_{p}^{i(p)}| \lesssim \Delta^{-1} \leq \delta^{-1}$. In other words, there are only $\approx_{\delta}$ possible choices for $i(p)$, and consequently at least one family 
\begin{displaymath} \mathcal{P}_{j}^{i} := \{p \in \mathcal{P}_{j} : i(p) = i\} \end{displaymath}
has to satisfy $\mu(\cup \mathcal{P}_{j}^{i}) \gtrapprox_{\delta} \mu(\cup \mathcal{P}_{j}) \gtrapprox_{\delta} \mu(\R^{2})$. Since the rough constancy of $p \mapsto \mu(p)$ remains true on $\mathcal{P} := \mathcal{P}_{j}^{i}$, the lemma is now valid setting $\mathcal{T}_{p} := \mathcal{T}_{p}^{i(p)}$ for $p \in \mathcal{P}$. \end{proof} 
 
The rough constancy in \nref{B2} will be used for obtaining a $(\Delta,s)$-set property for the family $\mathcal{T}_{p}$, provided that $(\mu \times \sigma_{x})(p \times \mathcal{A}(2,1)) \approx_{\delta} \mu(p)$. This is based on the following:
\begin{lemma}\label{lemma2a} Let $(\mu,\{\sigma_{x}\})$ be a configuration. Let $\Delta \in 2^{-\N}$, and let $p \subset \R^{2}$ be a dyadic cube. Assume that every $\sigma_{x}$ is an $(s,C)$-Frostman measure. Assume that $\mathcal{T} \subset \mathcal{T}^{\Delta}$ is a family such that $T \mapsto (\mu \times \sigma_{x})(p \times T)$ is roughly constant on $\mathcal{T}$. Then, $\mathcal{T}$ is a $(\Delta,s,\mathbf{C})$-set with constant
\begin{displaymath} \mathbf{C} \lesssim \frac{C\mu(p)}{(\mu \times \sigma_{x})(p \times (\cup \mathcal{T}))}. \end{displaymath}
\end{lemma} 

\begin{proof} Let $m > 0$ be such that $(\mu \times \sigma_{x})(p \times T) \sim m$ for all $T \in \mathcal{T}$. Fix $\Delta \leq \bar{\Delta} \leq 1$, and $\mathbf{T} \in \mathcal{T}^{\bar{\Delta}}$. The sets $\{(x,\theta) : \ell_{x,\theta} \subset T\}$ are disjoint for various $T \in \mathcal{T} \cap \mathbf{T}$. Therefore,
\begin{displaymath} m|\mathcal{T} \cap \mathbf{T}| \lesssim \mathop{\sum_{T \in \mathcal{T}}}_{T \subset \mathbf{T}} (\mu \times \sigma_{x})(p \times T) \leq \int_{p} \sigma_{x}(\{\theta : \ell_{x,\theta} \subset \mathbf{T}\}) \, d\mu(x) \lesssim C\mu(p)\bar{\Delta}^{s}. \end{displaymath}
When this inequality is combined with $m|\mathcal{T}| \sim (\mu \times \sigma_{x})(p \times (\cup \mathcal{T}))$, we find
\begin{displaymath} |\mathcal{T} \cap \mathbf{T}| \lesssim \frac{C\mu(p)}{(\mu \times \sigma_{x})(p \times (\cup \mathcal{T}))} \cdot \bar{\Delta}^{s}|\mathcal{T}|. \end{displaymath}
This completes the proof. \end{proof} 

We then arrive at the main proposition. The words "tightness" or "renormalisation" are not literally present in the statement (they will only appear in the eventual proof of Proposition \ref{t:renormalisation}). The reader may already view property (4) as a precedent for the tightness condition \nref{T4}. The scale $\underline{\Delta}$ in Proposition \ref{t:renormalisation} will eventually coincide with the scale $\bar{\delta} = \delta_{j + 1}$ of Proposition \ref{prop6}.

\begin{proposition}\label{prop6} For every $A \geq 1$ and $\epsilon \in (0,1]$, there exists $\Delta_{0} = \Delta_{0}(A,\epsilon) > 0$ such that the following holds for all $\Delta \in (0,\Delta_{0}]$ and $\delta \in 2^{-\N} \cap (0,\Delta]$. Let $(\mu,\{\sigma_{x}\})$ be a configuration, where $\mu(\R^{2}) \in [\Delta^{A},\Delta^{-A}]$, and $\sigma_{x}([0,1]) = 1$ for $\mu$ almost all $x \in \R^{2}$. Let $n := \ceil{2/\epsilon}$, and let $\{\delta_{j}\}_{j = 0}^{n} \subset 2^{-\N} \cap (0,\Delta]$ be the increasing scale sequence
\begin{displaymath} \delta_{j} := \Delta^{-j}\delta, \qquad j \in \{0,\ldots,n - 1\}. \end{displaymath}

Then, there exist consecutive scales $\underline{\delta} = \delta_{j}$ and $\bar{\delta} = \delta_{j + 1}$, a measure $\bar{\mu} := \mu|_{B}$, where $B \subset [0,1)^{2}$ is Borel (a union of elements in $\mathcal{D}_{\delta}$), and the following objects:
\begin{enumerate}
\item A family $\mathcal{Q} \subset \mathcal{D}_{\bar{\delta}}(\spt \mu)$ with $\bar{\mu}(\cup \mathcal{Q}) \gtrapprox_{\delta} \mu(\R^{2})$.
\item For each $Q \in \mathcal{Q}$ a family $\mathcal{P}_{Q} \subset \mathcal{D}_{\underline{\delta}}(Q)$ such that $\bar{\mu}(\cup \mathcal{P}_{Q}) \approx_{\delta} \bar{\mu}(Q)$.
\item For each $Q \in \mathcal{Q}$ a family $\mathbb{T}_{Q} \subset \mathcal{T}^{\Delta}$ of dyadic $\Delta$-tubes intersecting $Q$ such that 
\begin{displaymath} (\bar{\mu} \times \sigma_{x})(Q \times (\cup \mathbb{T}_{Q})) \approx_{\delta} \bar{\mu}(Q). \end{displaymath}
\item For each $p \in \mathcal{P}_{Q}$ a family $\mathcal{T}_{p} \subset \mathbb{T}_{Q}$ such that $|\mathbb{T}_{Q}| \lessapprox_{\delta} \Delta^{-\epsilon}|\mathcal{T}_{p}|$ for all $p \in \mathcal{P}_{Q}$.
\item For each $p \in \mathcal{P}_{Q}$ and $\mathbf{T} \in \mathcal{T}_{p}$, it holds $(\bar{\mu} \times \sigma_{x})(p \times \mathbf{T}) \gtrapprox_{\delta} \bar{\mu}(p)/|\mathcal{T}_{p}|$.
\end{enumerate}
The notation $f \lessapprox_{\delta} g$ means that $f \leq C(\log(\delta^{-1}))^{C}g$, where $C \geq 1$ may depend on $A,\epsilon$. 

Additionally, if each $\sigma_{x}$ is $(s,C)$-Frostman for some $C,s > 0$, then $\mathbb{T}_{Q}$ can be selected to be a $(\Delta,s,\mathbf{C})$-set with $\mathbf{C} \lessapprox_{\delta} C$.
\end{proposition}

\begin{remark} Note that (4) is only useful if $\delta$ and $\Delta$ are somewhat comparable, say $\delta \geq \Delta^{C}$ for a constant $C \geq 1$. In the application of the proposition, this will be the case.

 We use the following terminology in the proof.  If $(\mu,\{\sigma_{x}\})$ is a configuration, a \emph{sub-configuration} is any configuration $(\bar{\mu},\overline{\sigma}_{x})$, where $\mu$ is a restriction of $\mu$ to a Borel set, and each $\overline{\sigma}_{x}$ is a restriction of $\sigma_{x}$ to a Borel set. \end{remark}

\begin{proof}[Proof of Proposition \ref{prop6}] Recall that $n = \ceil{2/\epsilon}$. In the first part of the proof, we define a sequence of configurations $(\mu_{j},\{\sigma_{x}^{j}\})$, $0 \leq j \leq n - 1$, where $(\mu_{j + 1},\{\sigma_{x}^{j + 1}\})$ is always a sub-configuration of $(\mu_{j},\{\sigma_{x}^{j}\})$ for $0 \leq j \leq n - 2$.

To get started with this, we apply Lemma \ref{lemma1} once to the configuration $(\mu,\{\sigma_{x}\})$ at scales $\delta_{0} =\delta$ and $\Delta$. This produces the families $\mathcal{P}_{0} \subset \mathcal{D}_{\delta}$ and $\mathcal{T}^{0}_{p} \subset \mathcal{T}^{\Delta}$ for each $p \in \mathcal{P}$, satisfying the rough constancy conditions \nref{B1}-\nref{B2}. We define $\mu_{0}$ as the restriction of $\mu$ to $\cup \mathcal{P}_{0}$, and $\sigma_{x}^{0}$ as the restriction of $\sigma_{x}$ to $\mathcal{T}^{0}_{p}$, or more precisely to the set 
\begin{displaymath} \{\theta \in [0,1] : \ell_{x,\theta} \subset T \text{ for some } T \in \mathcal{T}_{p}^{0}\}, \end{displaymath}
whenever $x \in p \in \mathcal{P}$. This gives an initial sub-configuration $(\mu_{0},\{\sigma_{x}^{0}\})$.

We then assume that $0 \leq j \leq n - 2$, and the (sub-)configuration $(\mu_{j},\{\sigma_{x}^{j}\})$ satisfies the following properties relative to certain families $\mathcal{P}_{j} \subset \mathcal{D}_{\delta_{j}}$ and $\mathcal{T}_{p}^{j} \subset \mathcal{T}^{\Delta}$ for $p \in \mathcal{P}_{j}$:
\begin{enumerate}
\item[(I1) \phantomsection \label{I1}] $\mathcal{P}_{j} = \{p \in \mathcal{D}_{\delta_{j}} : \mu_{j}(p) > 0\}$, $p \mapsto \mu(p)$ is roughly constant on $\mathcal{P}_{j}$, and 
\begin{displaymath} \mu_{j}(\cup \mathcal{P}_{j}) \approx_{\delta} \mu(\R^{2}). \end{displaymath}
\item[(I2) \phantomsection \label{I2}] $\mathcal{T}_{p}^{j} = \{\mathbf{T} \in \mathcal{T}^{\Delta} : (\mu_{j} \times \sigma_{x}^{j})(p \times \mathbf{T}) > 0\}$ for all $p \in \mathcal{P}_{j}$, and 
\begin{equation}\label{form30} (\mu_{j} \times \sigma_{x}^{j})(p \times [0,1]) \approx_{\delta} \mu_{j}(p), \qquad p \in \mathcal{P}_{j}. \end{equation} 
\item[(I3) \phantomsection \label{I3}] $(p,\mathbf{T}) \mapsto (\mu_{j} \times \sigma_{x}^{j})(p \times \mathbf{T})$ is roughly constant on the family
\begin{displaymath} \mathcal{S}_{j} := \{(p,\mathbf{T}) \in \mathcal{P}_{j} \times \mathcal{T}^{\Delta} : \mathbf{T} \in \mathcal{T}_{p}^{j}\} =  \{(p,\mathbf{T}) \in \mathcal{D}_{\delta_{j}} \times \mathcal{T}^{\Delta} : (\mu_{j} \times \sigma_{x}^{j})(p \times \mathbf{T}) > 0\}. \end{displaymath} \end{enumerate}
For $j = 0$, the properties \nref{I1}-\nref{I3} follow from \nref{B1}-\nref{B2}: in \nref{I2}, the equation \eqref{form30} is based on \nref{B2}, and additionally the hypothesis that each $\sigma_{x}$ is a probability measure. 

To spell out the meaning of \nref{I3}, there exists $l_{j} > 0$ such that $(\mu_{j} \times \sigma_{x}^{j})(p \times \mathbf{T}) \in [l_{j},2l_{j}]$ whenever $(\mu_{j} \times \sigma_{x}^{j})(p \times \mathbf{T}) > 0$. As a consequence:
\begin{displaymath} l_{j}|\mathcal{T}_{p}^{j}| \sim (\mu_{j} \times \sigma_{x}^{j})(p \times [0,1]) \approx_{\delta} \mu_{j}(p), \qquad p \in \mathcal{P}_{j}. \end{displaymath}
Since also $p \mapsto \mu_{j}(p)$ is roughly constant on $\mathcal{P}_{j}$, say $\mu_{j}(p) \in [m_{j},2m_{j}]$, we infer that the cardinality
\begin{equation}\label{form24} |\mathcal{T}_{p}^{j}| \approx_{\delta} m_{j}/l_{j} =: M_{j} \end{equation}
is independent of $p$, at least up to a multiplicative error $\approx_{\delta} 1$. For ease of reference, we spell out the following rearrangement of the previous equation:
\begin{equation}\label{form27} (\mu_{j} \times \sigma_{x}^{j})(p \times \mathbf{T}) \sim l_{j} \sim \mu_{j}(p)/M_{j}, \qquad p \in \mathcal{P}_{j}, \, \mathbf{T} \in \mathcal{T}^{j}_{p}. \end{equation}

Recalling that $0 \leq j \leq n - 2$, we will next define families $\mathcal{P}_{j + 1} \subset \mathcal{D}_{\delta_{j + 1}}$ and $\mathcal{T}^{j + 1}_{Q} \subset \mathcal{T}^{\Delta}$ for each $Q \in \mathcal{P}_{j + 1}$ such that \nref{I1}-\nref{I3} hold at index $j + 1$. This is achieved by applying Lemma \ref{lemma1} to the configuration $(\mu_{j},\{\sigma_{x}^{j}\})$, at scales $\delta_{j + 1}$ and $\Delta$. The outcomes are:
\begin{enumerate}
\item A family $\mathcal{P}_{j + 1}$ such that $Q \mapsto \mu_{j}(Q)$ is roughly constant on $\mathcal{P}_{j + 1}$, and $\mu_{j}(\cup \mathcal{P}_{j + 1}) \gtrapprox_{\delta} \mu_{j}(\R^{2})$. We define $\mu_{j + 1}$ as the restriction of $\mu_{j}$ to $\cup \mathcal{P}_{j + 1}$. The measure $\mu_{j + 1}$ and the family $\mathcal{P}_{j + 1}$ then satisfy \nref{I1} at index $j + 1$.
\item For each $Q \in \mathcal{P}_{j + 1}$ a family $\mathcal{T}_{Q}^{j + 1} \subset \mathcal{T}^{\Delta}$ such that $(Q,\mathbf{T}) \mapsto (\mu_{j} \times \sigma_{x}^{j})(Q \times \mathbf{T})$ is roughly constant on $\{(Q,\mathbf{T}) \in \mathcal{P}_{j + 1} \times \mathcal{T}^{\Delta} : \mathbf{T} \in \mathcal{T}_{Q}^{j + 1}\}$, and moreover
\begin{equation}\label{form25} (\mu_{j} \times \sigma_{x}^{j})(Q \times (\cup \mathcal{T}^{j + 1}_{Q})) \gtrapprox_{\delta} (\mu_{j} \times \sigma_{x}^{j})(Q \times [0,1]). \end{equation} \end{enumerate}
For $x \in Q \in \mathcal{P}_{j + 1}$, we define $\sigma_{x}^{j + 1}$ to be the restriction of $\sigma_{x}^{j}$ to set $\{\theta \in [0,1] : \ell_{x,\theta} \subset \mathbf{T} \text{ for some } \mathbf{T} \in \mathcal{T}_{Q}^{j + 1}\}$. Then indeed 
\begin{displaymath} \mathcal{T}^{j + 1}_{Q} = \{\mathbf{T} \in \mathcal{T}^{\Delta} : (\mu_{j + 1} \times \sigma^{j + 1}_{x})(Q \times \mathbf{T}) > 0\}, \qquad Q \in \mathcal{P}_{j + 1}, \end{displaymath}
as required by \nref{I2}. With these definitions, $(Q \times \mathbf{T}) \mapsto (\mu_{j + 1} \times \sigma_{x}^{j + 1})(Q \times \mathbf{T})$ is roughly constant on the set
\begin{align*} \mathcal{S}_{j + 1} & := \{(Q,\mathbf{T}) \in \mathcal{P}_{j + 1} \times \mathcal{T}^{\Delta} : \mathbf{T} \in \mathcal{T}_{Q}^{j + 1}\}\\
& =  \{(Q,\mathbf{T}) \in \mathcal{D}_{\delta_{j + 1}} \times \mathcal{T}^{\Delta} : (\mu_{j + 1} \times \sigma_{x}^{j + 1})(Q \times \mathbf{T}) > 0\}, \end{align*}
as required by \nref{I3} at index $j + 1$.
It remains to verify that $(\mu_{j + 1} \times \sigma_{x}^{j + 1})(Q \times [0,1]) \approx_{\delta} \mu_{j + 1}(Q)$ for all $Q \in \mathcal{P}_{j + 1}$. Starting from \eqref{form25}, and recalling how $\mu_{j + 1}$ and $\sigma_{x}^{j + 1}$ were defined,
\begin{align*} (\mu_{j + 1} \times \sigma_{x}^{j + 1})(Q \times [0,1]) & \stackrel{\eqref{form25}}{\approx_{\delta}} (\mu_{j} \times \sigma_{x}^{j})(Q \times [0,1])\\
& \stackrel{\textup{\nref{I1}}}{=} \sum_{p \in \mathcal{P}_{j} \cap Q} (\mu_{j} \times \sigma_{x}^{j})(p \times [0,1])\\
& \stackrel{\textup{\nref{I2}}}{\approx_{\delta}} \sum_{p \in \mathcal{P}_{j} \cap Q} \mu_{j}(p) \stackrel{\textup{\nref{I1}}}{=} \mu_{j}(Q) = \mu_{j + 1}(Q). \end{align*} 
We have now verified that the configuration $(\mu_{j + 1},\{\sigma_{x}^{j + 1}\})$ satisfies \nref{I1}-\nref{I3}. In particular, the numbers $l_{j + 1} \sim (\mu_{j + 1} \times \sigma_{x}^{j + 1})(Q \times \mathbf{T})$ and $m_{j + 1} \sim \mu_{j + 1}(Q)$ and
\begin{displaymath} M_{j + 1} := m_{j + 1}/l_{j + 1} \approx_{\delta} |\mathcal{T}_{Q}^{j + 1}|, \qquad Q \in \mathcal{P}_{j + 1} \end{displaymath}
are well-defined. The "$\approx_{\delta}$" equation above is proven as in \eqref{form24}, since it was only based on the properties \nref{I1}-\nref{I3}.

This completes the inductive construction of the configurations $(\mu_{j},\{\sigma_{x}^{j}\})$, and the associated objects $\mathcal{P}_{j},\mathcal{T}_{p}^{j},l_{j},m_{j},M_{j}$ for all $0 \leq j \leq n - 1$.

Note that $1 \leq M_{j} \lessapprox_{\delta} \Delta^{-1}$ for all $0 \leq j \leq n - 1$ (since $M_{j} \approx_{\delta} |\mathcal{T}_{p}^{j}|$ by \eqref{form24}, and $\mathcal{T}_{p}^{j}$ is a family of dyadic $\Delta$-tubes intersecting the common $\delta_{j}$-square $p$, where $\delta_{j} \leq \Delta$). Based on this, and $n = \ceil{2/\epsilon}$, we claim that the sequence $M_{0},\ldots,M_{n - 1}$ contains a pair of consecutive elements $M_{j}, M_{j + 1}$ with $0 \leq j \leq n - 2$ such that
\begin{equation}\label{form21} M_{j + 1} \lessapprox_{\delta} \Delta^{-\epsilon}M_{j}. \end{equation}
Indeed, if this failed for all $0 \leq j \leq n - 2$, then 
\begin{displaymath} \Delta^{-1} \gtrapprox_{\delta} M_{n - 1} \gg_{\delta} \Delta^{-\epsilon (n - 1)}M_{0} \geq \Delta^{-1}. \end{displaymath}
Here $M_{n - 1} \gg_{\delta} \Delta^{-\epsilon n}M_{0}$ precisely means that $M_{n - 1} \geq C_{1}\log(1/\delta)^{C_{1}} \Delta^{-\epsilon (n - 1)}M_{0}$ for a suitable constant $C_{1} > 0$ depending on $A,\epsilon$, determined right below. To see that this leads to a contradiction, recall that the notation $A \lessapprox_{\delta} B$ means $A \leq C(\log(1/\delta))^{C} B$, where $C$ may depend on $A,\epsilon$. In particular, the inequalities $M_{n - 1} \lessapprox_{\delta} \Delta^{-1}$ and $M_{n - 1} \gg_{\delta} \Delta^{-1}$ are incompatible if the constant $C_{1} > 0$ is chosen large enough. Thus, there exists $0 \leq j \leq n - 2$ satisfying \eqref{form21}.

\begin{remark} The following information is not needed in the argument, so the reader may skip this remark. We remark that the converse of \eqref{form21}, namely $M_{j} \lesssim M_{j + 1}$, holds for all indices $0 \leq j \leq n - 2$. Indeed, note that $M_{j + 1}l_{j + 1} = m_{j + 1}$ by the definition of $M_{j + 1}$. Further, for $Q \in \mathcal{P}_{j + 1}$,
\begin{equation}\label{form31} M_{j + 1}l_{j + 1} = m_{j + 1} \sim \mu_{j + 1}(Q) = \mu_{j}(Q) = \sum_{p \in \mathcal{P}_{j} \cap Q} \mu_{j}(p) \sim \sum_{p \in \mathcal{P}_{j} \cap Q} M_{j}l_{j} \end{equation}
by the rough constancy of $p \mapsto \mu_{j}(p) \sim m_{j}$, and the definition of $M_{j}$. Finally, for $\mathbf{T} \in \mathcal{T}_{Q}^{j + 1}$ arbitrary,
\begin{displaymath} l_{j} \gtrsim (\mu_{j} \times \sigma_{x}^{j})(p \times \mathbf{T}) \geq (\mu_{j + 1} \times \sigma_{x}^{j + 1})(p \times \mathbf{T}), \qquad p \in \mathcal{P}_{j} \cap Q, \end{displaymath}
by the rough constancy of $(p,\mathbf{T}) \mapsto (\mu_{j} \times \sigma_{x}^{j})(p \times \mathbf{T})$. When the lower bound for $l_{j}$ is plugged back into \eqref{form31}, and using $\mu_{j + 1}|_{Q} = \mu_{j}|_{Q}$ for $Q \in \mathcal{P}_{j + 1}$, we find
\begin{displaymath} M_{j + 1}l_{j + 1} \gtrsim M_{j}\sum_{p \in \mathcal{P}_{j} \cap Q} (\mu_{j + 1} \times \sigma_{x}^{j + 1})(p \times \mathbf{T}) = M_{j}(\mu_{j + 1} \times \sigma_{x}^{j + 1})(Q \times \mathbf{T}). \end{displaymath}
Since $\mathbf{T} \in \mathcal{T}^{j + 1}_{Q}$, we again have $(\mu_{j + 1} \times \sigma_{x}^{j + 1})(Q \times \mathbf{T}) \sim l_{j + 1}$, hence $M_{j + 1} \gtrsim M_{j}$. \end{remark}

For any index $j \in \{0,\ldots,n - 2\}$ satisfying \eqref{form21}, write 
\begin{displaymath} \underline{\delta} := \delta_{j}, \quad \bar{\delta} := \delta_{j + 1}, \quad \text{and} \quad \bar{\mu} := \mu_{j}. \end{displaymath}
Define also $\mathcal{Q} := \mathcal{P}_{j + 1}$. Then Proposition \ref{prop6}(1) is satisfied.

For $Q \in \mathcal{Q}$, we set $\mathbb{T}_{Q} := \mathcal{T}_{Q}^{j + 1}$. Then Proposition \ref{prop6}(3) is satisfied by \eqref{form30} (with index $j + 1$), and since $\bar{\mu}(Q) = \mu_{j + 1}(Q)$ for $Q \in \mathcal{Q}$. Before proceeding with the other "main" claims (2), (4)-(5), we verify the $(\Delta,s,\mathbf{C})$-set property of the families $\mathbb{T}_{Q}$ under the assumption that each $\sigma_{x}$ satisfies an $s$-dimensional Frostman condition. Indeed, then $(\mu_{j + 1},\{\sigma_{x}^{j + 1}\})$ is a configuration with the properties that each $\sigma_{x}^{j + 1}$ is $(s,C)$-Frostman, and $\mathbf{T} \mapsto (\mu_{j + 1} \times \sigma_{x}^{j + 1})(Q \times \mathbf{T})$ is roughly constant on $\mathbb{T}_{Q}$. We may now infer from Lemma \ref{lemma2a} that $\mathbb{T}_{Q}$ is a $(\Delta,s,\mathbf{C})$-set with
\begin{displaymath} \mathbf{C} \lesssim \frac{C\mu_{j + 1}(Q)}{(\mu_{j + 1} \times \sigma_{x}^{j + 1})(Q \times (\cup \mathbb{T}_{Q}))} \stackrel{\textup{\nref{I2}}}{=} \frac{C\mu_{j + 1}(Q)}{(\mu_{j + 1} \times \sigma_{x}^{j + 1})(Q \times [0,1])}  \stackrel{\textup{\nref{I2}}}{\lessapprox_{\delta}} C. \end{displaymath}
This is what we claimed.

Regarding properties (2), (4)-(5), the families $\mathcal{P}_{Q} \subset \mathcal{D}_{\underline{\delta}}(Q)$ and $\mathcal{T}_{p} \subset \mathcal{T}^{\Delta}$ could almost be taken to be $\mathcal{P}_{Q} := \mathcal{P}_{j} \cap Q = \{p \in \mathcal{P}_{j} : p \subset Q\}$ and $\mathcal{T}_{p} := \mathcal{T}_{p}^{j}$, but the inclusion $\mathcal{T}_{p} \subset \mathbb{T}_{Q}$ ($= \mathcal{T}_{Q}^{j + 1}$) is not guaranteed. We need a slight refinement to fix this.

Note that
\begin{displaymath} (\bar{\mu} \times \sigma_{x}^{j})(Q \times (\cup \mathbb{T}_{Q})) \geq (\mu_{j + 1} \times \sigma_{x}^{j + 1})(Q \times (\cup \mathbb{T}_{Q})) \stackrel{\textup{\nref{I2}}}{\gtrapprox_{\delta}} \bar{\mu}(Q), \end{displaymath}
since $\mu_{j + 1}|_{Q} = \bar{\mu}|_{Q}$. It follows that there exists a family $\mathcal{P}_{Q} \subset \mathcal{P}_{j} \cap Q$ such that $\bar{\mu}(\cup \mathcal{P}_{Q}) \approx_{\delta} \bar{\mu}(Q)$, and 
\begin{equation}\label{form28} (\bar{\mu} \times \sigma_{x})(p \times (\cup \mathbb{T}_{Q})) \geq (\bar{\mu} \times \sigma_{x}^{j})(p \times (\cup \mathbb{T}_{Q})) \gtrapprox_{\delta} \bar{\mu}(p), \qquad p \in \mathcal{P}_{Q}. \end{equation}
Now Proposition \ref{prop6}(2) has been verified. 

For $p \in \mathcal{P}_{Q}$, set $\mathcal{T}_{p} := \mathcal{T}^{j}_{p} \cap \mathbb{T}_{Q} \subset \mathbb{T}_{Q}$. To verify Proposition \ref{prop6}(4), we need to check that $|\mathbb{T}_{Q}| \lessapprox_{\delta} \Delta^{-\epsilon}|\mathcal{T}_{p}|$ for all $p \in \mathcal{P}_{Q}$. Since $(\bar{\mu} \times \sigma^{j}_{x})(p \times \mathbf{T}) \lesssim \bar{\mu}(p)/M_{j}$ for all $\mathbf{T} \in \mathcal{T}_{p} \subset \mathcal{T}_{p}^{j}$ by the rough constancy of $\mathbf{T} \mapsto (\mu_{j} \times \sigma_{x}^{j})(p \times \mathbf{T}) = (\bar{\mu} \times \sigma^{j}_{x})(p \times \mathbf{T})$ (see \eqref{form27}),
\begin{displaymath} \bar{\mu}(p) \stackrel{\eqref{form28}}{\lessapprox_{\delta}} (\bar{\mu} \times \sigma^{j}_{x})(p \times (\cup \mathbb{T}_{Q})) \lesssim |\mathcal{T}_{p}| \cdot \bar{\mu}(p)/M_{j}, \qquad p \in \mathcal{P}_{Q}. \end{displaymath}
Therefore, 
\begin{equation}\label{form29} |\mathcal{T}_{p}| \gtrapprox_{\delta} M_{j} \stackrel{\eqref{form21}}{\gtrapprox_{\delta}} \Delta^{\epsilon}M_{j + 1} \approx_{\delta} \Delta^{\epsilon}|\mathbb{T}_{Q}|, \qquad p \in \mathcal{P}_{Q}. \end{equation}
This verifies Proposition \ref{prop6}(4). Finally, the estimate in Proposition \ref{prop6}(5), namely
\begin{displaymath} (\bar{\mu} \times \sigma_{x})(p \times \mathbf{T}) \gtrapprox_{\delta} \bar{\mu}(p)/|\mathcal{T}_{p}|, \qquad p \in \mathcal{P}_{Q}, \, \mathbf{T} \in \mathcal{T}_{p}, \end{displaymath}
follows from \eqref{form27} and $|\mathcal{T}_{p}| \gtrapprox_{\delta} M_{j}$, as observed in \eqref{form29}. This completes the proof.  \end{proof}

   
\subsection{Proof of Theorem \ref{t:renormalisation}} In this section we deduce Theorem \ref{t:renormalisation} from Proposition \ref{prop6}. We also need a few lesser auxiliary definitions and results, discussed below. 

Recall that $S_{Q} \colon \R^{2} \to \R^{2}$ is the homothety which maps $Q \to [0,1)^{2}$. 

\begin{lemma}\label{lemma3} Let $\Delta \in 2^{-\N}$, and $\underline{\delta},\bar{\delta} \in 2^{-\N}$ with $\underline{\delta} = \Delta \bar{\delta}$. Let $Q \in \mathcal{D}_{\bar{\delta}}$ and $p \in \mathcal{D}_{\underline{\delta}}(Q)$. Let $\mathbf{T} \in \mathcal{T}^{\Delta}$, and let $\mathcal{L}(p,\mathbf{T}) := \{\ell \in \mathcal{A}(2,1) : \ell \cap p \neq \emptyset \text{ and } \ell \subset \mathbf{T}\}$. Then, $S_{Q}(\mathcal{L}(p,\mathbf{T}))$ can be covered by $C \lesssim 1$ tubes $\mathbf{T}_{1},\ldots,\mathbf{T}_{C} \in \mathcal{T}^{\Delta}$ such that $|\sigma(\mathbf{T}_{j}) - \sigma(\mathbf{T})| \lesssim \Delta$ for all $1 \leq j \leq C$.
\end{lemma} 

\begin{proof} Since all lines $\ell = \{(x,y) : y = ax + b\} \in \mathcal{L}(p,\mathbf{T})$ are contained in $\mathbf{T}$, their slopes "$a$" are contained on some interval $I = I_{\mathbf{T}} \in \mathcal{D}_{\Delta}([-1,1))$. The map $S_{Q}$ does not affect slopes, so the same remains true for $S_{Q}(\mathcal{L}(p,\mathbf{T}))$. Moreover,  all the lines in $S_{Q}(\mathcal{L}(p,\mathbf{T}))$ intersect $S_{Q}(p) \in \mathcal{D}_{\underline{\delta}/\bar{\delta}} = \mathcal{D}_{\Delta}$. Thus, $S_{Q}(\mathcal{L}(p,\mathbf{T}))$ consists of lines with (i) slopes in $I \subset [-1,1)$, and (ii) all intersecting a fixed element of $\mathcal{D}_{\Delta}$. We claim that any such line family can be covered by $\lesssim 1$ dyadic $\Delta$-tubes.

Indeed, let $\ell = \{(x,y) : y = ax + b\}$ and $\ell' = \{(x,y) : y = a'x + b'\}$ be elements of $S_{Q}(\mathcal{L}(p,\mathbf{T}))$ (or any line family with properties (i)-(ii)). Then $|a' - a| \leq \Delta$ by hypothesis (i). Let $q \in \mathcal{D}_{\Delta}$ be as in the hypothesis (ii), and let $(x,y) \in \ell \cap q$ and $(x',y') \in \ell' \cap q$. Then, $|x - x'| \leq \Delta$ and $|y - y'| \leq \Delta$, so $|b - b'| = |(y - ax) - (y' - a'x')| \lesssim \Delta$. Thus, $|(a,b) - (a',b')| \lesssim \Delta$, and the claim follows. \end{proof} 

We will need the notion of \emph{$\{\Delta_j\}_{j=0}^{n - 1}$-uniform sets}. For a more extensive introduction to $\{\Delta_j\}_{j=0}^{n - 1}$-uniform sets, see \cite[Section 2.3]{OS23}. 

\begin{definition}\label{def:uniformity}
Let $n \geq 1$, and let $\{\Delta_{j}\}_{j = 0}^{n} \subset 2^{-\N}$ with
\begin{displaymath} \delta := \Delta_{n} < \Delta_{n - 1} < \ldots < \Delta_{1} \leq \Delta_{0} := 1. \end{displaymath}
A set $\mathcal{P} \subset \mathcal{D}_{\delta}$ is called \emph{$\{\Delta_j\}_{j=0}^{n - 1}$-uniform} if there is a sequence $\{N_j\}_{j=0}^{n - 1}$ such that $N_{j} \in 2^{\N}$ and $|\mathcal{P} \cap Q|_{\Delta_{j + 1}} = N_j$ for all $j\in \{0,\ldots,n - 1\}$ and all $Q\in\mathcal{D}_{\Delta_{j}}(\mathcal{P})$. \end{definition}

A key feature of uniform sets is, roughly, that every set $\mathcal{P} \subset \mathcal{D}_{\delta}$ contains uniform subsets of large cardinality. The following variant of the principle is \cite[Lemma 3.6]{Sh}:

\begin{lemma}\label{l:uniformSubsets} Let $\Delta = 2^{-T}$, $T \in \N$, and let $\delta = \Delta^{n} = 2^{-nT}$ for some $n \in \N$. Let $\mathcal{P} \subset \mathcal{D}_{\delta}$. Then, there exists a $\{\Delta^{j}\}_{j = 0}^{n - 1}$-uniform set $\overline{\mathcal{P}} \subset \mathcal{P}$ with 
\begin{displaymath} |\overline{\mathcal{P}}| \geq (2T)^{-n}|\mathcal{P}|. \end{displaymath}  \end{lemma}

Given a uniform $(\delta,t)$-set $\mathcal{P} \subset \mathcal{D}_{\delta}$, and $\tau \in [0,t)$, the next lemma allows us to find many intermediate scales $\Delta \in [\delta,1]$ such that the renormalised set $\mathcal{P}^{Q} \subset \mathcal{D}_{\delta/\Delta}$ is a $(\delta,\tau)$-set. The result with "one scale" instead of many is (roughly) \cite[Corollary 2.12]{OS23}.

\begin{lemma}\label{lemma5} Let $t \in (0,d]$, $\tau \in (0,t)$, $\gamma \in (0,1]$, $\epsilon \in (0,\tfrac{1}{8d}\gamma(t - \tau)^{2}]$, and $\Delta \in 2^{-\N}$. Then, there exist $n_{0} = n_{0}(d,\epsilon) \in \N$ and $\mathbf{A} \lesssim_{d} \Delta^{-3d}$, such that the following holds for all $n \geq n_{0}$.

Let $\delta := \Delta^{n}$, and let $\mathcal{P} \subset \mathcal{D}_{\delta}([0,1)^d)$ be a $\{\Delta^{j}\}_{j = 0}^{n - 1}$-uniform $(\delta,t,\delta^{-\epsilon})$-set. Then, there exists
\begin{equation}\label{form64} \mathcal{G} \subset \{0,\ldots,\gamma n\} \quad \text{ with } \quad |\mathcal{G}| \geq n \cdot \gamma(t - \tau)^{2}/(10 d^{2}) \end{equation}
such that $\mathcal{P}^{Q} = S_{Q}(\mathcal{P} \cap Q) \subset \mathcal{D}_{\delta/\Delta^{j}}$ is a $(\delta/\Delta^{j},\tau,\mathbf{A})$-set for all $j \in \mathcal{G}$, and $Q \in \mathcal{D}_{\Delta^{j}}(\mathcal{P})$. \end{lemma}

The proof is postponed to Appendix \ref{appA}. We are then prepared to prove Theorem \ref{t:renormalisation}. We repeat the statement:

\begin{thm}\label{t:renormalisation2} Let $s,t \in (0,2]$, $\tau \in (0,t)$, $\mathbf{C} > 0$, and $\epsilon > 0$. Then, there exist $\Delta_{0} = \Delta_{0}(\mathbf{C},\epsilon,t,\tau) > 0$ and $n = n(\epsilon,t,\tau) \in \N$ such that the following holds for all $\Delta_{1} \in 2^{-\N} \cap (0,\Delta_{0}]$. Let $(\mu,\{\sigma_{x}\})$ be a configuration, where $\mu$ is a $(t,\mathbf{C})$-Frostman probability measure, and $\sigma_{x}$ is an $(s,\mathbf{C})$-Frostman probability measure for $\mu$ almost all $x \in \R^{2}$.

Then, there exist 
\begin{itemize}
\item dyadic scales $\underline{\Delta},\Delta \in [\Delta_{1}^{n},\Delta_{1}]$ with $\underline{\Delta} \leq \Delta$
\item a measure $\bar{\mu} = \mu|_{B}$, where $B \subset [0,1)^{2}$ is Borel, and 
\item a square $Q \in \mathcal{D}_{\underline{\Delta}}$
\end{itemize}
such that the $Q$-renormalised configuration $(\bar{\mu}^{Q},\{\sigma^{Q}_{y}\})$ is $\Delta^{-\epsilon}$-tight at scale $\Delta$ with data $(\mathcal{Q},\mathbb{T}) \subset \mathcal{D}_{\Delta} \times \mathcal{T}^{\Delta}$. The square $Q$ can be selected so that $\bar{\mu}^{Q}$ is $(\tau,\Delta^{-\epsilon})$-Frostman, and $\mathbb{T}$ can be selected so that $\sigma(\mathbb{T})$ is a non-empty $(\Delta,s,\Delta^{-\epsilon})$-set. \end{thm}

\begin{remark}  The tightness part of the statement does not use the $t$-Frostman property of $\mu$ or the $s$-Frostman property of $\sigma_{x}$. These hypotheses are only needed to ensure the corresponding properties for $\bar{\mu}^{Q}$ and $\sigma(\mathbb{T})$. \end{remark} 
 
 \begin{proof}[Proof of Theorem \ref{t:renormalisation2}] The proof is divided into two steps, called \emph{Initial uniformisation of $\mu$} and the the \emph{Main argument}.

\subsubsection*{Initial uniformisation of $\mu$} Fix $\epsilon,t,\tau$ as in the statement. We may assume that 
\begin{equation}\label{form85} 0 < \epsilon \leq \tfrac{1}{160}(t - \tau)^{3}, \end{equation}
since Theorem \ref{t:renormalisation2} with a smaller $\epsilon$ implies Theorem \ref{t:renormalisation2} with a larger $\epsilon$. Throughout this proof, the notation $A \lessapprox_{\Delta} B$ means that $A \leq C(\log(\Delta^{-1}))^{C} B$, where $C$ is allowed to depend on $\epsilon,t,\tau$. Let
\begin{equation}\label{def:rho} \rho := (t - \tau)^{3}/10^{4}. \end{equation} 
Then, let $n = n(\epsilon,t,\tau) \in \N$ be so large that every set $\mathcal{G} \subset \{0,\ldots,n\}$ with $|\mathcal{G}| \geq \tfrac{1}{6}\epsilon \rho n$ contains an arithmetic progression $\mathcal{A}$ of length $|\mathcal{A}| \geq \ceil{2/\epsilon} + 1$. The number $n$ exists by Szemer\'edi's theorem \cite{MR369312}.

Let $\Delta_{0} := \Delta_{0}(C,\epsilon,t,\tau) \in 2^{-\N}$ be a scale to be determined during the proof, and let $\Delta_{1} \in (0,\Delta_{0}]$. We require explicitly that $\Delta_{0}$, hence $\Delta_{1}$, is smaller than the threshold in Proposition \ref{prop6} is applied with parameters $A := 1$ and $\epsilon$, and that $C \leq \Delta_{0}^{-\epsilon/2}$. Additionally, we will need that $\mathbf{C} \leq \Delta_{0}^{-\epsilon}$, and
\begin{displaymath} C \lessapprox_{\Delta_{0}} 1 \quad \Longrightarrow \quad C \leq \Delta_{0}^{-\epsilon}, \end{displaymath} 
where we recall that the implicit constants in the "$\approx$" notation may depend on $\epsilon,t,\tau$.

Let $\bar{\delta} := \Delta_{1}^{n}$, and $\bar{\delta}_{j} := \Delta_{1}^{-j}\bar{\delta}$ for $0 \leq j \leq n$. The scale $\Delta$ will eventually have the form $\Delta = \Delta_{1}^{m}$ for some $2/\epsilon \leq m \leq n$, so all the notations $\approx_{\Delta}$, $\approx_{\bar{\delta}}$, and $\approx_{\Delta_{1}}$ are equivalent. This will be crucial when applying Proposition \ref{prop6}(4) in the \emph{Main argument} part of the proof. We will use the notation $\approx_{\Delta}$ from now on, even though $\Delta$ will only be fixed at \eqref{form79}.

By the pigeonhole principle (or Lemma \ref{lemma1}), find a family $\mathcal{P} \subset \mathcal{D}_{\bar{\delta}}$ such that $p \mapsto \mu(p)$ is roughly constant on $\mathcal{P}$, say $\mu(p) \sim \mathfrak{m}$ for all $p \in \mathcal{P}$, and $\mu(\cup \mathcal{P}) \gtrapprox_{\Delta} \mu(\R^{2}) = 1$.

Next, using Lemma \ref{l:uniformSubsets}, locate a $\{\bar{\delta}_{j}\}_{j = 1}^{n}$-uniform subset $\overline{\mathcal{P}} \subset \mathcal{P}$ with 
\begin{displaymath} |\overline{\mathcal{P}}| \geq (2\log(\tfrac{1}{\Delta_{1}}))^{-n}|\mathcal{P}| \gtrapprox_{\Delta} |\mathcal{P}|. \end{displaymath}
(It is useful to note that the scale sequence $\{\bar{\delta}_{j}\}$ is increasing in $j$, with $\bar{\delta}_{n} = 1$, whereas the scale sequence in Definition \ref{def:uniformity} is decreasing, with $\Delta_{0} = 1$. So, to be accurate, we should write that $\overline{\mathcal{P}}$ is $\{\bar{\delta}_{j}\}_{j = n}^{1}$-uniform. This remark will explain the difference between \eqref{form63} and \eqref{form64}.) We record that $\overline{\mathcal{P}}$ is a $(\bar{\delta},t,\mathbf{C}')$-set with $\mathbf{C}' \lessapprox_{\Delta} \mathbf{C}$. This follows from the inequality
\begin{equation}\label{form71} |\overline{\mathcal{P}} \cap Q| \cdot \mathfrak{m} \lesssim \mu(Q) \lesssim \mathbf{C}r^{t}, \qquad Q \in \mathcal{D}_{r}, \, \bar{\delta} \leq r \leq 1, \end{equation}
valid by the rough constancy of $p \mapsto \mu(p)$ on $\overline{\mathcal{P}} \subset \mathcal{P}$, and since
\begin{displaymath} |\overline{\mathcal{P}}| \cdot \mathfrak{m} \sim \mu(\cup \overline{\mathcal{P}}) \gtrapprox_{\Delta} \mu(\cup \mathcal{P}) \gtrapprox_{\Delta} 1, \end{displaymath}
using again (multiple times) the rough constancy of $p \mapsto \mu(p)$.

We have now established that $\overline{\mathcal{P}}$ is a $\{\bar{\delta}_{j}\}_{j = 1}^{n}$-uniform $(\bar{\delta},t,\mathbf{C}')$-set with $\mathbf{C}' \lessapprox_{\Delta} \mathbf{C}$. In particular, $\overline{\mathcal{P}}$ is a $(\bar{\delta},t,\bar{\delta}^{-\epsilon})$-set for $\bar{\delta} > 0$ sufficiently small (i.e. $\Delta_{0} > 0$ sufficiently small depending on $\mathbf{C},\epsilon$). According to Lemma \ref{lemma5} applied with parameter $\gamma := \tfrac{1}{10}(t - \tau)$ (this is legitimate by the constraint \eqref{form85} on $\epsilon$), there exists a set of indices 
\begin{equation}\label{form63} \mathcal{G} \subset \{n - \tfrac{1}{10}(t - \tau)n,\ldots,n\} \quad \text{with} \quad |\mathcal{G}| \geq \rho n \end{equation}
such that $\overline{\mathcal{P}}^{Q} \subset \mathcal{D}_{\bar{\delta}/\bar{\delta}_{j}}$ is a $(\bar{\delta}/\bar{\delta}_{j},\tau,O(\Delta_{1}^{-6}))$-set for all $j \in \mathcal{G}$, and all $Q \in \mathcal{D}_{\bar{\delta}_{j}}(\overline{\mathcal{P}})$. Recall from \eqref{def:rho} that $\rho = (t - \tau)^{3}/10^{4}$. We remark in the passing that
\begin{equation}\label{form80} \bar{\delta}_{j} = \Delta_{1}^{n - j} \geq \Delta_{1}^{(t - \tau)n/10} = \bar{\delta}^{(t - \tau)/10}, \qquad j \in \mathcal{G}. \end{equation}

\begin{remark} The appearance of the constant $O(\Delta_{1}^{-6})$ is a technical problem which led us to applying Szemer\'edi's theorem. If this constant could be taken to be of the order $O(\Delta_{1}^{-\epsilon})$, we could simply set $\Delta := \Delta_{1}$. Using Szemer\'edi's theorem to fix this issue seems, at the same time, convenient, and rather too complicated. \end{remark} 

Pick arbitrarily a $(6/\epsilon)$-separated subset of $\mathcal{G}' \subset \mathcal{G}$ of cardinality $|\mathcal{G}'| \geq \tfrac{1}{6}\epsilon \rho n$. By the choice of $n = n(\epsilon,t,\tau)$ and Szemer\'edi's theorem, there exists an arithmetic progression $\mathcal{A} \subset \mathcal{G}'$ of cardinality $|\mathcal{A}| \geq \ceil{2/\epsilon} + 1$. This progression has the form 
\begin{displaymath} \mathcal{A} = \{a_{0},a_{0}+ m,a_{0} + 2m,\ldots,a_{0} + |\mathcal{A}|m\} \end{displaymath}
for some $a_{0} \in \{0,\ldots,n\}$ and $6/\epsilon \leq m \leq n$. We define 
\begin{equation}\label{form79} \Delta := \Delta_{1}^{m} \leq \Delta_{1}^{6/\epsilon}. \end{equation} 
In particular, for $j \in \mathcal{G}$ and $Q \in \mathcal{D}_{\bar{\delta}_{j}}(\overline{\mathcal{P}})$, the renormalisation $\overline{\mathcal{P}}^{Q}$ is a $(\bar{\delta}/\bar{\delta}_{j},\tau,\Delta^{-\epsilon})$-set.

Let $\mu_{\mathrm{uni}}$ be the restriction of $\mu$ to $\cup \overline{\mathcal{P}}$. Then $\mu_{\mathrm{uni}}(\R^{2}) \approx_{\Delta} 1$. We claim that $\mu_{\mathrm{uni}}^{Q}$ is $(\tau,\Delta^{-\epsilon})$-Frostman for all $j \in \mathcal{G}$, and $Q \in \mathcal{D}_{\bar{\delta}_{j}}(\spt \mu_{\mathrm{uni}}) = \mathcal{D}_{\bar{\delta}_{j}}(\overline{\mathcal{P}})$.

For radii $r \geq \bar{\delta}/\bar{\delta}_{j}$, this is based on the $(\bar{\delta}/\bar{\delta}_{j},\tau,\Delta^{-\epsilon})$-set property of $\overline{\mathcal{P}}^{Q}$: for $x \in \R^{2}$ and $y = S_{Q}(x)$,
\begin{align*} \mu_{\mathrm{uni}}^{Q}(B(x,r)) & \stackrel{\mathrm{def.}}{=} \mu_{\mathrm{uni}}(Q)^{-1}\mu_{\mathrm{uni}}(B(y,\bar{\delta}_{j}r) \cap Q)\\
& \lesssim \mu_{\mathrm{uni}}(Q)^{-1}\mathfrak{m} \cdot |\overline{\mathcal{P}} \cap B(y,2\bar{\delta}_{j}r) \cap Q|\\
& = \mu_{\mathrm{uni}}(Q)^{-1}\mathfrak{m} \cdot |\overline{\mathcal{P}}^{Q} \cap B(x,2r)|\\
& \lesssim \Delta^{-\epsilon}\mu_{\mathrm{uni}}(Q)^{-1}\mathfrak{m} \cdot r^{\tau}|\overline{\mathcal{P}}^{Q}|\\
& = \Delta^{-\epsilon} \mu_{\mathrm{uni}}(Q)^{-1}\mathfrak{m} \cdot r^{\tau}|\overline{\mathcal{P}} \cap Q| \sim \Delta^{-\epsilon}r^{\tau}. \end{align*}
This estimate no longer works in the range $r \leq \bar{\delta}/\bar{\delta}_{j}$, since the $(\bar{\delta}/\bar{\delta}_{j},\tau)$-set property of $\overline{\mathcal{P}}^{Q}$ says nothing about such radii. In that range, we instead use that $\bar{\delta}_{j} \geq \bar{\delta}^{(\tau - t)/10}$ for $j \in \mathcal{G}$, as recorded in \eqref{form80}, so in particular 
\begin{displaymath} r^{t - \tau} \leq (\bar{\delta}/\bar{\delta}_{j})^{t - \tau} \leq \bar{\delta}_{j}^{2}. \end{displaymath}
Making crude estimates such as $\mu_{\mathrm{uni}}(Q) \gtrapprox_{\Delta} \bar{\delta}_{j}^{2}$ (since $Q \mapsto \mu_{\mathrm{uni}}(Q)$ is roughly constant on $\mathcal{D}_{\bar{\delta}_{j}}(\overline{\mathcal{P}})$, and $\mu_{\mathrm{uni}}(\cup \overline{\mathcal{P}}) \gtrapprox_{\Delta} 1$), and using the $(t,\mathbf{C})$-Frostman property of $\mu$ directly, we obtain (again for $x \in \R^{2}$ and $y = S_{Q}(x)$)
\begin{displaymath} \mu_{\mathrm{uni}}^{Q}(B(x,r)) = \mu_{\mathrm{uni}}(Q)^{-1}\mu_{\mathrm{uni}}(B(y,\delta_{j}r)) \lessapprox_{\Delta} \mathbf{C}\bar{\delta}_{j}^{-2}r^{t} \leq \mathbf{C}r^{\tau}. \end{displaymath}
This completes the proof of the $(\tau,\Delta^{-\epsilon})$-Frostman property of $\mu_{\mathrm{uni}}^{Q}$.

\subsubsection*{Main argument} Here is where the arithmetic progression $\mathcal{A} \subset \mathcal{G}$ is used: for Proposition \ref{prop6}, we need a long increasing scale sequence with constant ratios between consecutive scales, and $\mathcal{A}$ determines such a sequence. We set $\delta := \Delta_{1}^{n - a_{0}}$ and
\begin{displaymath} \delta_{j} := \Delta^{-j}\delta, \qquad 0 \leq j \leq |\mathcal{A}|. \end{displaymath}
Since $\mathcal{A} \subset \mathcal{G}$, all the scales $\delta_{j}$ are contained in $\{\bar{\delta}_{i} : i \in \mathcal{G}\}$. Recall that $|\mathcal{A}| \geq \ceil{2/\epsilon} + 1$. In Proposition \ref{prop6}, it was also assumed that $\{\delta_{j}\} \subset (0,\Delta]$: since $\{\delta_{j}\}_{j = 0}^{|\mathcal{A}|} \subset (0,1]$, it holds $\{\delta_{j}\}_{j = 0}^{m} \subset (0,\Delta]$ with $m := |\mathcal{A}| - 1 \geq \ceil{2/\epsilon}$.

 Let 
 \begin{displaymath} \delta_{j} < \delta_{j + 1} =: \underline{\Delta} \end{displaymath}
 be the special scales provided by Proposition \ref{prop6} applied with parameters $A = 1$ and $\epsilon$ to the configuration $(\mu_{\mathrm{uni}},\{\sigma_{x}\})$ and the sequence $\{\delta_{j}\}_{j = 0}^{m}$. (Note that $\mu_{\mathrm{uni}}(\R^{2}) \gtrapprox_{\Delta} 1$, so certainly $\mu_{\mathrm{uni}}(\R^{2}) \geq \Delta^{A}$ for $\Delta > 0$ sufficiently small.) Let also $\bar{\mu} = (\mu_{\mathrm{uni}})|_{B}$ be the measure given by Proposition \ref{prop6}, and let $\mathcal{Q}_{\mathrm{uni}} \subset \mathcal{D}_{\underline{\Delta}}(\spt \mu_{\mathrm{uni}})$ be the collection from Proposition \ref{prop6}(1). Pick a square $Q \in \mathcal{Q}_{\mathrm{uni}}$ with 
\begin{equation}\label{form57} \bar{\mu}(Q) \gtrapprox_{\Delta} \mu_{\mathrm{uni}}(Q). \end{equation}
This is possible by Proposition \ref{prop6}(1). The square $Q$ selected here will remain fixed for the remainder of the proof. We checked in the first part of the proof that $\mu_{\mathrm{uni}}^{Q}$ is $(\tau,\Delta^{-\epsilon})$-Frostman (recall: $\mathcal{A} \subset \mathcal{G}$), and it now follows from \eqref{form57} that also $\bar{\mu}^{Q}$ is $(\tau,\Delta^{-2\epsilon})$-Frostman: indeed $\bar{\mu}^{Q}(B) \lessapprox_{\Delta} \mu_{\mathrm{uni}}^{Q}(B)$ for all Borel sets $B \subset \R^{2}$. 

 We next claim that the $Q$-renormalised configuration $(\bar{\mu}^{Q},\{\sigma_{x}^{Q}\})$ is $\Delta^{-2\epsilon}$-tight, provided $\Delta > 0$ is small enough. Recall that $\bar{\mu}(\cup \mathcal{P}_{Q}) \gtrapprox_{\Delta} \bar{\mu}(Q)$ by Proposition \ref{prop6}(2). Let
\begin{displaymath} \mathcal{Q} := \{S_{Q}(p) : p \in \mathcal{P}_{Q}\} \subset \mathcal{D}_{\Delta}. \end{displaymath}
Then, $\bar{\mu}^{Q}(\cup \mathcal{Q}) = \tfrac{1}{\bar{\mu}(Q)}\bar{\mu}(\cup \mathcal{P}_{Q}) \gtrapprox_{\Delta} 1$. This verifies the tightness condition \nref{T2}. We then proceed to define the families $\mathbb{T}_{q} \subset \mathcal{T}^{\Delta}$, $q \in \mathcal{Q}$, from the definition of tightness. Fix $p \in \mathcal{P}_{Q}$ and $T \in \mathcal{T}_{p} \subset \mathcal{T}^{\Delta}$ (as in Proposition \ref{prop6}(4)). Consider
\begin{displaymath} \mathcal{L}(p,T) := \{\ell \in \mathcal{A}(2,1) : \ell \cap p \neq \emptyset \text{ and } \ell \subset T\}. \end{displaymath}
Then, by Proposition \ref{prop6}(5),
\begin{align*} (\bar{\mu} \times \sigma_{x})(p \times \mathcal{L}(p,T)) & \stackrel{\mathrm{def.}}{=} (\bar{\mu} \times \sigma_{x})(\{(x,\theta) \in p \times [0,1] : \ell_{x,\theta} \in \mathcal{L}(p,T)\})\\
& = (\bar{\mu} \times \sigma_{x})(\{(x,\theta) \in p \times [0,1] : \ell_{x,\theta} \subset T\})\\
& = (\bar{\mu} \times \sigma_{x})(p \times T) \gtrapprox_{\Delta} \bar{\mu}(p)/|\mathcal{T}_{p}|. \end{align*} 
Moreover, by Lemma \ref{lemma3}, the image $S_{Q}(\mathcal{L}(p,T))$ can be covered by a family $\mathbb{T}(p,T) \subset \mathcal{T}^{\Delta}$ with $|\mathbb{T}(p,T)| \lesssim 1$. We may therefore select one distinguished element $\mathbf{T} := \mathbf{T}(p,T) \in \mathbb{T}(p,T)$ such that also
\begin{equation}\label{form34} (\bar{\mu} \times \sigma_{x})(p \times S_{Q}^{-1}(\mathbf{T})) \gtrapprox_{\Delta} \bar{\mu}(p)/|\mathcal{T}_{p}|. \end{equation}
For $q = S_{Q}(p) \in \mathcal{Q}$, we now define $\mathbb{T}_{q} := \{\mathbf{T}(p,T) : T \in \mathcal{T}_{p}\} \subset \mathcal{T}^{\Delta}$. We note that 
\begin{equation}\label{form35} |\mathbb{T}_{S_{Q}(p)}| \sim |\mathcal{T}_{p}|, \qquad p \in \mathcal{P}_{Q}. \end{equation}
This is because the slopes of the elements in $\mathcal{T}_{p}$ are $\Delta$-separated: $\mathcal{T}_{p}$ is a family of dyadic $\Delta$-tubes intersecting a fixed $\delta_{j}$-square with $\delta_{j} \leq \Delta$. Moreover, $T \in \mathcal{T}_{p}$ (almost) uniquely determines the slope of $\mathbf{T}(p,T)$ according to the last part of Lemma \ref{lemma3}.

With this definition, if $q = S_{Q}(p) \in \mathcal{Q}$, and $\mathbf{T} \in \mathbb{T}_{q}$, then,
\begin{align*} (\bar{\mu}^{Q} \times \sigma_{y}^{Q})(q \times \mathbf{T}) & \stackrel{\mathrm{def.}}{=} \int_{q} \sigma_{y}^{Q}(\{\theta : \ell_{y,\theta} \subset \mathbf{T}\}) \, d\bar{\mu}^{Q}(y)\\
& \stackrel{\mathrm{def.}}{=} \tfrac{1}{\bar{\mu}(Q)} \int_{p} \sigma^{Q}_{S_{Q}(x)}(\{\theta : \ell_{S_{Q}(x),\theta} \subset \mathbf{T}\}) \, d\bar{\mu}(x) \\
& \stackrel{\mathrm{def.}}{=} \tfrac{1}{\bar{\mu}(Q)} \int_{p} \sigma_{x}(\{\theta : \ell_{S_{Q}(x),\theta} \subset \mathbf{T}\}) \, d\bar{\mu}(x) \\
& = \tfrac{1}{\bar{\mu}(Q)} \int_{p} \sigma_{x}(\{\theta : \ell_{x,\theta} \subset S_{Q}^{-1}(\mathbf{T})\}) \, d\bar{\mu}(x) \\
& \stackrel{\mathrm{def.}}{=} \tfrac{1}{\bar{\mu}(Q)} (\bar{\mu} \times \sigma_{x})(p \times S_{Q}^{-1}(\mathbf{T}))\\
& \stackrel{\eqref{form34}}{\gtrapprox_{\Delta}} \frac{\bar{\mu}(p)/|\mathcal{T}_{p}|}{\bar{\mu}(Q)} = \bar{\mu}^{Q}(q)/|\mathcal{T}_{p}| \stackrel{\eqref{form35}}{\sim} \bar{\mu}^{Q}(q)/|\mathbb{T}_{q}|. \end{align*} 
This is almost the tightness condition \nref{T3}, except that we have not defined the number "$M$" (i.e. the "common cardinality") yet. We do this now. Recall from Proposition \ref{prop6}(4) that all the families $\mathcal{T}_{p}$ are contained in a common family $\mathbb{T}_{Q} \subset \mathcal{T}^{\Delta}$ of cardinality 
\begin{equation}\label{form36} |\mathbb{T}_{Q}| \lessapprox_{\Delta} \Delta^{-\epsilon}|\mathcal{T}_{p}| \lesssim \Delta^{-\epsilon}|\mathbb{T}_{S_{Q}(p)}|. \end{equation}
We define $M := c\Delta^{\epsilon}|\mathbb{T}_{Q}|$ for a suitable small constant $c \approx_{\Delta} 1$ to be chosen in a moment. Using \eqref{form35}, we infer that $|\mathbb{T}_{q}| \lesssim |\mathbb{T}_{Q}| = c^{-1}\Delta^{-\epsilon}M$ for all $q \in \mathcal{Q}$, so the long computation above shows that
\begin{displaymath} (\bar{\mu}^{Q} \times \sigma_{y}^{Q})(q \times \mathbf{T}) \gtrapprox_{\Delta} \Delta^{\epsilon}\bar{\mu}^{Q}(q)/M, \qquad q \in \mathcal{Q}, \, \mathbf{T} \in \mathbb{T}_{q}. \end{displaymath}
This yields the $\Delta^{-2\epsilon}$-tightness condition \nref{T3}, provided that $\Delta > 0$ is sufficiently small in terms of $\epsilon$. Moreover, since $M \leq |\mathbb{T}_{q}|$ for all $q \in \mathcal{Q}$ according to \eqref{form36} (and taking $c \approx_{\Delta} 1$ small enough), we may simply reduce the families $\mathbb{T}_{q}$ if necessary so that they all have common cardinality exactly $M$. This gives the tightness condition \nref{T1}.

We next check the tightness condition \nref{T4}. Since $\mathcal{T}_{p} \subset \mathbb{T}_{Q}$, in particular
\begin{displaymath} \sigma(\mathcal{T}_{p}) \subset \sigma(\mathbb{T}_{Q}) \subset (\Delta \cdot \Z), \qquad p \in \mathcal{P}_{Q}. \end{displaymath}
Now, for each $p \in \mathcal{P}_{Q}$ and $T \in \mathcal{T}_{p}$, the slope of the tube $\mathbf{T}(p,T) \in \mathbb{T}_{S_{Q}(p)}$ is within $\lesssim \Delta$ of the slope of $T$, see the last statement of Lemma \ref{lemma3}. Therefore, the total slope set of the family $\mathbb{T} := \cup \{\mathbb{T}_{q} : q \in \mathcal{Q}\}$ satisfies $|\sigma(\mathbb{T})| \lesssim |\sigma(\mathbb{T}_{Q})| \leq |\mathbb{T}_{Q}| \sim \Delta^{-\epsilon}M$. This is what is required by the tightness condition \nref{T4}. 

Finally, it remains to check the $(\Delta,s,\Delta^{-\epsilon})$-set property of $\sigma(\mathbb{T})$. According to the final part of Proposition \ref{prop6}, the family $\mathbb{T}_{Q}$ can be chosen to be a $(\Delta,s,\mathbf{C}')$-set with $\mathbf{C}' \lessapprox_{\Delta} \mathbf{C}$; since all the elements of $\mathbb{T}_{Q}$ intersect the fixed square $Q \in \mathcal{D}_{\underline{\Delta}}$ with $\underline{\Delta} \leq \Delta$, this is equivalent (Lemma \ref{lemma4}) to $\sigma(\mathbb{T}_{Q})$ being a $(\Delta,s,\mathbf{C}')$-set, and $|\sigma(\mathbb{T}_{Q})| \sim |\mathbb{T}_{Q}|$.

Moreover, we just argued above that $\sigma(\mathbb{T})$ lies in the $\lesssim \Delta$-neighbourhood of $\sigma(\mathbb{T}_{Q})$, and $|\sigma(\mathbb{T})| \gtrsim |\sigma(\mathcal{T}_{p})| \geq \Delta^{\epsilon}|\sigma(\mathbb{T}_{Q})|$ for any $p \in \mathcal{P}_{Q}$ according to \eqref{form36}. It follows that $\sigma(\mathbb{T})$ is a $(\Delta,s,\Delta^{-\epsilon}\mathbf{C}')$-set, and therefore a $(\Delta,s,\Delta^{-2\epsilon})$-set for $\Delta > 0$ sufficiently small. \end{proof}


\section{Proof of Theorem \ref{t:configurations}}\label{s4}

We finally have all the ingredients to prove the measure-theoretic version of Theorem \ref{t:main} stated in Theorem \ref{t:configurations}. Recall from Section \ref{s3} that Theorem \ref{t:configurations} implies Theorem \ref{t:main}. We start by restating Theorem \ref{t:configurations}:

\begin{thm}\label{t:configurationsRestated} For all $t \in (1,2]$, $s \in (2 - t,1]$, and $C > 0$, there exist a radius $r = r(C,s,t) > 0$ such that the following holds. Let $(\mu,\{\sigma_{x}\})$ be a configuration, where $\mu$ is a $(t,C)$-Frostman probability measure, and $\sigma_{x}$ is an $(s,C)$-Frostman probability measure for $\mu$ almost all $x \in \R^{2}$. Then,
\begin{displaymath} \inf \{\dist(y,L_{x}) : x,y \in \spt \mu, \, |x - y| \geq r\} = 0, \end{displaymath} 
where $L_{x} := \cup \{\ell_{x,\theta} : \theta \in \spt \sigma_{x}\}$, and $\ell_{x,\theta} = \pi_{\theta}^{-1}\{\pi_{\theta}(x)\}$. \end{thm}

\begin{proof} Fix the configuration $(\mu,\{\sigma_{x}\})$. Write
\begin{displaymath} \tau := \tau(s,t) := \tfrac{1}{2}[(2 - s) + t] \in (2 - s,t). \end{displaymath}
Note that still $s + \tau > 2$.

Let $\eta = \eta(s,\tau) > 0$ and $\Delta_{0} = \Delta_{0}(s,\tau) > 0$ be sufficiently small that the conclusions of Theorem \ref{t2} hold. Next, given this $\eta = \eta(s,\tau) > 0$, apply the (main) Lemma \ref{l:main} to find a positive constant $\epsilon = \epsilon(\eta,s,\tau) \in (0,\eta]$. Next, apply Theorem \ref{t:renormalisation} with constants $C$ (provided in the hypothesis of Theorem \ref{t:configurationsRestated}) and $\epsilon/2,t,\tau$. This yields another threshold $\Delta_{0}' = \Delta_{0}'(C,\epsilon/2,t,\tau) > 0$, and an integer $n = n(\epsilon/2,t,\tau)$, which eventually just depend on $C,s,t$. Let $\Delta_{1} := \min\{\Delta_{0},\Delta_{0}'\}$, and set
\begin{displaymath} r := \Delta_{1}^{2n}. \end{displaymath}
We claim that the conclusion of Theorem \ref{t:configurationsRestated} holds with this choice of "$r$".

To see this, apply Theorem \ref{t:renormalisation} to find scales $\underline{\Delta},\Delta \in [\Delta_{1}^{n},\Delta_{1}]$, a sub-configuration $(\bar{\mu},\{\sigma_{x}\})$, and a square $Q \in \mathcal{D}_{\underline{\Delta}}$ such that the $Q$-renormalised configuration $(\bar{\mu}^{Q},\{\sigma_{x}^{Q}\})$ is $\Delta^{-\epsilon/2}$-tight at scale $\Delta$, with data
\begin{displaymath} (\mathcal{Q},\mathbb{T}) \subset \mathcal{D}_{\Delta} \times \mathcal{T}^{\Delta}. \end{displaymath}
Moreover, by Theorem \ref{t:renormalisation}, $\bar{\mu}^{Q}$ is $(\tau,\Delta^{-\epsilon/2})$-Frostman, and $\sigma(\mathbb{T})$ is a non-empty $(\Delta,s,\Delta^{-\epsilon})$-set. The $(\tau,\Delta^{-\epsilon/2})$-Frostman property of $\bar{\mu}^{Q}$, the rough constancy of $q \mapsto \bar{\mu}^{Q}(q)$ on $\mathcal{Q}$, and $\bar{\mu}^{Q}(\cup \mathcal{Q}) \geq \Delta^{\epsilon/2}$ (by $\Delta^{-\epsilon/2}$-tightness) together imply that $\mathcal{Q}$ is a $(\Delta,\tau,\Delta^{-\epsilon})$-set (repeating the argument at \eqref{form71}, for example).

This information brings us to a position to apply the (main) Lemma \ref{l:main} to the configuration $(\bar{\mu}^{Q},\{\sigma_{x}^{Q}\})$, with parameters $\eta,s,\tau$ (as specified above). The conclusion is that there exist $\Delta$-separated subsets $\mathcal{G}_{1},\mathcal{G}_{2} \subset \mathcal{Q}$ such that the following holds for $G_{j} := \cup \mathcal{G}_{j}$:
\begin{itemize}
\item $\min\{\bar{\mu}^{Q}(G_{1}),\bar{\mu}^{Q}(G_{2})\} \geq \Delta^{\eta}$, where $G_{j} := (\cup \mathcal{G}_{j}) \cap \spt \bar{\mu}^{Q}$.
\item $X[\bar{\mu}^{Q}|_{G_{1}},\{\sigma^{Q}_{x}\}]_{\Delta}(y) \geq \Delta^{\eta}$ for all $y \in G_{2}$.
\end{itemize}
Since $\Delta \leq \Delta_{1} \leq \Delta_{0}$, and recalling that $\eta = \eta(s,\tau) > 0$ was the threshold required by Theorem \ref{t2}, that result now implies
\begin{displaymath} \inf \{\dist(y,L_{x}^{Q}) : x \in G_{1} \text{ and } y \in G_{2}\} = 0. \end{displaymath}
Here $L_{x}^{Q} = \cup \{\ell_{x,\theta} : \theta \in \spt \sigma_{x}^{Q}\}$. Since $\dist(G_{1},G_{2}) \geq \Delta$, and $\spt \bar{\mu} \subset \spt \mu$, in particular
\begin{displaymath} \inf \{\dist(y,L_{x}^{Q}) : x,y \in \spt \mu^{Q} \text{ and } |x - y| \geq \Delta\} = 0. \end{displaymath}
Recalling that $\Delta \geq \Delta_{1}^{n}$, and the definitions of the rescaled measures $\mu^{Q}$ and $\sigma^{Q}_{x}$ from Definition \ref{def:renormalisation}, and finally that $\ell(Q) = \underline{\Delta} \geq \Delta_{1}^{n}$, this implies 
\begin{displaymath} \inf \{\dist(y,L_{x}) : x,y \in \spt \mu, \, |x - y| \geq \Delta_{1}^{2n}\} = 0. \end{displaymath}
This completes the proof of Theorem \ref{t:configurationsRestated}. \end{proof}


\appendix

\section{A lemma on Lipschitz functions}\label{appA}

This section contains the proof of Lemma \ref{lemma5}, restated here:

\begin{lemma}\label{lemma5Restated} Let $t \in (0,d]$, $\tau \in (0,t)$, $\gamma \in (0,1]$, $\epsilon \in (0,\tfrac{1}{8d}\gamma(t - \tau)^{2}]$, and $\Delta \in 2^{-\N}$. Then, there exist $n_{0} = n_{0}(d,\epsilon) \in \N$ and $\mathbf{A} \lesssim_{d} \Delta^{-3d}$, such that the following holds for all $n \geq n_{0}$.

Let $\delta := \Delta^{n}$, and let $\mathcal{P} \subset \mathcal{D}_{\delta}([0,1)^d)$ be a $\{\Delta^{j}\}_{j = 0}^{n - 1}$-uniform $(\delta,t,\delta^{-\epsilon})$-set. Then, there exists
\begin{displaymath} \mathcal{G} \subset \{0,\ldots,\gamma n\} \quad \text{ with } \quad |\mathcal{G}| \geq n \cdot \gamma(t - \tau)^{2}/(10 d^{2}) \end{displaymath}
such that $\mathcal{P}^{Q} = S_{Q}(\mathcal{P} \cap Q) \subset \mathcal{D}_{\delta/\Delta^{j}}$ is a $(\delta/\Delta^{j},\tau,\mathbf{A})$-set for all $j \in \mathcal{G}$, and $Q \in \mathcal{D}_{\Delta^{j}}(\mathcal{P})$. \end{lemma}

Lemma \ref{lemma5Restated} will be proven by studying the behaviour of the \emph{branching function} associated to every uniform set:
\begin{definition}[Branching function]\label{def:branchingFunction} Let $\Delta \in 2^{-\N}$, and let $\mathcal{P} \subset \mathcal{D}_{\delta}$ be a $\{\Delta^{j}\}_{j = 0}^{n - 1}$-uniform set, $\delta = \Delta^{n}$. Let 
\begin{displaymath} \{N_{j}\}_{j = 0}^{n - 1} \subset \{1,\ldots,\Delta^{-d}\}^{n - 1} \end{displaymath}
be the associated sequence, as in Definition \ref{def:uniformity}. The \emph{branching function} $\beta \colon [0,n] \to [0,dn]$ is defined by setting $\beta(0) = 0$, and
\begin{displaymath} \beta(j) := \frac{\log |\mathcal{P}|_{\Delta^{j}}}{-\log_{2}(\Delta)} = \frac{1}{-\log_{2}(\Delta)} \sum_{i = 0}^{j - 1} \log N_{i}, \qquad j \in \{1,\ldots,n\}, \end{displaymath}
and then interpolating linearly. \end{definition}

Note that $\beta$ defines a (piecewise linear) non-decreasing $d$-Lipschitz function on $[0,n]$. The following simple lemma, combining \cite[Lemmas 2.22 and 2.24]{2023arXiv230110199O}, shows that the $(\delta,t)$-set properties of uniform sets, and their renormalisations, can be characterised by the "superlinear" behaviour of their branching functions on intervals of the form $[a,n]$.

\begin{lemma}\label{OSLemma} 
  Let $\Delta\in 2^{-\N}$, and let $\mathcal{P} \subset \mathcal{D}_{\delta}$ be $\{\Delta^{j}\}_{j = 0}^{n - 1}$-uniform, $\delta = \Delta^{n}$. Let $\beta \colon [0,n] \to [0,dn]$ be the associated branching function. Let $t \in [0,d]$, $C \geq 1$, and $\epsilon > 0$.
  \begin{enumerate}
  \item If $\mathcal{P}$ is a $(\delta,t,\delta^{-\epsilon})$-set, then 
  \begin{displaymath} \beta(x) \geq tx - \epsilon n - O_{d}(1), \qquad x \in [0,n]. \end{displaymath}
\item Fix $a\in \{0,\ldots,n - 1\}$ and $Q \in \mathcal{D}_{\Delta^{a}}(\mathcal{P})$. If
\begin{displaymath} \beta(x) - \beta(a) \geq t(x - a) - C, \qquad x \in [a,n], \end{displaymath}
then $\mathcal{P}^{Q}$ is a $(\delta/\Delta^{a},t,O_{d}(\Delta^{-(C + d)}))$-set. 
\end{enumerate} \end{lemma}

\begin{remark} The constant $O_{d}(\Delta^{-(C + d)})$ is more precise than stated in \cite[Lemma 2.22]{2023arXiv230110199O}, where the constant is $O_{\Delta,d}(1)$. This part of lemma is actually proven in \cite[Lemma 8.3(1)]{OS23}, and one can easily track from that argument that the constant is $O_{d}(\Delta^{-(C + d)})$. In case the reader does this, let us still mention that the estimate in the proof of \cite[Lemma 8.3(1)]{OS23} contains a typo, and there "$\delta^{-\epsilon m}$" should be "$\Delta^{-\epsilon m}$".  \end{remark} 

Lemma \ref{OSLemma} shows that in order to prove Lemma \ref{lemma5Restated}, it suffices to study the behaviour of non-decreasing $d$-Lipschitz functions $f \colon [0,n] \to [0,dn]$ satisfying $f(0) = 0$, or equivalently $g \colon [0,1] \to [0,d]$ (via the rescaling $g(x) := \tfrac{1}{n}f(nx)$).

\begin{figure}[h!]
\begin{center}
\begin{overpic}[scale = 0.9]{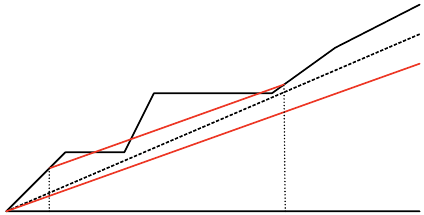}
\put(9,-5){$x$}
\put(63,-5){$h(x)$}
\put(48,32){$f$}
\end{overpic}
\caption{The functions $f$ and $h$ in Lemma \ref{lemma2}. The dotted line is the graph of $x \mapsto tx$, whereas the (longer) red line is the graph of $x \mapsto \tau x$.}\label{fig1}
\end{center}
\end{figure}

\begin{lemma}\label{lemma2} Let $t \in (0,d]$, $\tau \in (0,t)$, $\gamma \in (0,1)$, and $\epsilon \in (0,\tfrac{1}{4d}\gamma(t - \tau)^{2}]$. Let $f \colon [0,1] \to [0,\infty)$ be a $d$-Lipschitz function satisfying $f(0) = 0$ and $f(x) \geq t x - \epsilon$ for all $x \in [0,1]$. Then, there exists an analytic set $G \subset [0,\gamma]$ of measure $\mathcal{H}^{1}(G) \geq \gamma(t - \tau)^{2}/(10d^{2})$ such that
\begin{displaymath} f(y) - f(x) \geq \tau(y - x), \qquad x \in G, \, y \in [x,1]. \end{displaymath}
\end{lemma}

\begin{proof} Let $\mathfrak{c} := \mathfrak{c}(d,\gamma,t,\tau) := \gamma (t - \tau)/(2d) > 0$,  
\begin{displaymath} G := \{x \in [0,1] : f(y) - f(x) \geq \tau(y - x) \text{ for all } y \in [x,1]\}, \end{displaymath}
and $\pi(a,b) := -\tau a + b$. For $x \in [0,\mathfrak{c}]$, define
\begin{displaymath} h(x) := \sup \{y \in [0,1] : \pi(y,f(y)) = \pi(x,f(x))\}, \end{displaymath} 
see Figure \ref{fig1} for an illustration. We make three remarks. First, note that we are not taking a "$\sup$" over an empty set, since $y = x$ itself satisfies $\pi(y,f(y)) = \pi(x,f(x))$. Second, the "$\sup$" is really a "$\max$", so in particular $\pi(h(x),f(h(x))) = \pi(x,f(x))$. This follows readily from the continuity of $f$ and $\pi$. Third, note that $\pi(y,f(y)) = \pi(x,f(x))$ is equivalent to $f(y) - f(x) = \tau(y - x)$. 

We first claim that $h(x) \in [0,\gamma]$ for $x \in [0,\mathfrak{c}]$. In fact, we prove something a little stronger: if $x \in [0,\mathfrak{c}]$, and $y \in [x,1]$ is any point such that $f(y) - f(x) \leq \tau(y - x)$, then $y \leq \gamma$. In fact, the opposite inequality $y > \gamma$ would lead to
\begin{align*} f(y) \le f(x) + \tau(y - x) & \leq f(x) + ty - (t - \tau)y\\
& < \mathfrak{c}d + t y - \gamma (t - \tau)\\
& = ty - \tfrac{1}{2}\gamma(t - \tau) \le t y - \epsilon,  \end{align*}
contradicting our main hypothesis.

We next claim that $h(x) \in G$ for all $x \in [0,\mathfrak{c}]$. To see this, assume to the contrary that $x_{1} := h(x) \notin G$. This means that there exists $y \in (x_{1},1]$ such that
\begin{displaymath} f(y) - f(x_{1}) < \tau(y - x_{1}). \end{displaymath}
Consequently also $f(y) - f(x) < \tau(y - x)$. We have shown above that this implies $y \leq \gamma$, so the opposite inequality $f(y') - f(x) \geq \tau(y' - x)$ has to hold for $y' > \gamma$. Therefore, there exists a point $y'' \in [y,y']$ satisfying 
\begin{displaymath} f(y'') - f(x) = \tau(y'' - x), \end{displaymath}
or equivalently $\pi(x,f(x)) = \pi(y'',f(y''))$. This means that $h(x) \geq y''$, which is a contradiction, since $y'' > x_{1} = h(x)$.

We have now shown that $h \colon [0,\mathfrak{c}] \to G \cap [0,\gamma]$. Next, we note that 
\begin{align*} \mathcal{H}^{1}(\pi(\{(x,f(x)) : x \in [0,\mathfrak{c}]\})) & \geq |\pi(\mathfrak{c},f(\mathfrak{c}))| = |f(\mathfrak{c}) - \tau\mathfrak{c}|\\
& \geq \mathfrak{c}t - \epsilon - \tau\mathfrak{c} \geq \tfrac{\mathfrak{c}(t - \tau)}{2} = \tfrac{\gamma (t - \tau)^{2}}{4d}.  \end{align*}
In the penultimate inequality we used the hypothesis $\epsilon \leq \tfrac{1}{4d}\gamma(t - \tau)^{2} = \tfrac{1}{2}\mathfrak{c}(t - \tau)$. Finally, note that since $\pi(h(x),f(h(x))) = \pi(x,f(x))$ for all $x \in [0,\mathfrak{c}]$, we also have
\begin{displaymath} \pi(\{(h(x),f(h(x))) : x \in [0,\mathfrak{c}]\}) = \pi(\{(x,f(x)) : x \in [0,\mathfrak{c}]\}), \end{displaymath}
and therefore
\begin{displaymath} \mathcal{H}^{1}(\pi(\{(y,f(y)) : y \in G \cap [0,\gamma]\})) \geq \mathcal{H}^{1}(\pi(\{(h(x),f(h(x))) : x \in [0,\mathfrak{c}]\})) \geq \tfrac{\gamma (t - \tau)^{2}}{4d}. \end{displaymath} 
Finally, the composition $y \mapsto \pi(y,f(y))$ is $2d$-Lipschitz, so $\mathcal{H}^{1}(G \cap [0,\gamma]) \geq \tfrac{\gamma(t - \tau)^{2}}{10d^{2}}$. \end{proof}

The proof of Lemma \ref{lemma5Restated} is a straightforward combination of Lemmas \ref{OSLemma} and \ref{lemma2}:

\begin{proof}[Proof of Lemma \ref{lemma5Restated}] Since $\mathcal{P}$ is a $\{\Delta^{j}\}_{j = 0}^{n - 1}$-uniform $(\delta,t,\delta^{-\epsilon})$-set, the branching function $\beta \colon [0,n] \to [0,dn]$ satisfies 
\begin{displaymath} \beta(x) \geq tx - \epsilon n - O_{d}(1), \qquad x \in [0,n], \end{displaymath}
according to Lemma \ref{OSLemma}. In particular, $\beta(x) \geq tx - 2\epsilon n$ for $n \geq n_{0}(d,\epsilon)$, and the rescaled function $\bar{\beta}(\cdot) := \tfrac{1}{n}\beta(n \, \cdot) \colon [0,1] \to [0,d]$ satisfies $\bar{\beta}(x) \geq tx - 2\epsilon$. Since $2\epsilon \leq \tfrac{1}{4d}\gamma(t - \tau)^{2}$ by assumption, Lemma \ref{lemma2} may be applied to $\bar{\beta}$. The conclusion is that there exists a set $\bar{G} \subset [0,\gamma]$ with $\mathcal{H}^{1}(\bar{G}) \geq \gamma(t - \tau)^{2}/(10d^{2})$ such that 
\begin{displaymath} \bar{\beta}(y) - \bar{\beta}(x) \geq \tau(y - x), \qquad x \in \bar{G}, \, y \in [x,1]. \end{displaymath}
Writing $G := \{nx : x \in G\} \subset [0,\gamma n]$, the above yields $\mathcal{H}^{1}(G) \geq n \cdot \gamma(t - \tau)^{2}/(10d^{2})$, and
\begin{equation}\label{form58} \beta(y) - \beta(x) \geq \tau(y - x), \qquad x \in G, \, y \in [x,n]. \end{equation}
Let $\mathcal{G} := \{\lfloor x\rfloor : x \in G\} \subset \{0,\ldots,\gamma n\}$ be the integer parts of elements in $G$. Then $|\mathcal{G}| \geq \gamma(t - \tau)^{2}/(10d^{2})$. Using \eqref{form58}, and the $d$-Lipschitz property of $\beta$,
\begin{displaymath} \beta(y) - \beta(j) \geq \tau(y - j) - 2d, \qquad j \in \mathcal{G}, \, y \in [x,n]. \end{displaymath}
Now the second part of Lemma \ref{OSLemma} implies that for all $j \in \mathcal{G}$, the renormalisation $\mathcal{P}^{Q} \subset \mathcal{D}_{\delta/\Delta^{j}}$ is a $(\delta/\Delta^{j},\tau,O_{d}(\Delta^{-3d}))$-set for all $Q \in \mathcal{D}_{\Delta^{j}}(\mathcal{P})$. This completes the proof. \end{proof}


\section{Counter examples}\label{appB}

The purpose of this section is to provide counter examples for \cite[Theorem 6.9]{MR3617376}, as discussed in Remark \ref{rem2}. Recall that $\ell_{x,e} = x + \mathrm{span}(e)$ for $x \in \R^{2}$ and $e \in S^{1}$. For $t \in (1,2]$, let $\gamma(t) \in [0,1]$ be the infimum over $\gamma \in [0,1]$ such that the following holds:
\begin{itemize}
\item Let $K \subset \R^{2}$ be compact with $\mathcal{H}^{t}(K) < \infty$. Then, there exists a set $E \subset S^{1}$ with 
\begin{displaymath} \Hd E \leq \gamma \end{displaymath}
such that for $\mathcal{H}^{t}$ almost all $x \in K$, it holds $|K \cap \ell_{x,e}| \geq 2$ for all $e \in S^{1} \, \setminus \, E$.
\end{itemize}
Proposition \ref{appProp1} shows that $\gamma(t)$ satisfies no non-trivial bounds for $t \in (1,2)$:

\begin{proposition}\label{appProp1} $\gamma(t) = 1$ for $t \in (1,2)$. \end{proposition}

To prove the proposition, we need two classical facts about graphs of H\"older functions. Recall that a function $f \colon [a,b] \to \R$ is \emph{$\alpha$-H\"older continuous} (for $\alpha \in (0,1]$) if there exists a constant $C > 0$ such that $|f(x) - f(y)| \leq C|x - y|^{\alpha}$ for all $x,y \in [a,b]$. For a function $f \colon [0,1] \to \R$, and $B \subset [a,b]$, we write
\begin{displaymath} \Gamma_{f}(B) := \{(x,f(x)) : x \in B\} \subset \R^{2}. \end{displaymath} 

\begin{lemma} Let $\alpha \in (0,1]$, and let $f \colon [a,b] \to \R$ be $\alpha$-H\"older continuous. Then,
\begin{equation}\label{kahane} \Hd \Gamma_{f}(A) \leq \Hd A + 1 - \alpha, \qquad A \subset [a,b]. \end{equation} \end{lemma} 

\begin{proof} The proof can be found in \cite[Section 7, Theorem 6]{MR254888}. \end{proof}

Conversely, Besicovitch and Ursell \cite{BeUr} have constructed for every $t \in [1,2)$ a $(2 - t)$-H\"older function $f \colon [0,1] \to [\tfrac{1}{2},1]$ such that $\Hd \Gamma_{f}([0,1]) = t$.

For the proof of Proposition \ref{appProp1}, we need "radial" versions of \eqref{kahane} and the Besicovitch-Ursell construction. Fix $t \in [1,2)$, and let $f = f_{t} \colon [0,\tfrac{\pi}{2}] \to [\tfrac{1}{2},1]$ be the $(2 - t)$-H\"older function constructed by Besicovitch and Ursell, scaled to the interval $[0,\tfrac{\pi}{2}]$. 

For $\theta \in [0,\tfrac{\pi}{2}]$, let $e_{\theta} := (\cos \theta,\sin \theta)$, and let $S := \{e_{\theta} : \theta \in [0,\tfrac{\pi}{2}]\} \subset S^{1}$. Consider the function $g \colon S \to \R$ defined by $g(e_{\theta}) := f(\theta)$, and the "radial" graph 
\begin{displaymath} \Gamma_{g}^{\mathrm{rad}} := \Gamma^{\mathrm{rad}}_{g}(S) := \{g(e)e : e \in S\}. \end{displaymath}
Then $\Gamma_{g}^{\mathrm{rad}}$ is the image of $\Gamma_{f}([0,\tfrac{\pi}{2}])$ under the map $(r,\theta) \mapsto re_{\theta}$ (which is bi-Lipschitz on $[0,\tfrac{\pi}{2}] \times [\tfrac{1}{2},1]$), so 
\begin{displaymath} \Hd \Gamma_{g}^{\mathrm{rad}} = \Hd \Gamma_{f}([0,\tfrac{\pi}{2}]) = t. \end{displaymath}
Similarly, it follows from \eqref{kahane} with $\alpha = 2 - t$ that
\begin{equation}\label{kahane2} \Hd \Gamma_{g}^{\mathrm{rad}}(A) \leq \Hd A + t - 1, \qquad A \subset S. \end{equation}
Lastly, we observe that if $x = g(e)e \in \Gamma_{g}^{\mathrm{rad}}$, then
\begin{equation}\label{form77} |\Gamma_{g}^{\mathrm{rad}} \cap \ell_{x,e}| = 1. \end{equation}

We are then prepared to prove Proposition \ref{appProp1}.

\begin{proof} Fix $t \in (1,2)$ and $\gamma < 1$. It suffices to prove that $\gamma(t) \geq \gamma$. Let $\Gamma_{g}^{\mathrm{rad}} \subset \R^{2} \, \setminus \, \{0\}$ be the radial version of the Besicovitch-Ursell graph with parameter $\bar{t} \in (t,2)$ such that
\begin{displaymath} 1 + (t - \bar{t}) > \gamma. \end{displaymath}
Let $K \subset \Gamma_{g}^{\mathrm{rad}}$ be an arbitrary compact subset with $0 < \mathcal{H}^{t}(K) < \infty$.

We make a counter assumption: $\gamma(t) < \gamma$. In particular, with $K$ as above, there exists a set $E \subset S^{1}$ with $\Hd E \leq \gamma$, and a $\mathcal{H}^{t}$ full measure subset $B \subset K \subset \Gamma_{g}^{\mathrm{rad}}$ such that 
\begin{equation}\label{form78} |K \cap \ell_{x,e}| \geq 2, \qquad x \in B, \, e \in S^{1} \, \setminus \, E. \end{equation} 
Let $A := \{e \in S : \mathrm{span}(e) \cap B \neq \emptyset\}$, so $B = \Gamma_{g}^{\mathrm{rad}}(A)$. Then \eqref{kahane2} implies 
\begin{displaymath} t = \Hd B = \Hd \Gamma_{g}^{\mathrm{rad}}(A) \leq \Hd A + \bar{t} - 1, \end{displaymath}
thus $\Hd A \geq 1 + (t - \bar{t}) > \gamma$. In particular, there exists $e \in A \, \setminus \, E$. Let $x = g(e)e \in B$. Since $K \subset \Gamma^{\mathrm{rad}}_{g}$,
\begin{displaymath} 1 \stackrel{\eqref{form77}}{=} |\Gamma_{g}^{\mathrm{rad}} \cap \ell_{x,e}| \geq |K \cap \ell_{x,e}| \stackrel{\eqref{form78}}{\geq} 2, \end{displaymath}
which gives the desired contradiction.  \end{proof}

\bibliographystyle{plain}
\bibliography{references}

\end{document}